\documentclass{amsart}
\usepackage{amsmath,amsthm}
\usepackage{amsfonts,amssymb}
\usepackage[shortlabels]{enumitem}
\usepackage{mathrsfs}
\usepackage{float}
\usepackage{subcaption}
\usepackage{multirow}
\usepackage{mathrsfs}
\usepackage{tikz}
\copyrightinfo{2019}{H. De Clercq}

\calclayout

\newsavebox{\imagebox}

\newtheorem{theorem}{Theorem}[section]
\newtheorem{proposition}[theorem]{Proposition}
\newtheorem{lemma}[theorem]{Lemma}
\newtheorem{corollary}[theorem]{Corollary}

\newtheorem{theoremKolb}{Theorem}

\newtheorem{propositionKolb}[theoremKolb]{Proposition}
\newtheorem{lemmaKolb}[theoremKolb]{Lemma}

\theoremstyle{definition}
\newtheorem{definition}[theorem]{Definition}

\theoremstyle{remark}
\newtheorem{remark}{Remark}

\newcommand{\N}{\mathbb{N}}
\newcommand{\Z}{\mathbb{Z}}
\newcommand{\K}{\mathbb{K}}
\newcommand{\g}{\mathfrak{g}}
\newcommand{\h}{\mathfrak{h}}
\newcommand{\Uqg}{U_q(\mathfrak{g})}
\newcommand{\Uqgp}{U_q(\mathfrak{g}')}
\renewcommand{\k}{\mathfrak{k}}
\newcommand{\Bcs}{B_{\mathbf{c},\mathbf{s}}}

\newcommand{\bc}{\mathbf{c}}
\newcommand{\bs}{\mathbf{s}}
\newcommand{\bss}{\boldsymbol{s}}
\newcommand{\ad}{\mathrm{ad}}
\newcommand{\bl}{\boldsymbol{\ell}}
\newcommand{\plsk}{\mathfrak{p}_{\boldsymbol{\ell},\boldsymbol{s},k}^{(i,j,a_{ij})}}
\newcommand{\Aut}{\mathrm{Aut}}
\newcommand{\rlskd}{\mathfrak{r}_{\boldsymbol{\ell},\boldsymbol{s},k,d}^{(i,j,a_{ij})}}
\newcommand{\Ei}{\widetilde{E_i}}
\newcommand{\Ej}{\widetilde{E_j}}
\newcommand{\bN}{\boldsymbol{N}}

\newcommand{\equationinalign}[1]{%
	\multispan{2}%
	\hfill$\displaystyle{#1}$\hfill
	\ignorespaces
}

\title[Defining relations for quantum symmetric pair coideals of Kac-Moody type]{Defining relations for quantum symmetric pair coideals of Kac-Moody type}
\author{Hadewijch De Clercq}
\email{Hadewijch.DeClercq@UGent.be}
\address{Department of Electronics and Information Systems\\Faculty of Engineering and Architecture\\Ghent University\\Building S8, Krijgslaan 281, 9000 Gent\\ Belgium.}

\date{\today}
\begin{document}
	
\begin{abstract}
	Classical symmetric pairs consist of a symmetrizable Kac-Moody algebra $\mathfrak{g}$, together with its subalgebra of fixed points under an involutive automorphism of the second kind. Quantum group analogs of this construction, known as quantum symmetric pairs, replace the fixed point Lie subalgebras by one-sided coideal subalgebras of the quantized enveloping algebra $U_q(\mathfrak{g})$. We provide a complete presentation by generators and relations for these quantum symmetric pair coideal subalgebras. These relations are of inhomogeneous $q$-Serre type and are valid without restrictions on the generalized Cartan matrix. We draw special attention to the split case, where the quantum symmetric pair coideal subalgebras are generalized \(q\)-Onsager algebras.
\end{abstract}

\maketitle

\section{Introduction} 

A classical symmetric pair consists of a Lie algebra \(\g\) together with its subalgebra \(\k\) of fixed points under a Lie algebra involution \(\theta\). Quantum analogs of this construction, known as quantum symmetric pairs, have emerged in the beginning of the 1990's. They replace \(\g\) by its quantized universal enveloping algebra \(U_q(\g)\) and \(\k\) by a one-sided coideal subalgebra \(B_{\bc,\bs}\) of \(U_q(\g)\), which is called a quantum symmetric pair (QSP) coideal subalgebra. The algebras \(\Bcs\) were first constructed by Noumi, Sugitani and Dijkhuizen \cite{Noumi-1996, Noumi&Sugitani-1995, Noumi&Dijkhuizen&Sugitani-1997} under the name quantum Grassmannians. A different approach, based on the Drinfeld-Jimbo presentation of \(U_q(\g)\), was pursued by Letzter. She developed a comprehensive theory of quantum symmetric pairs for semisimple Lie algebras \(\g\) in an elaborate series of papers \cite{Letzter-1999,Letzter-2002,Letzter-2003}. This has allowed to identify the zonal spherical functions on quantum symmetric spaces as Macdonald-Koornwinder polynomials \cite{Letzter-2004}. This whole theory was later extended to symmetrizable Kac-Moody algebras \(\g\) by Kolb in \cite{Kolb-2014}, which treats the structure theory of the Kac-Moody QSP coideal subalgebras \(\Bcs\) in great detail.

Over the years, it has become increasingly apparent that quantum symmetric pairs play a crucial role in quantum integrability, notably of the reflection equation \cite{Cherednik-1984, Sklyanin-1988}. The latter replaces the quantum Yang-Baxter equation when reflecting boundary conditions are imposed by \(K\)-matrices. Such boundaries break the quantum symmetry down to a coideal subalgebra of the quantum affine algebra which encodes the symmetries in the bulk of a quantum spin chain \cite{Delius&MacKay-2003}. It has been suggested that representations of the QSP coideal subalgebras \(\Bcs\) give rise to universal solutions of the reflection equation, just like solutions of the quantum Yang-Baxter equation arise naturally from representations of quantum affine algebras \cite{Delius&MacKay-2006}. This has been worked out for several specific QSPs \cite{Kolb&Stokman-2009,Kolb-2015,Reshetikhin&Stokman&Vlaar-2016}. Invaluable tools in this respect are the definition of a bar involution for quantum symmetric pairs by Balagovi\'{c} and Kolb in \cite{Balagovic&Kolb-2015}, and the construction of quasi \(K\)-matrices as QSP analogs of Lusztig's quasi \(R\)-matrices, both in recursive \cite{Balagovic&Kolb-2019} and in factorized form \cite{Dobson&Kolb-2019}. 

In view of the same idea, many more fundamental concepts in representation theory have been extended from quantum affine algebras to quantum symmetric pairs. A special class of QSP coideal subalgebras arises in the categorification of \(U_q(\g)\)-representations based on skew Howe duality, as was established by Ehrig and Stroppel in \cite{Ehrig&Stroppel-2018}. A thorough theory of canonical bases and Schur-Jimbo duality for quantum symmetric pairs was set up by Bao and Wang in \cite{Bao&Wang-2018,Bao&Wang-2018-2}, leading them to prove longstanding conjectures on irreducible characters of orthosymplectic Lie superalgebras. 

A first step towards a classification of finite-dimensional \(\Bcs\)-modules was recently set in \cite{Letzter-2019}, which establishes a quantum Cartan decomposition of QSP coideal subalgebras for semisimple Lie algebras \(\g\). Moreover, it was shown in \cite{Kolb-2020} that the finite-dimensional representations of \(\Bcs\) form a braided module category over the braided monoidal category of finite-dimensional \(U_q(\g)\)-representations. It is well-known that this braided monoidal category encodes the structure of the universal \(R\)-matrix of \(U_q(\g)\). A similar categorical framework encoding the universal \(K\)-matrix of a quantum symmetric pair was given in \cite{Weelinck-2019} under the name \(\mathbb{Z}_2\)-braided pair. In \cite{Jordan&Ma-2011} the quantum analog of the classical symmetric pair \((\mathfrak{gl}_N,\mathfrak{gl}_p\times \mathfrak{gl}_{N-p}) \) is used to relate representations of (double) affine braid groups and (double) affine Hecke algebras of type \(C^{\vee}C_n\). Finally, generalizations of quantum symmetric pairs have been considered in \cite{Regelskis&Vlaar-2018}, which replace the \(\g\)-involution \(\theta\) by a more general semisimple automorphism of \(\g\). 

In this paper we will adopt the notational conventions of \cite{Kolb-2014}. We will write \(\g = \g(A)\) for the Kac-Moody algebra associated to a symmetrizable generalized Cartan matrix \(A\) of dimension \(n\). We take \(I\) to be the set \(\{0,\dots,n-1\}\), such that we can write \(A = (a_{ij})_{i,j\in I}\). We will use Kolb's definition of admissible pairs, as will be repeated later in Definition \ref{admissible pair def}, to parametrize the involutive automorphisms of \(\g\) of the second kind. To each such admissible pair one can associate a quantum symmetric pair and hence a coideal subalgebra \(\Bcs\) of \(\Uqg\), which also depends on a multiparameter \((\bc,\bs)\).

The QSP coideal subalgebras \(\Bcs\) can be presented in terms of generators and relations. The set of generators depends on the choice of admissible pair, but always contains certain elements \(B_i\), with \(i\in I\). A set of defining relations which describe these algebras abstractly in terms of their generators was given by Kolb in \cite[Theorem 7.1]{Kolb-2014} and will be repeated in the upcoming Theorem \ref{theorem these are defining relations}. One of these relations states that
\begin{equation}
\label{first relation}
\sum_{m = 0}^{1-a_{ij}}(-1)^m\begin{bmatrix}
1-a_{ij}\\m
\end{bmatrix}_{q_i}B_i^{1-a_{ij}-m}B_jB_i^{m}
\end{equation}
can be written as a lower-degree polynomial in \(B_i\) and \(B_j\) which depends on the entry \(a_{ij}\) of \(A\). However, Kolb's theorem does not provide a precise form for this polynomial, which he denotes by \(C_{ij}(\bc)\). Up to present, expressions for \(C_{ij}(\bc)\) were only known for a few possible values of \(a_{ij}\), namely \(a_{ij}\in\{0,-1,-2,-3\}\). These have been obtained in \cite{Kolb-2014} and \cite{Balagovic&Kolb-2015} by explicit calculations, which follow similar results in \cite{Letzter-2003} for finite-dimensional \(\g\). It was suggested by Kolb that the same rationale could lead to expressions for \(C_{ij}(\bc)\) valid for all \(a_{ij}\), but this has not been explicitized before. This paper provides for the first time closed expressions for the polynomials \(C_{ij}(\bc)\), valid without restrictions on the Cartan matrix or the admissible pair. It thereby completes the presentation of the quantum symmetric pair coideal subalgebras by generators and relations.

Such a presentation is highly desirable in view of the representation theory of the algebras \(\Bcs\). This was already indicated in \cite{Balagovic&Kolb-2015}, where the definition of a new bar involution for quantum symmetric pairs was validated by showing that it respects the defining relations of \(\Bcs\). By the absence of such relations beyond the case \(\vert a_{ij}\vert \leq 3\), this could only be done for a limited class of Cartan matrices. Our results allow to extend this to any quantum symmetric pair.

Our approach will be as follows. We will rewrite (\ref{first relation}) as a complicated expression in \(U_q(\g)^{\otimes 2}\), where one of the tensor components is acted upon with a projection operator. This leads to the upcoming expressions (\ref{expr lemma Case 1 elimination}) and (\ref{expr lemma Case 2 elimination}), which were essentially already contained in \cite{Kolb-2014}. The main novelty of our approach lies in how we expand these expressions further. We will distinguish two cases, based on the behavior of \(i\) and \(j\) with respect to the admissible pair, each leading to a different expression for \(C_{ij}(\bc)\). In Propositions \ref{prop term 1 expansion with epsilon and P} and \ref{prop term 2 expansion with epsilon and P} we will perform a binary distributive expansion to rewrite (\ref{first relation}) as a polynomial which, in the first of these two cases, is of the form
\begin{equation}
\label{second relation}
\sum_{m,m'}\rho_{m,m'}\mathcal{Z}_i^{\frac{1-a_{ij}-m-m'}{2}}B_i^mB_jB_i^{m'},
\end{equation}
whereas in the second case one finds
\begin{equation}
\label{third relation}
\sum_{m,m',t}\rho_{m,m',t}\mathcal{Z}_i^tB_i^mB_jB_i^{m'}\mathcal{Z}_i^{\frac{1-a_{ij}-m-m'}{2}-t}+ \sum_{m,t}\sigma_{m,t}\mathcal{Z}_i^t\mathcal{W}_{ij}K_j\mathcal{Z}_i^{\frac{-1-a_{ij}-m}{2}-t}B_i^{m},
\end{equation}
where the elements \(\mathcal{Z}_i\), \(\mathcal{W}_{ij}\) and \(K_j\) are well-defined in terms of the generators and where all sums are finite. The major difficulty lies in the determination of the coefficients \(\rho_{m,m'}\), \(\rho_{m,m',t}\) and \(\sigma_{m,t}\), which we will refer to as the structure constants of the algebra \(\Bcs\). Initially, we will describe these in terms of monomials in \(\Uqg\) acted upon with a projection operator. Closed expressions for the actions of these projection operators and hence for the structure constants are consequently derived in Theorems \ref{theorem F_ij(B_i,B_j) for Case 1} and \ref{theorem F_ij(B_i,B_j) Case 2}. It may not surprise that the formulae we obtain there turn out to be rather computationally extensive. Indeed, even the expressions obtained in \cite{Kolb-2014} and \cite{Balagovic&Kolb-2015} for small values of \(a_{ij}\), as displayed in the upcoming Tables \ref{Table of rho_m,m'}, \ref{Table of rho_m,m',t} and \ref{Table of sigma_m,t}, were already quite intricate. Nevertheless, our formulae contain nothing but finite sums and products, which can easily be carried out either by hand or by a computer.

Many interesting examples of quantum symmetric pairs and their coideal subalgebras have appeared in the literature. The paper \cite{Aldenhoven&Koelink&Roman-2017} studies a QSP coideal subalgebra associated to the quantum group \(U_q(\mathfrak{su}(3))\). Its finite-dimensional irreducible representations are classified and its highest weight vectors are expressed in terms of dual \(q\)-Krawtchouk polynomials. For finite-dimensional \(\g\) one can easily extend the \(\g\)-involution \(\theta\) to the loop algebra \(L(\g)\) or the untwisted affine Lie algebra \(\widehat{\g}\). Quantum analogs of the fixed point Lie subalgebra of \(L(\g)\) under this extended automorphism are known as twisted quantum loop algebras (of the second kind). These comprise the twisted \(q\)-Yangians of \cite{Molev&Ragoucy&Sorba-2003} and the QSPs of type AIII of \cite{Chen&Guay&Ma-2014}. A unified approach to the structure theory of these twisted quantum loop algebras is contained in \cite[Section 11]{Kolb-2014}. Other remarkable examples of the subalgebras \(\Bcs\) for non-affine \(\g\) are the quantized GIM Lie algebras of \cite{Lv&Tan-2013}, treated by Kolb in \cite[Section 12]{Kolb-2014}. 

In this paper, we will draw special attention to the QSP coideal subalgebras in the split case, corresponding to the trivial admissible pair \((\emptyset,\mathrm{id})\). These are known as generalized \(q\)-Onsager algebras. Their name has been derived from the algebra defined by Onsager in \cite{Onsager-1944} as a tool towards his analytic solution of the planar Ising model in zero magnetic field. This algebra was presented in \cite{Dolan&Grady-1982} and \cite{Perk-1989} as the infinite-dimensional Lie algebra with generators \(B_0\) and \(B_1\) subject to the Dolan-Grady relations
\[
[B_0,[B_0,[B_0,B_1]]] = -4[B_0,B_1], \qquad [B_1,[B_1,[B_1,B_0]]] = -4[B_1,B_0].
\]
It has received much attention in special function theory and integrable lattice models \cite{Davies-1990,Klishevich&Plyushchay-2003,Hartwig&Terwilliger-2007,Baxter-2009,Vernier&OBrien&Fendley-2019}. It can be embedded in the affine Lie algebra \(\widehat{\mathfrak{sl}_2}\) as its subalgebra of fixed points under the Chevalley involution \cite{Roan-1991}, and hence together with \(\widehat{\mathfrak{sl}_2}\) it forms a (split) classical symmetric pair. The theory of quantum symmetric pairs thus offers a solid framework to deform the Onsager algebra to a quantum algebra. The resulting \(q\)-Onsager algebra \cite{Baseilhac-2005,Baseilhac&Koizumi-2005} is abstractly defined by the \(q\)-Dolan-Grady relations
\begin{equation}
\label{q-Dolan-Grady basic}
[B_0,[B_0,[B_0,B_1]_q]_{q^{-1}}] = \rho [B_0,B_1], \quad [B_1,[B_1,[B_1,B_0]_q]_{q^{-1}}] = \rho [B_1,B_0],
\end{equation}
where \([A,B]_q = qAB-q^{-1}BA\) is the \(q\)-commutator and \(\rho\) is a scalar depending on \(q\). The \(q\)-Onsager algebra has become an important object of study in quantum integrability \cite{Baseilhac-2005,Belliard&Fomin-2012,Kuniba&Pasquier-2019} and in connection with \(q\)-orthogonal polynomials \cite{Baseilhac&Vinet&Zhedanov-2017} and Leonard pairs \cite{Ito&Terwilliger-2009}. Upon adding a defining relation in its equitable presentation, the \(q\)-Onsager algebra is refined to the Askey-Wilson algebra \cite{Zhedanov-1991}, as was shown in \cite{Terwilliger-2000}. A central extension of the latter, known as the universal Askey-Wilson algebra \cite{Terwilliger-2011}, was also identified as a quotient of the \(q\)-Onsager algebra \cite{Terwilliger-2018}. This Askey-Wilson algebra provides an algebraic framework for the \(q\)-Askey scheme of orthogonal polynomials \cite{Baseilhac&Martin&Vinet&Zhedanov-2019,Koornwinder&Mazzocco-2018}, see also \cite{DeBie&DeClercq&vandeVijver-2018,DeBie&DeClercq-2019,Groenevelt-2018} for some recent multivariate generalizations.

The left-hand side of (\ref{q-Dolan-Grady basic}) can be rewritten as
\[
B_i^3B_j - [3]_{q}B_i^2B_jB_i + [3]_{q}B_iB_jB_i^2 - B_jB_i^3
\]
for \(i\neq j\in\{0,1\}\). This coincides with the expression (\ref{first relation}) for \(n = 2\) and \(a_{01} = a_{10} = -2\), i.e.\ for \(\g = \widehat{\mathfrak{sl}_2}\). It is hence apparent that the \(q\)-Onsager algebra coincides with the quantum symmetric pair coideal subalgebra \(\Bcs\) of \(U_q(\widehat{\mathfrak{sl}_2})\) for the trivial admissible pair and a special choice of the parameters \(\bc,\bs\). 

Kac-Moody generalizations of the \(q\)-Onsager algebra were constructed by Baseilhac and Belliard in \cite{Baseilhac&Belliard-2010}. A presentation with generators and relations was given for affine Lie algebras \(\g\), again for a limited set of Cartan matrices. The relations we will derive in Theorem \ref{theorem F_ij(B_i,B_j) for Case 1} extend this to symmetrizable Kac-Moody algebras without restrictions on the Cartan matrix. Moreover, we will use a recent result by Chen, Lu and Wang \cite{Chen&Lu&Wang-2019} to obtain alternative, transparent expressions of quantum Serre type for the defining relations of these generalized \(q\)-Onsager algebras. These relations, which will be given in Theorem \ref{theorem Case 1 with CLW}, turn out to hold even for any quantum symmetric pair provided the indices \(i\) and \(j\) satisfy the conditions of the aforementioned Case 1. In addition, these allow to prove symmetry properties of the structure constants \(\rho_{m,m'}\) from (\ref{second relation}).

For \(q = 1\), such inhomogeneous Serre relations for generalized Onsager algebras had already been obtained by Stokman in \cite{Stokman-2019}. His classical generalized Onsager algebras extend those of \cite{Roan-1991, Uglov&Ivanov-1996,Date&Usami-2004,Neher&Savage&Senesi-2012} to arbitrary root systems. The defining relations he provides, involve a set of coefficients which are defined in a recursive fashion. Our approach now allows to derive closed expressions for these coefficients and thus solve the recursion relations, by taking the limit \(q\to 1\) of the analogous expressions in the quantum case. This will be performed in Theorem \ref{theorem Stokman via limits}.

The paper is organized as follows. In Section \ref{Section Construction of the generators} we recall the necessary prerequisites on quantum symmetric Kac-Moody pairs in the notation of \cite{Kolb-2014}. We treat the classical symmetric pairs \((\g,\k)\) in Subsection \ref{Subsection classical case} and their quantum analogs \((U_q(\g),\Bcs)\) in Subsection \ref{Subsection quantum case}. In Subsection \ref{Subsection Kolb's projection technique} we state some of the results obtained by Kolb in \cite{Kolb-2014}, which we will need in what follows. The main body of work is contained in Section \ref{Section quantum Serre relations}, where the missing defining relations for \(\Bcs\) will be derived. In Subsection \ref{Subsection Binary expansions} we will perform a binary distributive expansion to reduce the computation of the polynomials \(C_{ij}(\bc)\) to an easier problem, namely determining the coefficients in (\ref{second relation}) and (\ref{third relation}) through the action of the counit and a certain projection operator on monomials in \(U_q(\g)\). This problem will be solved in Subsections \ref{Subsection Case 1} and \ref{Subsection Case 2} treating Cases 1 and 2 respectively. The principal results are presented in Theorems \ref{theorem F_ij(B_i,B_j) for Case 1} and \ref{theorem F_ij(B_i,B_j) Case 2}. To conclude, we will derive alternative and more accessible expressions for the polynomials \(C_{ij}(\bc)\) in Case 1 based on the work \cite{Chen&Lu&Wang-2019} in Theorem \ref{theorem Case 1 with CLW}. Finally, we turn our attention to the generalized \(q\)-Onsager algebras and their classical counterparts. We repeat the obtained relations applied to the split case and reconsider them in the limit \(q\to 1\) to solve the recursion relations of \cite{Stokman-2019} in Theorem \ref{theorem Stokman via limits}.

\section{Construction of the generators}
\label{Section Construction of the generators}

Let us start by recalling some crucial concepts and notations introduced in \cite{Kolb-2014}. 

Let \(\K\) be an algebraically closed field of characteristic 0. Let \(A\) be an indecomposable generalized Cartan matrix of dimension \(n\) and let us denote by \(I\) the set \(\{0,1,\dots,n-1\}\). This means that \(A = (a_{ij})_{i,j\in I}\) satisfies the properties
\begin{enumerate}
	\item[i.] \(a_{ii} = 2\), for all \(i\in I\),
	\item[ii.] \(a_{ij}\in \Z^{-} \), if \(i\neq j\in I\),
	\item[iii.] \(a_{ij} = 0 \Leftrightarrow a_{ji} = 0\), for any \(i,j\in I\)
	\item[iv.] For every non-empty proper subset \(I'\subset I\) there exist \(i\in I', j\in I\setminus I'\) such that \(a_{ij}\neq 0\).
\end{enumerate} 
Moreover, we assume \(A\) to be symmetrizable, i.e.\ there exists a diagonal matrix \(D = \mathrm{diag}(\epsilon_i: i\in I)\), with mutually coprime and nonzero entries \(\epsilon_i\in\N\), such that \(DA\) is symmetric. 

In Subsection \ref{Subsection classical case}, we will construct the classical symmetric pair \((\g,b_{\bs})\), where \(\g=\g(A)\) is the Kac-Moody algebra associated to \(A\). This construction will motivate the definition of the quantum symmetric pair \((U_q(\g),B_{\bc,\bs})\) inside the corresponding quantum group \(U_q(\g)\), which will be given in Subsection \ref{Subsection quantum case}.

\subsection{The classical case}
\label{Subsection classical case}

Let \((\h=\h(A), \Pi = \{\alpha_i: i\in I\}, \Pi^{\vee} = \{h_i: i\in I\})\) be a minimal realization of \(A\). This means that \(\h\) is a \(\K\)-vector space of dimension \(2n-\mathrm{rank}(A)\) and that \(\Pi^{\vee}\) and \(\Pi\) are linearly independent subsets of \(\h\) and its dual \(\h^{\ast}\) respectively, subject to \(\alpha_j(h_i) = a_{ij}\) for any \(i,j\in I\). Let \(Q = \mathbb{Z}\Pi\) be the corresponding root lattice. 

The Kac-Moody algebra \(\g = \g(A)\) associated to \(A\) is the Lie algebra over \(\K\) generated by \(\h\) and \(2n\) Chevalley generators \(e_i, f_i\) with \(i\in I\), with defining relations
\begin{gather}
\label{Kac-Moody relations 1}
[h,h'] = 0, \qquad
[h,e_i] = \alpha_i(h)e_i, \qquad [h,f_i] = -\alpha_i(h)f_i, \qquad
[e_i,f_j] = \delta_{ij}h_i, \\
\label{Kac-Moody relations 2}
\left(\mathrm{ad}\,e_i\right)^{1-a_{ij}}e_j = \left(\mathrm{ad}\,f_i\right)^{1-a_{ij}}f_j = 0,
\end{gather}
for all \(i,j\in I\) and \(h,h'\in\h\). Here we denoted by \(\mathrm{ad}\) the adjoint mapping
\begin{equation}
\label{adjoint map def}
\mathrm{ad}: \g \to \mathfrak{gl}_{\g}: x \mapsto \mathrm{ad}\,x, \qquad \mathrm{ad}\,x: \g \to \g: y \mapsto [x,y].
\end{equation}
The derived Lie subalgebra \(\g' = [\g,\g]\) of \(\g\) is generated by \(\h' = \sum_{i\in I}\K h_i\) and the elements \(e_i, f_i\) with \(i\in I\). 

As usual, we will write \(\g_{\beta} = \{x\in \g: [h,x] = \beta(h)x, \forall h\in \h \}\) for any \(\beta\in \h^{\ast}\) and \(\Phi = \{\beta\in\h^{\ast}: \g_{\beta}\neq \{0\}\}\) for the corresponding root system. 

For any \(i\in I\) we denote by \(r_i\in \mathrm{GL}(\h)\) the fundamental reflection which acts on \(h\in\h\) by
\[
r_i(h) = h - \alpha_i(h)h_i.
\]
The subgroup \(W\) of \(\mathrm{GL}(\h)\) generated by all such \(r_i\) stands as the Weyl group of \(\g\). Via duality, \(W\) can also act on \(\h^{\ast}\) and hence in particular on \(Q\), via
\begin{equation}
\label{W acting on h dual}
r_i(\alpha) = \alpha - \alpha(h_i)\alpha_i,
\end{equation}
for any \(\alpha\in\h^{\ast}\).

Consider a subset \(X\subseteq I\). Let \(\g_X\) be the corresponding Lie subalgebra of \(\g\), generated by the elements \(e_i\), \(f_i\) and \(h_i\) with \(i\in X\). Write \(\Phi_X\subseteq \Phi\) for its root system and \(\rho_X^{\vee}\) for half the sum of the positive coroots of \(\Phi_X\). We will write \(W_X\) for the parabolic subgroup of the Weyl group \(W\) associated to \(X\), and \(w_X\) for its longest element. 
Finally, let us denote by \(\mathrm{Aut}(A,X)\) the set of permutations \(\sigma\) of \(I\) subject to
\[
\sigma(X) = X \quad\mathrm{and}\quad a_{\sigma(i),\sigma(j)}=a_{ij}, \quad\forall i,j\in I.
\]
Any \(\sigma\in \Aut(A,X)\) extends to an automorphism of \(\g\) by taking
\begin{equation}
\label{extension of Aut(A,X)}
\sigma(e_i) = e_{\sigma(i)}, \quad \sigma(f_i) = f_{\sigma(i)}, \quad \sigma(h_i) = h_{\sigma(i)}
\end{equation}
and defining the action of \(\sigma\) on \(h\in\h\setminus\h'\) as described in \cite[Section 4.19]{Kac&Wang-1992}. Similarly, \(\sigma\in\Aut(A,X)\) extends to an automorphism of \(Q\) upon setting 
\begin{equation}
\label{extension of Aut(A,X) 2}
\sigma(\alpha_i) = \alpha_{\sigma(i)}.
\end{equation}

This terminology allows to repeat the definition of an admissible pair, as given in \cite[Definition 2.3]{Kolb-2014}.

\begin{definition}
	\label{admissible pair def}
	An admissible pair \((X,\tau)\) consists of a subset \(X\subseteq I\) and an automorphism \(\tau\in \mathrm{Aut}(A,X)\) subject to the following conditions:
	\begin{enumerate}
		\item\label{def admissible pair requirement 1} \(\tau\) is an involution, i.e.\ \(\tau^2 = \mathrm{id}\).
		\item\label{def admissible pair requirement 2} The action of \(\tau\) on \(X\) coincides with the corresponding action of \(-w_X\), i.e.\ for any \(j\in X\) one has \(h_{\tau(j)} = -w_X(h_j)\) and \(\alpha_{\tau(j)} = -w_X(\alpha_j)\), where we have used the interpretations of \(\tau\) and \(w_X\) according to (\ref{W acting on h dual})--(\ref{extension of Aut(A,X)})--(\ref{extension of Aut(A,X) 2}).
		\item\label{def admissible pair requirement 3} For any \(i\in I\setminus X\) satisfying \(\tau(i) = i\), one has \(\alpha_i(\rho_X^{\vee})\in\Z\). 
	\end{enumerate}
\end{definition}

An important motivation for introducing admissible pairs is that they arise naturally as Kac-Moody generalizations of Satake diagrams \cite{Araki-1962}. Moreover they parametrize the so-called involutive automorphisms of \(\g\) of the second kind \cite{Levstein-1988,Kac&Wang-1992} up to conjugation by elements of \(\mathrm{Aut}(\g)\), as was shown in \cite[Theorem 2.7]{Kolb-2014}. The automorphism \(\theta(X,\tau)\) corresponding to an admissible pair \((X,\tau)\) can be constructed using the following four key concepts.

The first is the element \(\tau\in \mathrm{Aut}(A,X)\), interpreted as an automorphism of \(\g\) according to (\ref{extension of Aut(A,X)}).

Furthermore, we will need the Chevalley involution \(\omega\in\mathrm{Aut}(\g)\) given by 
\begin{equation}
\label{Chevalley involution def}
\omega(e_i) = -f_i, \quad \omega(f_i) = -e_i, \quad \omega(h) = -h,
\end{equation}
for any \(i\in I\) and \(h\in\h\).

Moreover, the longest element \(w_X\) of \(W_X\) can be lifted to an element \(m_X\) of the Kac-Moody group of \(\g'\), with corresponding automorphism \(\mathrm{Ad}(m_X)\in\mathrm{Aut}(\g)\). For details we refer to \cite[Section 1.3]{Kac&Wang-1992} and \cite[Section 2]{Kolb-2014}.

Finally, one can define a group morphism \(s(X,\tau): Q\to \K^{\times}\) from the root lattice \(Q\) to the multiplicative group \(\K^{\times}\), by
\begin{equation}
\label{S(X,tau) def}
s(X,\tau)(\alpha_j) = \left\{
\arraycolsep=1.4pt\def\arraystretch{1.3}
\begin{array}{ll}
1 & \mathrm{if}\ j\in X\ \mathrm{or}\ \tau(j) = j, \\
i^{\alpha_j(2\rho_X^{\vee})}\quad & \mathrm{if}\ j\in I\setminus X\ \mathrm{and}\ \tau(j) > j, \\
(-i)^{\alpha_j(2\rho_X^{\vee})}\quad & \mathrm{if}\ j\in I\setminus X\ \mathrm{and}\ \tau(j) < j,
\end{array}
\right.
\end{equation}
where \(i\in\K\) is a square root of \(-1\). The corresponding automorphism \(\mathrm{Ad}(s(X,\tau))\in\mathrm{Aut}(\g)\) is defined by
\begin{equation}
\label{Ad(s(X,tau)) def}
\mathrm{Ad}(s(X,\tau))(h) = h, \quad \mathrm{Ad}(s(X,\tau))(v) = s(X,\tau)(\alpha)v,
\end{equation}
for all \(h\in\h\) and \(v\in\g_{\alpha}\), \(\alpha\in \Phi\).

These four ingredients can now be combined to yield the following involutive automorphism \(\theta(X,\tau)\).

\begin{definition}
	To each admissible pair \((X,\tau)\) we associate the automorphism \(\theta(X,\tau)\) of \(\g\) given by
	\begin{equation}
	\label{theta(X,tau) def}
	\theta(X,\tau) = \mathrm{Ad}(s(X,\tau))\circ \mathrm{Ad}(m_X)\circ \tau\circ \omega.
	\end{equation}
	It is an involutive \(\g\)-automorphism of the second kind by \cite[Theorem 2.5]{Kolb-2014}.
\end{definition}

Let us from now on fix an admissible pair \((X,\tau)\) and write \(\theta\) for the above defined automorphism \(\theta(X,\tau)\). Then \(\theta\) gives rise to an algebra which will be of special interest in this paper.

\begin{definition}
	\label{b_s definition}
	For any vector \(\bs=(s_i)_{i\in I\setminus X}\in\K^{I\setminus X}\) we define \(b_{\bs} = b_{\bs}(X,\tau)\) to be the subalgebra of \(U(\g')\) generated by the elements
	\begin{align}
	\label{b_s def}
	\begin{split}
	f_i + \theta(f_i)+s_i&\ \mathrm{with}\ i\in I\setminus X,\\
	e_i, f_i, h_i &\ \mathrm{with}\ i\in X, \\
	h_i&\ \mathrm{with}\ \theta(h_i) = h_i, i\in I.
	\end{split}
	\end{align} 
	The couple \((\g,b_{\bs})\) stands as the (classical) symmetric pair associated to the admissible pair \((X,\tau)\).
\end{definition}

In \cite{Kolb-2014}, the algebra \(b_{\bs}\) was denoted by \(U(\k')\). We have chosen to adopt this alternative notation, to emphasize that \(b_{\bs}\) is a classical counterpart of the quantum algebra \(B_{\bc,\bs}\), which we will define in Subsection \ref{Subsection quantum case}. The defining relations of \(b_{\bs}\) will then follow as a limit \(q\to 1\) of the quantum Serre relations for \(B_{\bc,\bs}\) we will derive in Section \ref{Section quantum Serre relations}.

\subsection{The quantum case}
\label{Subsection quantum case}

Let \(q\) be an indeterminate, assumed not to be a root of unity in the field \(\K\). We denote by \(\K(q)\) the field of rational functions in \(q\). 

Recall the matrix \(D = \mathrm{diag}(\epsilon_i: i\in I)\) we have introduced above. For each \(i\in I\) we set \(q_i = q^{\epsilon_i}\). For any \(m\in\N\), we define the \(q_i\)-number \([m]_{q_i}\) and the \(q_i\)-factorial \([m]_{q_i}!\) as
\[
[m]_{q_i} = \frac{q_i^m-q_i^{-m}}{q_i-q_i^{-1}}, \qquad [m]_{q_i}! = \prod_{\ell = 1}^m[\ell]_{q_i},
\]
with the convention that \([0]_{q_i}! = 1\). For \(N, m\in \N\) with \(N\geq m\), we define the \(q_i\)-binomial coefficient as
\[
\begin{bmatrix}
N \\ m
\end{bmatrix}_{q_i} = \frac{[N]_{q_i}!}{[m]_{q_i}![N-m]_{q_i}!}.
\]
Like in the classical case, one has
\begin{equation}
\label{prop q-binom symbol}
\begin{bmatrix}
N \\ m
\end{bmatrix}_{q_i} = \begin{bmatrix}
N \\ N-m
\end{bmatrix}_{q_i}.
\end{equation}
We will often use the following polynomial in two non-commutative variables \(x\) and \(y\), which we will refer to as the quantum Serre polynomial: for \(i,j\in I\) we write
\begin{equation}
\label{Fij def}
F_{ij}(x,y) = \sum_{\ell = 0}^{1-a_{ij}}(-1)^{\ell}\begin{bmatrix} 1-a_{ij} \\ \ell
\end{bmatrix}_{q_i} x^{1-a_{ij}-\ell}yx^{\ell}.
\end{equation}
Throughout the paper, we will perform calculations in the quantized universal enveloping algebra \(\Uqg\) of \(\g\). In fact, it will suffice to work with its Hopf subalgebra \(\Uqgp\), the associative \(\K(q)\)-algebra generated by \(4n\) elements \(E_i\), \(F_i\), \(K_i\) and \(K_i^{-1}\) with \(i\in I\), subject to the relations
\begin{align}
\label{U_q(g) relations}
\begin{split}
K_i^{\pm 1}K_i^{\mp 1} = 1, \qquad & [K_i^{\pm 1},K_j^{\pm 1}] = 0, \\ 
K_i E_j = q_i^{a_{ij}} E_jK_i, \qquad & K_iF_j = q_i^{-a_{ij}} F_jK_i, \\
K_i^{-1} E_j = q_i^{-a_{ij}} E_jK_i^{-1}, \qquad & K_i^{-1}F_j = q_i^{a_{ij}} F_jK_i^{-1},
\\
\end{split} \\
\equationinalign{\label{U_q(g) relations 2}
	\ [E_i,F_j] = \delta_{ij}\frac{K_i-K_i^{-1}}{q_i-q_i^{-1}},
} \\
\equationinalign{
	\label{U_q(g) relations 3}
	F_{ij}(E_i,E_j) = F_{ij}(F_i,F_j) = 0,
}
\end{align}
for all \(i,j\in I\). The relations (\ref{U_q(g) relations 3}) are referred to as the quantum Serre relations. 

\begin{remark}
	\label{Remark specialization}
	The quantum group \(\Uqgp\) can be considered a \(q\)-deformation of \(\g'\), upon viewing \(e_i\) and \(f_i\) as the limits of \(E_i\) and \(F_i\) respectively as \(q\) goes to 1, and identifying \(K_i\) with \(q_i^{h_i}\). To view the quantum Serre relations (\ref{U_q(g) relations 3}) as \(q\)-deformations of the relations (\ref{Kac-Moody relations 2}), it will be useful to introduce the \(q\)-commutators
	\[
	\mathrm{ad}_{q_i^m}: \Uqg\to\Aut(\Uqg): x \mapsto \mathrm{ad}_{q_i^m}(x), \qquad \mathrm{ad}_{q_i^m}(x):\Uqg\to\Uqg: y \mapsto [x,y]_{q_i^m} = q_i^mxy-q_i^{-m}yx,
	\]
	with \(m\in\mathbb{Q}\). Notice that \(\mathrm{ad}_{q_i^m}\) reduces to \(\mathrm{ad}\) defined in (\ref{adjoint map def}) in the limit \(q\to 1\), for any \(m\in\mathbb{Q}\). It can easily be shown by induction that one has
	\begin{equation}
	\label{nested q-commutators induction}
	\left(\overrightarrow{\prod_{m = \frac{1-r}{2}}^{\frac{r-1}{2}}}\mathrm{ad}_{q_i^m}(A)\right)(B) = \sum_{k = 0}^r(-1)^k\begin{bmatrix} r \\ k
	\end{bmatrix}_{q_i}A^{r-k}BA^k
	\end{equation}
	for any \(A,B\in\Uqg\) and any \(r\in\N\), which, upon substituting \(r = 1-a_{ij}\), becomes
	\[
	\left(\overrightarrow{\prod_{m = \frac{a_{ij}}{2}}^{\frac{-a_{ij}}{2}}}\mathrm{ad}_{q_i^m}(A)\right)(B) = F_{ij}(A,B).
	\]
	Hence in the limit \(q\to 1\), the expression \(F_{ij}(A,B)\) reduces to
	\[
	(\ad\, a)^{1-a_{ij}}(b),
	\]
	where \(a\) and \(b\) are the specializations of \(A\) and \(B\) respectively, and so (\ref{U_q(g) relations 3}) indeed translates to (\ref{Kac-Moody relations 2}). A detailed account on this notion of specialization, which is a formal way to implement this limiting process \(q \to 1\), can be found in \cite{DeConcini&Kac-1990}, \cite[Sections 3.3 and 3.4]{Hong&Kang-2002} and \cite[Section 10]{Kolb-2014}.
\end{remark}

The quantum group \(\Uqgp\) has the structure of a Hopf algebra, with the following expressions for the coproduct \(\Delta\), the counit \(\epsilon\) and the antipode \(S\):
\begin{align}
\label{Coproduct, counit, antipode def}
\begin{alignedat}{6}
\Delta(E_i) & = E_i\otimes 1 + K_i\otimes E_i, \quad& \Delta(F_i) &= F_i\otimes K_i^{-1} + 1\otimes F_i, \quad& \Delta(K_i^{\pm 1}) &= K_i^{\pm 1}\otimes K_i^{\pm 1}, \\
\epsilon(E_i) &=  0, & \epsilon(F_i) &= 0, & \epsilon(K_i^{\pm 1}) &= 1, \\
S(E_i) &= -K_i^{-1}E_i, & S(F_i) &= -F_iK_i, & S(K_i^{\pm 1}) &= K_i^{\mp 1}.
\end{alignedat}
\end{align}

Now let us once more fix an admissible pair \((X,\tau)\). A quantum analog of the automorphism \(\theta(X,\tau)\) defined in (\ref{theta(X,tau) def}) can be built from five fundamental constituents, one of which is the mapping \(\tau\in\mathrm{Aut}(A,X)\) viewed as an automorphism of \(\Uqgp\) by 
\[
\tau(E_i) = E_{\tau(i)}, \quad \tau(F_i) = F_{\tau(i)}, \quad \tau(K_i^{\pm 1}) = K_{\tau(i)}^{\pm 1}.
\]

Secondly, one can extend \(\mathrm{Ad}(s(X,\tau))\in\Aut(\g)\) to an automorphism of \(\Uqg\) by
\[
\mathrm{Ad}(s(X,\tau))(v) = s(X,\tau)(\alpha)v,
\]
for all \(v\in U_q(\g)_{\alpha} = \{u\in\Uqg: K_iu = q^{(\alpha_i,\alpha)}uK_i, \forall i\in I\}\), \(\alpha\in Q\). Here by \((\cdot,\cdot)\) we denote the bilinear form on \(\h^{\ast}\) satisfying \((\alpha_i,\alpha_j) = \epsilon_ia_{ij}\).

Next, we will need a \(q\)-deformation of the Chevalley involution (\ref{Chevalley involution def}), which we will again denote by \(\omega\). It is given by
\[
\omega(E_i) = -F_i, \quad \omega(F_i) =-E_i, \quad \omega(K_i) = K_{i}^{-1}
\]
and classifies as a coalgebra antiautomorphism of \(\Uqgp\). 

To obtain a quantum analog of the element \(\mathrm{Ad}(m_X)\) in (\ref{theta(X,tau) def}) one needs the Lusztig automorphisms \(T_i\), \(i\in I\), which appeared in \cite[Section 37.1]{Lusztig-1994} under the name \(T_{i,1}''\). Let \(w_X = r_{i_1}r_{i_2}\dots r_{i_k}\) be a reduced expression for the longest element \(w_X\) of the parabolic subgroup \(W_X\) of \(W\), then denote by \(T_{w_X}\) the corresponding automorphism \(T_{w_X} = T_{i_1}T_{i_2}\dots T_{i_k}\) of \(\Uqgp\). 

Finally, define another automorphism \(\psi:\Uqgp\to\Uqgp\) by
\[
\psi(E_i) = E_iK_i, \quad\psi(F_i) = K_i^{-1}F_i, \quad \psi(K_i) = K_i.
\]
These are all the tools needed to \(q\)-deform \(\theta(X,\tau)\).

\begin{definition}
	To each admissible pair \((X,\tau)\) we associate the automorphism \(\theta_q(X,\tau)\) of \(\Uqgp\) given by
	\begin{equation}
		\theta_q(X,\tau) = \mathrm{Ad}(s(X,\tau))\circ T_{w_X}\circ \psi \circ \tau\circ \omega.
	\end{equation}
\end{definition}
Note that \(\theta_q = \theta_q(X,\tau)\) is no longer involutive.

Finally, let us denote by \(Q^{\Theta}\) the set \(\{\alpha\in Q: -w_X\tau(\alpha) = \alpha\}\). Here, we interpret both \(\tau\in\mathrm{Aut}(A,X)\) and \(w_X\in W_X\) as automorphisms of \(Q\) according to (\ref{W acting on h dual}) and (\ref{extension of Aut(A,X) 2}). Moreover, if \(\beta = \sum_{i\in I}m_i\alpha_i\in Q\), we will write \(K_{\beta}\) for \(\prod_{i\in I}K_i^{m_i}\). 
This brings us to the definition of the quantum analog \(B_{\bc,\bs}\) of the algebra \(b_{\bs}\) defined in (\ref{b_s def}).

\begin{definition}
	For any vector \(\bc=(c_i)_{i\in I\setminus X}\in(\K(q)^{\times})^{I\setminus X}\) and \(\bs = (s_i)_{i\in I\setminus X}\in\K(q)^{I\setminus X}\), we define \(B_{\bc,\bs} = B_{\bc,\bs}(X,\tau)\) to be the subalgebra of \(\Uqgp\) generated by the elements
	\begin{align}
	\label{B_c,s def}
	\begin{split}
	B_i = F_i+c_i\theta_q(F_iK_i)K_i^{-1}+s_iK_i^{-1} & \ \mathrm{with}\ i\in I\setminus X, \\
	E_i, F_i, K_i^{\pm 1} &\ \mathrm{with}\ i\in X, \\
	K_{\beta}&\ \mathrm{with}\ \beta\in Q^{\Theta}.
	\end{split}
	\end{align}
\end{definition}

When applying the coproduct \(\Delta\) described in (\ref{Coproduct, counit, antipode def}) on the generators (\ref{B_c,s def}), one can make the following observation.

\begin{propositionKolb}[{\cite[Proposition 5.2]{Kolb-2014}}]
	\label{prop coideal subalgebra}
	For any \((\bc,\bs)\in (\K(q)^{\times})^{I\setminus X}\times \K(q)^{I\setminus X}\), the algebra \(\Bcs\) is a right coideal subalgebra of \(\Uqgp\), i.e.\ \(\Delta(\Bcs)\subset \Bcs\otimes\Uqgp\).
\end{propositionKolb}

Upon comparing (\ref{B_c,s def}) with (\ref{b_s def}) in the light of Remark \ref{Remark specialization}, it is immediately clear that \(\Bcs\) is a \(q\)-deformation of the algebra \(b_{\bs}\) under certain conditions on the parameters \(c_i\) and \(s_i\), and that it reduces to the latter under the specialization \(q\to 1\). The precise conditions are described in the following theorem.

\begin{theoremKolb}[{\cite[Theorems 10.8, 10.11]{Kolb-2014}}]
	\label{theorem specialization}
		Let \(\bc=(c_i)_{i\in I\setminus X}\) be a vector of parameters taking values in
		\begin{equation}
		\label{def C}
		\mathcal{C} = \{\bc\in(\K(q)^{\times})^{I\setminus X}: c_i = c_{\tau(i)}\ \mathrm{if}\ \tau(i)\neq i \ \mathrm{and}\ (\alpha_i,-w_X\tau(\alpha_i)) = 0\},
		\end{equation}
		where \(\tau\) and \(w_X\) are again interpreted as automorphisms of \(Q\). Let \(\bs = (s_i)_{i\in I\setminus X}\) be a vector of parameters with values in
		\begin{equation}
		\label{def S}
		\mathcal{S} = \{\bs \in \K(q)^{I\setminus X}: s_i \neq 0 \Rightarrow (i\in I_{ns}\ \mathrm{and}\ a_{ji}\in -2\N\setminus\{0\}, \forall j\in I_{ns}\setminus\{i\}) \},
		\end{equation}
		where
		\[
		I_{ns} = \{i\in I\setminus X: \tau(i) = i\ \mathrm{and}\ a_{ij} = 0, \forall j\in X \}.
		\]
		Moreover, let us assume that the vector \((\bc,\bs)\) is specializable, i.e.\ \(\lim_{q\to 1}(c_i) = 1\) for any \(i\in I\) and all \(c_i, s_i\) lie in the localization \(\K[q]_{(q-1)}\) of the polynomial ring \(\K[q]\) with respect to the ideal generated by \(q-1\). Then \(\Bcs\) reduces to the algebra \(b_{\bs}\) under the formal specialization \(q\to 1\) and is maximal with this property.
\end{theoremKolb}

Although the assumption of specializability is required to obtain \(b_{\bs}\) as an exact limit of \(B_{\bc,\bs}\) for \(q\to 1\), it is still commonly accepted to view \(B_{\bc,\bs}\) as a quantum analog of \(b_{\bs}\) even if the latter condition is not fulfilled. Hence Proposition \ref{prop coideal subalgebra} suggests the following terminology.

\begin{definition}
	For \((\bc,\bs)\in\mathcal{C}\times \mathcal{S}\), the algebra \(B_{\bc,\bs}\) is called a quantum symmetric pair coideal subalgebra.
\end{definition}

Throughout the rest of this paper, we will fix a vector of parameters \((\bc,\bs)\in \mathcal{C}\times\mathcal{S}\) and work with the corresponding quantum symmetric pair coideal subalgebra \(\Bcs\).

\subsection{Kolb's projection technique}
\label{Subsection Kolb's projection technique}

In this section, we repeat some of the results obtained by Kolb in \cite{Kolb-2014}, which we will use in Section \ref{Section quantum Serre relations} to derive the defining relations of the quantum symmetric pair coideal subalgebras \(\Bcs\). For ease of notation, we will write \(\mathcal{M}_X^+\) and \({U_{\Theta}^0}'\) for the subalgebras of \(\Uqgp\) generated by the sets \(\{E_i: i\in X\}\) and \(\{K_{\beta}: \beta\in Q^{\Theta}\}\) respectively, and set \(B_j:=F_j\) for \(j\in X\). Let \(U^+\), \(U^-\) and \({U^0}'\) be the subalgebras of \(\Uqgp\) generated by \(\{E_i: i\in I\}\), \(\{F_i: i\in I\}\) and \(\{K_i^{\pm 1}: i\in I\}\) respectively. It was explained in \cite[Section 5]{Kolb-2014} that \(\Bcs\cap {U^0}' = {U_{\Theta}^0}'\) for \((\bc,\bs)\in \mathcal{C}\times\mathcal{S}\). Hence one can describe \(\Bcs\) as the subalgebra of \(\Uqgp\) generated by \(\{B_i:i\in I\}\cup\mathcal{M}_X^+ \cup{U_{\Theta}^0}'\). Furthermore, for any \(J\in I^m\), \(m\in \N \), we will write \(B_J\) for the product \(B_{j_1}B_{j_2}\dots B_{j_m} = \overrightarrow{\prod_{k=1}^m}B_{j_k}\). Let us also denote by \(\mathcal{J}_{i,j}\) the set of multi-indices given by
\[
\mathcal{J}_{i,j}=\big\{ (\underbrace{i,i,\dots,i}_{s\ \mathrm{times}}): s\leq 1-a_{ij}\big\} \cup \big\{(\underbrace{i,\dots,i}_{\ell\ \mathrm{times}},j,\underbrace{i,\dots,i}_{s-\ell\ \mathrm{times}}): s \leq-a_{ij},\ \ell\leq s \big\}.
\]
With this notation, one can write down the following theorem.

\begin{theoremKolb}[{\cite[Theorem 7.1]{Kolb-2014}}]
	\label{theorem these are defining relations}
	For any distinct \(i,j\in I\) there exist elements
	\[
	C_{ij}(\bc) \in \sum_{J\in \mathcal{J}_{i,j}}\mathcal{M}_X^+{U_{\Theta}^0}'B_J
	\]
	depending on the parameter vector \(\bc\), such that \(F_{ij}(B_i,B_j) = C_{ij}(\bc)\), or equivalently: \(F_{ij}(B_i,B_j)\) can be expressed as a polynomial in \(B_i\) and \(B_j\) of smaller total degree with coefficients in \(\mathcal{M}_X^+{U_{\Theta}^0}'\), possibly depending on \(\bc\) but not on \(\bs\). Moreover, the algebra \(B_{\bc,\bs}\) is abstractly defined by the relations
	\begin{align}
	\label{defining relations B_c,s 1}
	F_{ij}(B_i,B_j) = C_{ij}(\bc) & \ \mathrm{for} \ i\neq j\in I, \\
	\label{defining relations B_c,s 2}
	[E_i,B_j] = \delta_{ij}\frac{K_i-K_i^{-1}}{q_i-q_i^{-1}} & \ \mathrm{for}\ i\in X, j\in I,\\
	\label{defining relations B_c,s 3}
	K_{\beta}B_i = q^{-(\beta,\alpha_i)}B_iK_{\beta} & \ \mathrm{for}\ \beta\in Q^{\Theta}, i\in I,
	\end{align}
	together with the relations
	\begin{align*}
	K_{\beta}K_{\beta'} = K_{\beta'}K_{\beta} & \ \mathrm{for}\ \beta,\beta'\in Q^{\Theta}, \\
	F_{ij}(E_i,E_j) = 0 & \ \mathrm{for}\ i,j\in X, \\
	K_{\beta}E_i = q^{(\beta,\alpha_i)}E_iK_{\beta}& \ \mathrm{for}\ i\in X\ \mathrm{and}\ \beta\in Q^{\Theta}
	\end{align*}
	describing \(\mathcal{M}_X^+\) and \({U_{\Theta}^0}'\) that follow from (\ref{U_q(g) relations}) and (\ref{U_q(g) relations 3}).
\end{theoremKolb}

Our main goal in this paper will be to find explicit expressions for these lower degree polynomials \(C_{ij}(\bc)\), which, up to present, had not been written down in general. A few special cases had however already been treated by Kolb.

\begin{theoremKolb}[{\cite[equation (5.20), Theorem 7.3]{Kolb-2014}}]
	\label{theorem F_ij(B_i,B_j) = 0}
	For any \(i,j\in I\) such that either \(i\in X\) or \(\tau(i)\notin\{i,j\}\), one has
	\(F_{ij}(B_i,B_j) = C_{ij}(\bc) = 0\).
\end{theoremKolb}

Another case was treated by Balagovi\'{c} and Kolb in \cite{Balagovic&Kolb-2015}. It requires us to introduce some more notation. We will denote by \(\mathrm{ad}\) the left adjoint action of \(\Uqg\) on itself: for every \(x,u\in\Uqg\) one has
\[
\mathrm{ad}(x)(u) = x_{(1)}uS(x_{(2)}),
\]
where we have used the Sweedler notation, i.e.\ \(\Delta(x) = \sum x_{(1)}\otimes x_{(2)}\). It is not to be confused with the adjoint map of the Kac-Moody algebra \(\g\), which we have introduced in (\ref{adjoint map def}) under the same notation. Recall also the notation \(T_{w_X}\) for the product of Lusztig automorphisms corresponding to a reduced expression of \(w_X\).

\begin{lemmaKolb}[{\cite[equation (4.4), Theorem 4.4]{Kolb-2014}}]
	\label{lemma Z_i^+}
	For any \(i\in I\setminus X\) there exists a monomial
	\begin{equation}
	\label{Z_i^+ def}
	Z_i^+ = E_{j_1}E_{j_2}\dots E_{j_r} \in\mathcal{M}_X^+,
	\end{equation}
	with \(j_1,\dots,j_r\in X\), such that
	\[
	T_{w_X}(E_i) = a_i \mathrm{ad}(Z_i^+)(E_i),
	\]
	for some \(a_i\in\K(q)\). Moreover, one has
	\[
	\theta_q(F_iK_i)=-v_i\ad(Z_{\tau(i)}^+)(E_{\tau(i)}),
	\]
	for some \(v_i\in\K(q)^{\times}\).
\end{lemmaKolb} 

For any \(i\in I\setminus X\) we may now define
\begin{equation}
\label{Z_i def}
\mathcal{Z}_i = -v_i \mathrm{ad}(Z_{\tau(i)}^+)(K_{\tau(i)}^2)K_{\tau(i)}^{-1}K_i^{-1},
\end{equation}
where \(Z_{\tau(i)}^+\) and \(v_i\) are as defined in Lemma \ref{lemma Z_i^+}. It follows immediately from (\ref{Z_i^+ def}) and the expression (\ref{Coproduct, counit, antipode def}) for \(\Delta(E_j)\) that \(\mathcal{Z}_i\) is a \(\K(q)\)-linear combination of elements of \(\mathcal{M}_X^+\), multiplied by \(K_{\tau(i)}K_i^{-1}\). For any \(i\in I\setminus X\) we have  \(K_{\tau(i)}K_i^{-1}\in\Bcs\cap {U^0}' ={U_{\Theta}^0}'\) by the requirement (\ref{def admissible pair requirement 2}) in Definition \ref{admissible pair def}, and hence \(\mathcal{Z}_i\) lies in \(\mathcal{M}_X^+{U_{\Theta}^0}'\).

Furthermore, we will use the notation
\[
(x;x)_m = \prod_{k=1}^m(1-x^k).
\]
This enables us to state the following theorem by Balagovi\'{c} and Kolb.

\begin{theoremKolb}[{\cite[Theorem 3.6]{Balagovic&Kolb-2015}}]
	\label{theorem Balagovic}
	For any \(i\in I\setminus X\) satisfying \(\tau(i) = j\neq i\) one has
	\[
	C_{ij}(\bc) = -\frac{1}{(q_i-q_i^{-1})^2}\left( 
	q_i^{a_{ij}-1}(q_i^2;q_i^2)_{1-a_{ij}}c_iB_i^{-a_{ij}}\mathcal{Z}_i+q_i(q_i^{-2};q_i^{-2})_{1-a_{ij}}c_jB_i^{-a_{ij}}\mathcal{Z}_j.
	\right)
	\]
\end{theoremKolb}

By Theorems \ref{theorem F_ij(B_i,B_j) = 0} and \ref{theorem Balagovic}, it only remains to compute \(C_{ij}(\bc)\) in 2 cases, namely
\begin{description}
	\item[Case 1] \(i\in I\setminus X\), \(j\in I\setminus X\) and \(\tau(i) = i\),
	\item[Case 2] \(i\in I\setminus X\), \(j\in X\) and \(\tau(i) = i\).
\end{description}
These cases turn out to be remarkably complicated. In \cite{Kolb-2014} and \cite{Balagovic&Kolb-2015}, explicit calculations have led to expressions for \(C_{ij}(\bc)\) for \(a_{ij}\in\{0,-1,-2,-3\}\) in Case 1 and for \(a_{ij}\in\{0,-1,-2\}\) in Case 2, but no attempt has been made to write down relations valid without restrictions on \(a_{ij}\). In Section \ref{Section quantum Serre relations}, we will derive such relations for both cases. As could be expected from the above mentioned calculations, these expressions will be rather intricate, but nevertheless easily computable, as they involve only finite sums and products of elements of \(\K(q)\).

The key tool to obtain such relations is the projection \(P_{-\lambda_{ij}}\) introduced by Kolb. The classical triangular decomposition for quantum groups can be deformed to 
\begin{equation}
\label{triangular decomposition deformed}
\Uqgp \cong U^+\otimes {U^0}'\otimes S(U^-),
\end{equation}
where the isomorphism is given by multiplication, and consequently
\begin{equation}
\label{decomposition with P}
\Uqgp = \bigoplus_{\beta\in Q}U^+K_{\beta}S(U^-).
\end{equation} 
Let 
\begin{equation}
\label{P_lambda def}
P_{-\lambda_{ij}}: U_q(\g')\to U^+K_{-\lambda_{ij}}S(U^-)
\end{equation} 
denote the corresponding projection with respect to the decomposition (\ref{decomposition with P}), where 
\begin{equation}
\label{lambda_ij def}
\lambda_{ij} = (1-a_{ij})\alpha_i+\alpha_j\in Q.
\end{equation}
Then one can prove the following statements.

\begin{lemmaKolb}[{\cite[equation (5.14)]{Kolb-2014}}]
	\label{lemma (5.14)}
	\(P_{-\lambda_{ij}}\) is a homomorphism of left \(\Uqgp\)-comodules:
	\[
	(\Delta\circ P_{-\lambda_{ij}})(v) = (\mathrm{id}\otimes P_{-\lambda_{ij}})\Delta(v),
	\]
	for any \(v\in U_q(\g')\).
\end{lemmaKolb}

\begin{propositionKolb}[{\cite[Proposition 5.16]{Kolb-2014}}]
	\label{prop 5.16 Kolb}
	For any distinct \(i,j\in I\) one has 
	\[
	P_{-\lambda_{ij}}\left(F_{ij}(B_i,B_j)\right) = 0.
	\]
\end{propositionKolb}

Combining Lemma \ref{lemma (5.14)}, Proposition \ref{prop 5.16 Kolb} and the fact that \(\Delta\) is an algebra morphism, we find that
\begin{align}
\label{term is 0}
\begin{split}
(\mathrm{id}\otimes P_{-\lambda_{ij}})(F_{ij}(\Delta(B_i),\Delta(B_j))) &= (\mathrm{id}\otimes P_{-\lambda_{ij}})\Delta(F_{ij}(B_i,B_j)) \\ & = (\Delta\circ P_{-\lambda_{ij}})(F_{ij}(B_i,B_j)) = 0.
\end{split}
\end{align}
Since \(K_{-\lambda_{ij}}\) is invariant under \(P_{-\lambda_{ij}}\) and sent to \(1\) by the counit \(\epsilon\), the expression (\ref{term is 0}) asserts 
\begin{equation}
\label{starting expression}
F_{ij}(B_i,B_j) = C_{ij}(\bc) = (\mathrm{id}\otimes \epsilon)(\mathrm{id}\otimes P_{-\lambda_{ij}})\left(F_{ij}(B_i,B_j)\otimes K_{-\lambda_{ij}}-F_{ij}(\Delta(B_i),\Delta(B_j))\right),
\end{equation}
where we identify \(\Uqgp\) with \(\K\otimes\Uqgp\).
Our main purpose in Section \ref{Section quantum Serre relations} will be to expand the right-hand side of (\ref{starting expression}) as a polynomial in \(B_J\), \(J\in\mathcal{J}_{i,j}\), with coefficients in \(\mathcal{M}_X^+{U_{\Theta}^0}'\). To do so, we will need an expression for the \(\Delta(B_i)\) and \(\Delta(B_j)\) in (\ref{starting expression}). These follow from the following lemma.

\begin{lemmaKolb}[{\cite[Lemma 7.7]{Kolb-2014}}]
	\label{lemma Delta(B_i)}
	Let \(i\in I\setminus X\) be such that \(\tau(i) = i\) and \(j\in X\). Then there exists an element \(\mathcal{W}_{ij}\in\mathcal{M}_X^+\), independent of \(\bc\), such that
	\begin{equation}
	\label{Delta(B_i) expression}
	\Delta(B_i) = B_i\otimes K_i^{-1} + 1\otimes F_i + c_i\mathcal{Z}_i\otimes E_iK_i^{-1} + c_i\mathcal{W}_{ij}K_j\otimes(E_jE_i-q_i^{a_{ij}}E_iE_j)K_i^{-1}+\Upsilon_i,
	\end{equation}
	for some 
	\[
	\Upsilon_i\in\mathcal{M}_X^+{U_{\Theta}^0}'\otimes \widehat{U^+_i}K_i^{-1},
	\]
	where \(\widehat{U^+_i} = \{u\in \mathcal{M}_X^+E_i\mathcal{M}_X^+: \exists \gamma\in Q, \gamma>\alpha_i, \gamma\neq\alpha_i+\alpha_j: u\in\Uqgp_{\gamma}\} \).
\end{lemmaKolb}

Note that the formulation of this lemma is somewhat stronger than the original one in \cite{Kolb-2014}, but one readily verifies the correctness of this extra restriction on \(\Upsilon_i\) upon computing \(\Delta(\ad(Z_i^+)(E_i))\).

Finally, let us note that the following relations follow immediately from (\ref{defining relations B_c,s 2})--(\ref{defining relations B_c,s 3}).

\begin{lemmaKolb}
	\label{lemma commutation B_i, Z_i, W_ij}
	Let \(i\in I\setminus X\) be such that \(\tau(i) = i\), then for any \(j\in I\setminus X\) one has
	\[
	[B_j,\mathcal{Z}_i] = 0,
	\]
	whereas for \(j\in X\) one has
	\[
	B_i\mathcal{W}_{ij}K_j = q_i^{a_{ij}}\mathcal{W}_{ij}K_jB_i.
	\]
\end{lemmaKolb}

\section{Quantum Serre relations for the algebras $B_{\bc,\bs}$}
\label{Section quantum Serre relations}

We are now ready to derive closed expressions for the quantum Serre relations (\ref{defining relations B_c,s 1}) by expanding the right-hand side of (\ref{starting expression}). Crucial in this respect is the presence of the morphism \(\mathrm{id}\otimes\epsilon\), which by (\ref{Coproduct, counit, antipode def}) tells us that no term containing a nontrivial element of \(U^+U^-\) in the second tensor component will survive in (\ref{starting expression}). This allows us to eliminate some of the terms in (\ref{Delta(B_i) expression}).

We will first focus on Case 1.

\begin{lemma}
	\label{lemma Case 1 elimination}
	Let \(i,j\in I\setminus X\) be distinct such that \(\tau(i) = i\). Then one has
	\begin{equation}
	\label{expr lemma Case 1 elimination}
	F_{ij}(B_i,B_j) = (\mathrm{id}\otimes (\epsilon\circ P_{-\lambda_{ij}}))\left[F_{ij}(B_i,B_j)\otimes K_{-\lambda_{ij}}-F_{ij}(B_i\otimes K_i^{-1}+1\otimes F_i+c_i\mathcal{Z}_i\otimes E_iK_i^{-1},B_j\otimes K_j^{-1})\right].
	\end{equation}
\end{lemma}
\begin{proof}
	First, let us note that the polynomial \(F_{ij}\) is of degree 1 and hence linear in its second argument. Since \(j\notin X\), the expression (\ref{Delta(B_i) expression}) for \(\Delta(B_i)\) contains no factors \(E_j\) or \(F_j\) in its second tensor component. Since \(\epsilon(E_j) = \epsilon(F_j) = 0\), the expression for \(\Delta(B_j)\) obtained from Lemma \ref{lemma Delta(B_i)}, together with the relation (\ref{starting expression}), asserts
	\begin{equation}
	\label{still Delta(B_i) to be expanded}
	F_{ij}(B_i,B_j) = (\mathrm{id}\otimes (\epsilon\circ P_{-\lambda_{ij}}))\left(F_{ij}(B_i,B_j)\otimes K_{-\lambda_{ij}}-F_{ij}(\Delta(B_i),B_j\otimes K_j^{-1})\right).
	\end{equation}
	When expanding \(\Delta(B_i)\) according to (\ref{Delta(B_i) expression}), there will be no contribution from the two latter terms
	\begin{equation}
	\label{rest does not contribute}
	c_i\mathcal{W}_{ik}K_k\otimes(E_kE_i-q_i^{a_{ik}}E_iE_k)K_i^{-1}+\Upsilon_i,
	\end{equation}
	with \(k\in X\), since each term in (\ref{rest does not contribute}) contains at least one factor \(E_{k'}\), \(k'\in X\), in its second tensor component, and again \(\epsilon(E_{k'}) = 0\). Hence \(\Delta(B_i)\) in (\ref{still Delta(B_i) to be expanded}) can be replaced by \(B_i\otimes K_i^{-1} + 1\otimes F_i + c_i\mathcal{Z}_i\otimes E_iK_i^{-1}\). 
\end{proof}

The same simplification can be performed for Case 2.

\begin{lemma}
	\label{lemma Case 2 elimination}
	Let \(i\in I\setminus X\) be such that \(\tau(i) = i\) and let \(j\in X\). Then one has
	\begin{align}
	\label{expr lemma Case 2 elimination}
	\begin{split}
	F_{ij}(B_i,B_j) =\ & (\mathrm{id}\otimes (\epsilon\circ P_{-\lambda_{ij}}))\left[F_{ij}(B_i,B_j)\otimes K_{-\lambda_{ij}}-F_{ij}(B_i\otimes K_i^{-1}+1\otimes F_i+c_i\mathcal{Z}_i\otimes E_iK_i^{-1},B_j\otimes K_j^{-1})\right. \\
	&\left.
	-F_{ij}(B_i\otimes K_i^{-1}+1\otimes F_i+c_i\mathcal{Z}_i\otimes E_iK_i^{-1}+c_i\mathcal{W}_{ij}K_j\otimes(E_jE_i-q_i^{a_{ij}}E_iE_j)K_i^{-1},1\otimes F_j)
	\right].
	\end{split}
	\end{align}
\end{lemma}
\begin{proof}
	Since \(j\in X\), we have \(B_j = F_j\). Hence it follows from (\ref{starting expression}), (\ref{Coproduct, counit, antipode def}) and the linearity of \(F_{ij}\) in its second argument that 
	\[
	F_{ij}(B_i,B_j) =  (\mathrm{id}\otimes (\epsilon\circ P_{-\lambda_{ij}}))\left(F_{ij}(B_i,B_j)\otimes K_{-\lambda_{ij}}-F_{ij}(\Delta(B_i),B_j\otimes K_j^{-1})-F_{ij}(\Delta(B_i),1\otimes F_j)\right).
	\]
	We will now expand \(\Delta(B_i)\) using (\ref{Delta(B_i) expression}) with the given \(j\). In the first occurrence of \(\Delta(B_i)\), both \(c_i\mathcal{W}_{ij}K_j\otimes(E_jE_i-q_i^{a_{ij}}E_iE_j)K_i^{-1}\) and \(\Upsilon_i\) will not contribute, since each of their terms contains at least one factor \(E_k\) with \(k\neq i\) in the second tensor component, and \(\epsilon(E_k) = 0\). For the second occurrence of \(\Delta(B_i)\), the situation is different. The term \(c_i\mathcal{W}_{ij}K_j\otimes(E_jE_i-q_i^{a_{ij}}E_iE_j)K_i^{-1}\) will effectively contribute, since when expanding \(F_{ij}(\Delta(B_i),1\otimes F_j)\), we may use the rule \(F_jE_j = E_jF_j - \frac{K_j-K_j^{-1}}{q_j-q_j^{-1}}\) and the last term in this expansion will turn out to be significant, as will be explained in what follows. The term \(\Upsilon_i\) in \(\Delta(B_i)\) however, will still not contribute. Indeed, each term in \(\Upsilon_i\) contains either a factor \(E_j^2\) or a factor \(E_k\) with \(k\in X\setminus \{j\}\), and both \(F_jE_j^2\) and \(F_jE_k\) cannot be expanded to yield a non-vanishing term under \(\epsilon\). This proves the claim.
\end{proof}

One observes immediately that the right-hand side of (\ref{expr lemma Case 2 elimination}) equals the right-hand side of (\ref{expr lemma Case 1 elimination}), added with a second term.
In what follows, we will treat both terms separately and thereby obtain explicit expressions for each of the two cases.

\subsection{Binary expansions}
\label{Subsection Binary expansions}

In this section, we will expand the right-hand sides of (\ref{expr lemma Case 1 elimination}) and (\ref{expr lemma Case 2 elimination}). We will first treat the right-hand side of (\ref{expr lemma Case 1 elimination}), which occurs in (\ref{expr lemma Case 2 elimination}) as well and which can, to a large extent, be rewritten irrespective of whether or not \(j\) lies in \(X\). The second term, which appears only in (\ref{expr lemma Case 2 elimination}), i.e.\ for \(j\in X\), will be addressed afterwards.

Our main strategy will be to perform a\ \textquotedblleft binary\textquotedblright\ distributive expansion, which requires summation over binary tuples \(\bl\in \{0,1\}^N\), \(N\in\N\). For any such tuple \(\bl\), we will use the notation
\begin{equation}
\vert \bl\vert = \ell_1+\ell_2+\dots+\ell_N, \quad \vert\bl\vert_{r;s} = \left\{
\begin{array}{ll}
\ell_r+\ell_{r+1}+\dots+\ell_s \qquad & \mathrm{if}\ r \leq s, \\
0 & \mathrm{otherwise}.
\end{array}\right.
\end{equation}
Throughout this paper, we will often meet finite sums and products over natural numbers. We will use the convention that a sum vanishes if its lower bound exceeds its upper bound or equivalently if it ranges over the empty set, whereas a product reduces to one in this situation. Otherwise stated, for any function \(a\) of \(r\) and any \(M>N\) we take \(\sum_{r=M}^N a(r) = \sum_{r\in\emptyset}a(r) = 0\) and \(\prod_{r=M}^Na(r) = \prod_{r \in\emptyset}a(r) = 1\). Note also that in our convention \(0\) is a natural number, i.e.\ \(0\in\N\), and \(0^0 = 1\).

\begin{proposition}
	\label{prop term 1 expansion with epsilon and P}
	Let \(i\in I\setminus X\) be such that \(\tau(i) = i\) and let \(j\in I\) be distinct from \(i\). Then one has
	\begin{align}
	\label{term 1 expansion with epsilon and P}
	\begin{split}
	&\left(\mathrm{id}\otimes(\epsilon\circ P_{-\lambda_{ij}})\right)\left[F_{ij}(B_i,B_j)\otimes K_{-\lambda_{ij}}-F_{ij}(B_i\otimes K_i^{-1}+1\otimes F_i+c_i\mathcal{Z}_i\otimes E_iK_i^{-1}, B_j\otimes K_j^{-1})\right] \\
	=\ & \sum_{k = 0}^{1-a_{ij}}\sum_{\substack{\bl\in\{0,1\}^{1-a_{ij}}\\\vert\bl\vert\neq 1-a_{ij}}}\sum_{\bss\in\{0,1\}^{1-a_{ij}-\vert\bl\vert}}(-1)^{k+1}\begin{bmatrix}
	1-a_{ij} \\ k
	\end{bmatrix}_{q_i}(\epsilon\circ P_{-\lambda_{ij}})(\mathfrak{p}_{\bl,\bss,k}^{(i,j, a_{ij})})\\&(c_i\mathcal{Z}_i)^{\sum_{r = 1}^{1-a_{ij}-k}(1-\ell_r)(1-s_{r-\vert\bl\vert_{1;r}})}B_i^{\vert\bl\vert_{1;1-a_{ij}-k}}B_jB_i^{\vert\bl\vert_{2-a_{ij}-k;1-a_{ij}}}(c_i\mathcal{Z}_i)^{\sum_{r = 2-a_{ij}-k}^{1-a_{ij}}(1-\ell_r)(1-s_{r-\vert\bl\vert_{1;r}})},
	\end{split}
	\end{align} 
	where
	\begin{equation}
	\label{p_l,s,k def}
	\mathfrak{p}_{\bl,\bss,k}^{(i,j,a_{ij})} =  \left(\overrightarrow{\prod_{r = 1}^{1-a_{ij}-k}} \mathcal{T}^{i}_{\bl,\bss,r}\right) K_j^{-1}\left(\overrightarrow{\prod_{r = 2-a_{ij}-k}^{1-a_{ij}}}\mathcal{T}^{i}_{\bl,\bss,r} \right),
	\end{equation}
	with
	\begin{equation}
	\label{Ti_l,s,r def}
	\mathcal{T}^{i}_{\bl,\bss,r} = K_i^{-\ell_r}F_i^{(1-\ell_r)s_{r-\vert\bl\vert_{1;r}}}(E_iK_i^{-1})^{(1-\ell_r)(1-s_{r-\vert\bl\vert_{1;r}})}.
	\end{equation}
\end{proposition}
\begin{proof}
	By the definition (\ref{Fij def}) of \(F_{ij}\), we have
	\begin{align*}
	& -F_{ij}(B_i\otimes K_i^{-1}+1\otimes F_i+c_i\mathcal{Z}_i\otimes E_iK_i^{-1},B_j\otimes K_j^{-1})\\ =\ & \sum_{k = 0}^{1-a_{ij}}(-1)^{k+1}\begin{bmatrix} 1-a_{ij} \\ k
	\end{bmatrix}_{q_i}\left(B_i\otimes K_i^{-1}+1\otimes F_i+c_i\mathcal{Z}_i\otimes E_iK_i^{-1} \right)^{1-a_{ij}-k}\left(B_j\otimes  K_j^{-1}\right)\\&\phantom{\sum\ \,}\left(B_i\otimes K_i^{-1}+1\otimes F_i+c_i\mathcal{Z}_i\otimes E_iK_i^{-1}\right)^{k}.
	\end{align*}
	The term \(\left(B_i\otimes K_i^{-1}+1\otimes F_i+c_i\mathcal{Z}_i\otimes E_iK_i^{-1}\right)^{1-a_{ij}-k}\) can be expanded distributively as 
	\[
	\sum_{\bl\in\{0,1\}^{1-a_{ij}-k}}\overrightarrow{\prod_{r = 1}^{1-a_{ij}-k}}\left(B_i\otimes K_i^{-1}\right)^{\ell_r}\left(1\otimes F_i+c_i\mathcal{Z}_i\otimes E_iK_i^{-1}\right)^{1-\ell_r}
	\] 
	and for each \(\bl\in\{0,1\}^{1-a_{ij}-k}\) one has
	\begin{align*}
	&\overrightarrow{\prod_{r = 1}^{1-a_{ij}-k}}\left(B_i\otimes K_i^{-1}\right)^{\ell_r}\left(1\otimes F_i+c_i\mathcal{Z}_i\otimes E_iK_i^{-1}\right)^{1-\ell_r} \\
	= & \sum_{\bss\in\{0,1\}^{1-a_{ij}-k-\vert\bl\vert_{1;1-a_{ij}-k}}}\overrightarrow{\prod_{r = 1}^{1-a_{ij}-k}}\left(B_i\otimes K_i^{-1}\right)^{\ell_r}\left(1\otimes F_i\right)^{(1-\ell_r)s_{r-\vert\bl\vert_{1;r}}}\left(c_i\mathcal{Z}_i\otimes E_iK_i^{-1}\right)^{(1-\ell_r)(1-s_{r-\vert\bl\vert_{1;r}})}.
	\end{align*}
	The rationale of this expansion is that for \(\ell_r = 1\), we get the contribution of \(B_i\otimes K_i^{-1}\), for \(\ell_r = 0\) and \(s_{r-\vert\bl\vert_{1;r}}= 1\) we find \(1\otimes F_i\), whereas for \(\ell_r = 0\) and \(s_{r-\vert\bl\vert_{1;r}}=0\) we have \(c_i\mathcal{Z}_i\otimes E_iK_i^{-1}\). Note that the indexation of the \(s\)-variables was chosen in such a way that there is only a summation over these variables in case the corresponding \(\ell_r = 0\). Indeed, if we were to write \(s_r\) instead of \(s_{r-\vert\bl\vert_{1;r}}\) and sum over all \(s_1,\dots,s_{1-a_{ij}-k}\in\{0,1\}\), then the terms corresponding to \(\ell_r = 1\) would contribute twice. 
	
	Since \(B_i\) commutes with \(\mathcal{Z}_i\) by Lemma \ref{lemma commutation B_i, Z_i, W_ij}, we obtain
	\begin{align*}
	&\left(B_i\otimes K_i^{-1}+1\otimes F_i+c_i\mathcal{Z}_i\otimes E_iK_i^{-1}\right)^{1-a_{ij}-k} \\
	=& \sum_{\bl\in\{0,1\}^{1-a_{ij}-k}}\sum_{\bss\in\{0,1\}^{1-a_{ij}-k-\vert\bl\vert_{1;1-a_{ij}-k}}}
	(c_i\mathcal{Z}_i)^{\sum_{r = 1}^{1-a_{ij}-k}(1-\ell_r)(1-s_{r-\vert\bl\vert_{1;r}})}B_i^{\vert\bl\vert_{1;1-a_{ij}-k}}\otimes \overrightarrow{\prod_{r = 1}^{1-a_{ij}-k}} \mathcal{T}^{i}_{\bl,\bss,r},
	\end{align*}
	with \(\mathcal{T}^{i}_{\bl,\bss,r}\) as in (\ref{Ti_l,s,r def}). Performing a similar expansion for the term \(\left(B_i\otimes K_i^{-1}+1\otimes F_i+c_i\mathcal{Z}_i\otimes E_iK_i^{-1}\right)^{k}\), we find that \(-F_{ij}(B_i\otimes K_i^{-1}+1\otimes F_i+c_i\mathcal{Z}_i\otimes E_iK_i^{-1},B_j\otimes K_j^{-1})\) equals
	\begin{align*}
	& \sum_{k = 0}^{1-a_{ij}}\sum_{\boldsymbol{\ell}\in\{0,1\}^{1-a_{ij}}}\sum_{\bss\in\{0,1\}^{1-a_{ij}-\vert\bl\vert}}
	(-1)^{k+1}\begin{bmatrix} 1-a_{ij} \\ k
	\end{bmatrix}_{q_i}\\& (c_i\mathcal{Z}_i)^{\sum_{r = 1}^{1-a_{ij}-k}(1-\ell_r)(1-s_{r-\vert\bl\vert_{1;r}})}B_i^{\vert\bl\vert_{1;1-a_{ij}-k}}B_jB_i^{\vert\bl\vert_{2-a_{ij}-k;1-a_{ij}}}(c_i\mathcal{Z}_i)^{\sum_{r = 2-a_{ij}-k}^{1-a_{ij}}(1-\ell_r)(1-s_{r-\vert\bl\vert_{1;r}})}\otimes\mathfrak{p}_{\bl,\bss,k}^{(i,j,a_{ij})},
	\end{align*}
	with \(\mathfrak{p}_{\boldsymbol{\ell},\bss,k}^{(i,j,a_{ij})}\) as in (\ref{p_l,s,k def}). It remains only to observe that the term corresponding to \(\vert\bl\vert = 1-a_{ij}\), i.e.\ \(\bl = (1,1,\dots,1)\), yields
	\[
	\sum_{k = 0}^{1-a_{ij}}
	(-1)^{k+1}\begin{bmatrix} 1-a_{ij} \\ k
	\end{bmatrix}_{q_i} B_i^{1-a_{ij}-k}B_jB_i^{k}\otimes K_i^{-(1-a_{ij})}K_j^{-1} = -F_{ij}(B_i,B_j)\otimes K_{-\lambda_{ij}}.
	\]
\end{proof}

Many of the \(\bss\) in the sum in (\ref{term 1 expansion with epsilon and P}) will have a vanishing contribution. One can make the following observation.

\begin{lemma}
	\label{lemma restrictions on s}
	Let \(i\in I\setminus X\) be such that \(\tau(i) = i\) and let \(j\in I\) be distinct from \(i\). Let \(\bl\in\{0,1\}^{1-a_{ij}}\) with \(\vert\bl\vert\neq 1-a_{ij}\), \(\bss\in\{0,1\}^{1-a_{ij}-\vert\bl\vert}\) and \(k\in\{0,\dots,1-a_{ij}\}\). Then one has
	\[
	(\epsilon\circ P_{-\lambda_{ij}})(\mathfrak{p}_{\bl,\bss,k}^{(i,j,a_{ij})}) = 0
	\] 
	if one of the following conditions is fulfilled:
	\begin{enumerate}[label=(\alph*)]
		\item\label{lemma restrictions on s condition a} \(a_{ij}+\vert\bl\vert\) is even,
		\item\label{lemma restrictions on s condition b} \(\vert\bss\vert \neq \frac{1-a_{ij}-\vert\bl\vert}{2}\),
		\item\label{lemma restrictions on s condition c} There exists \(p\in\{1,\dots,1-a_{ij}-\vert\bl\vert\}\) such that \(\vert \bss\vert_{1;p} < \tfrac{p}{2}\).
	\end{enumerate}
\end{lemma}
\begin{proof}
	To acquire the action of \(\epsilon\circ P_{-\lambda_{ij}}\) on \(\plsk\), we will write \(\plsk\) in a standard ordering, namely as a \(\K(q)\)-linear combination of elements of the form \(E_i^{N_1}F_i^{N_2}K_i^{N_3}K_j^{-1}\), with \(N_1,N_2\in\N\) and \(N_3\in\Z\). We may do so by applying the \(U_q(\g')\)-relations (\ref{U_q(g) relations})--(\ref{U_q(g) relations 2}). Each such element will be projected to either itself or 0 by \(P_{-\lambda_{ij}}\). But when applying \(\epsilon\), such a term can only survive if \(N_1 = N_2 = 0\), by (\ref{Coproduct, counit, antipode def}). Suppose now \(\bss\) is such that \(\plsk\) contains an unequal number of factors \(F_i\) and \(E_i\). Then each term in its standard ordering will still contain an unequal number of factors \(F_i\) and \(E_i\), as follows from (\ref{U_q(g) relations 2}). Hence the standard ordering will consist of terms \(E_i^{N_1}F_i^{N_2}K_i^{N_3}K_j^{-1}\) with either \(N_1\) or \(N_2\) non-zero, which will be killed by \(\epsilon\). Thus we must have an equal number of factors \(F_i\) and \(E_i\) in \(\plsk\). This number must then of course be half the total number of factors in \(\plsk\) with \(\ell_r = 0\), i.e.
	\[
	\vert\bss\vert = \frac{1-a_{ij}-\vert\bl\vert}{2}.
	\]
	
	If \(a_{ij}+\vert\bl\vert\) is even, then the total number of factors in \(\plsk\) with \(\ell_r = 0\) will be odd and hence the number of factors \(F_i\) and \(E_i\) in \(\plsk\) will always be unequal, for any \(\bss\). 
	
	Finally, suppose \(p\in\{1,\dots,1-a_{ij}-\vert\bl\vert\}\) is such that \(\vert \bss\vert_{1;p} < \tfrac{p}{2}\). This means that up to position \(p\), the number of factors \(E_i\) will exceed the number of factors \(F_i\). As the difference between these numbers is not altered by the relation (\ref{U_q(g) relations 2}), this means that the standard ordering of the corresponding term will consist only of terms \(E_i^{N_1}F_i^{N_2}K_i^{N_3}K_j^{-1}\) with \(N_1\geq 1\), which are again killed by \(\epsilon\). 
\end{proof}

This result will help us to simplify the notation used in Proposition \ref{prop term 1 expansion with epsilon and P}. Indeed, by Condition \ref{lemma restrictions on s condition b} in Lemma \ref{lemma restrictions on s}, we know that for any \((\bl,\bss,k)\) contributing non-trivially to (\ref{term 1 expansion with epsilon and P}), we have
\begin{equation}
\label{full sum}
\sum_{r = 1}^{1-a_{ij}}(1-\ell_r)(1-s_{r-\vert\bl\vert_{1;r}}) = 1-a_{ij}-\vert\bl\vert - \vert\bss\vert = \frac{1-a_{ij}-\vert\bl\vert}{2}
\end{equation}
and hence also 
\begin{align}
\label{full sums 2}
\begin{split}
\sum_{r = 1}^{1-a_{ij}-k}(1-\ell_r)(1-s_{r-\vert\bl\vert_{1;r}}) & = 1-a_{ij}-k-\vert\bl\vert_{1;1-a_{ij}-k}-\vert\bss\vert_{1;1-a_{ij}-k-\vert\bl\vert_{1;1-a_{ij}-k}}, \\
\sum_{r = 2-a_{ij}-k}^{1-a_{ij}}(1-\ell_r)(1-s_{r-\vert\bl\vert_{1;r}}) & = \frac{1-a_{ij}-\vert\bl\vert}{2}-\sum_{r = 1}^{1-a_{ij}-k}(1-\ell_r)(1-s_{r-\vert\bl\vert_{1;r}}).
\end{split}
\end{align}

Moreover,  we will need the notion of the even and an odd part of an integer number \(d\in\Z\), denoted by \(d_e\) and \(d_p\) respectively, and defined as
\begin{equation}
\label{even and odd part}
d_e = \left\lfloor\frac{d}{2}\right\rfloor = \left\{\arraycolsep=1.4pt\def\arraystretch{2}\begin{array}{ll}
\dfrac{d}{2} \quad &\mathrm{for}\ d\ \mathrm{even} \\
\dfrac{d-1}{2}\quad & \mathrm{for}\ d\ \mathrm{odd}
\end{array} \right., \qquad\qquad d_p = \left\{\arraycolsep=1.4pt\def\arraystretch{1.5}\begin{array}{ll}
0 \quad &\mathrm{for}\ d\ \mathrm{even} \\
1 & \mathrm{for}\ d\ \mathrm{odd}
\end{array} \right..
\end{equation}
Note that for any \(d\in\Z\) one has \(d = 2d_e+d_p\).

This will now help us to rewrite \(C_{ij}(\bc)\) for Case 1.

\begin{corollary}[Case 1]
	\label{cor F_ij(B_i,B_j) Case 1 with epsilon and P}
	Let \(i\in I\setminus X\) be such that \(\tau(i) = i\) and let \(j\in I\setminus X\) be distinct from \(i\). Then one has
	\begin{equation}
	\label{F_ij(B_i,B_j) Case 1 with epsilon and P}
	F_{ij}(B_i,B_j) = C_{ij}(\bc) = \sum_{m= 0}^{-1-a_{ij}}\sum_{m' = 0}^{-1-a_{ij}-m}\rho_{m,m'}^{(i,j,a_{ij})}\mathcal{Z}_i^{\frac{1-a_{ij}-m-m'}{2}}B_i^mB_jB_i^{m'},
	\end{equation}
	where
	\begin{equation}
	\label{rho_m,m' def}
	\rho_{m,m'}^{(i,j,a_{ij})} = (a_{ij}+m+m')_pc_i^{\frac{1-a_{ij}-m-m'}{2}}\sum_{k = m'}^{1-a_{ij}-m}\sum_{\bl\in\mathcal{L}_{m,m',k}}\sum_{\bss\in\mathscr{S}_{m,m'}}(-1)^{k+1}\begin{bmatrix}
	1-a_{ij}\\k
	\end{bmatrix}_{q_i}(\epsilon\circ P_{-\lambda_{ij}})(\plsk),
	\end{equation}
	with \(\plsk\) as in (\ref{p_l,s,k def}) and
	\begin{align}
	\label{L and S sets for Case 1}
	\begin{split}
	\mathcal{L}_{m,m',k}\ & = \{\bl\in\{0,1\}^{1-a_{ij}}: \vert\bl\vert_{1;1-a_{ij}-k} = m\ \mathrm{and}\ \vert\bl\vert_{2-a_{ij}-k;1-a_{ij}}= m'\}, \\
	\mathscr{S}_{m,m'}\ & = \{\bss\in\{0,1\}^{1-a_{ij}-m-m'}: \vert\bss\vert = \frac{1-a_{ij}-m-m'}{2}\ \mathrm{and}\ \vert\bss\vert_{1;p}\geq \frac{p}{2},\forall p\in\{1,\dots,1-a_{ij}-m-m'\}\}.
	\end{split}
	\end{align}
\end{corollary}
\begin{proof}
	Upon combining Lemma \ref{lemma Case 1 elimination}, Proposition \ref{prop term 1 expansion with epsilon and P}, the equation (\ref{full sum}) and the fact that \([B_i,\mathcal{Z}_i] = [B_j,\mathcal{Z}_i] = 0\) by Lemma \ref{lemma commutation B_i, Z_i, W_ij}, one finds that \(F_{ij}(B_i,B_j)\) equals
	\[
	\sum_{k = 0}^{1-a_{ij}}\sum_{\substack{\bl\in\{0,1\}^{1-a_{ij}}\\\vert\bl\vert\neq 1-a_{ij}}}\sum_{\bss\in\{0,1\}^{1-a_{ij}-\vert\bl\vert}}(-1)^{k+1}\begin{bmatrix}
	1-a_{ij} \\ k
	\end{bmatrix}_{q_i}(\epsilon\circ P_{-\lambda_{ij}})(\mathfrak{p}_{\bl,\bss,k}^{(i,j, a_{ij})})(c_i\mathcal{Z}_i)^{\frac{1-a_{ij}-\vert\bl\vert}{2}}B_i^{\vert\bl\vert_{1;1-a_{ij}-k}}B_jB_i^{\vert\bl\vert_{2-a_{ij}-k;1-a_{ij}}}.
	\]
	We can restrict the sum over \(\bl\) to one over \(\mathcal{L}_{m,m',k}\), by setting
	\begin{equation}
	\label{restrictions on l}
	m = \vert\bl\vert_{1;1-a_{ij}-k}, \quad m' = \vert\bl\vert_{2-a_{ij}-k;1-a_{ij}}.
	\end{equation}
	This requires an additional summation over \(m\) and \(m'\). A priori, we have \(m+m' = \vert\bl\vert \leq -a_{ij}\), but if \(m+m' = -a_{ij}\), then \((\epsilon\circ P_{-\lambda_{ij}})(\plsk)\) will vanish for any \(\bss\), by Condition \ref{lemma restrictions on s condition a} of Lemma \ref{lemma restrictions on s}. This explains the presence of \((a_{ij}+m+m')_p\) in (\ref{rho_m,m' def}) and the fact that in the sum in (\ref{F_ij(B_i,B_j) Case 1 with epsilon and P}) we restrict to \(m+m' \leq -1-a_{ij}\). Note also that the requirements (\ref{restrictions on l}) imply that
	\[
	1-a_{ij}-k\geq m\ \mathrm{and}\ k\geq m'.
	\]
	Similarly, the sum over \(\bss\) may be restricted to \(\mathscr{S}_{m,m'}\) by Conditions \ref{lemma restrictions on s condition b} and \ref{lemma restrictions on s condition c} of Lemma \ref{lemma restrictions on s}.
\end{proof}

For Case 2, the first line of the right-hand side of (\ref{expr lemma Case 2 elimination}) is identical to the right-hand side of (\ref{expr lemma Case 1 elimination}), and hence the first part of \(C_{ij}(\bc)\) can be expanded as above. Nevertheless, we have to take into account that in this case \(\mathcal{Z}_i\) and \(B_j\) no longer commute, which effects our notation.

\begin{corollary}[Case 2]
	\label{cor F_ij(B_i,B_j) Case 2 part 1 with epsilon and P}
	Let \(i\in I\setminus X\) be such that \(\tau(i) = i\) and let \(j\in X\). Then one has
	\begin{align}
	\label{F_ij(B_i,B_j) Case 2 with epsilon and P}
	\begin{split}
	&F_{ij}(B_i,B_j) = C_{ij}(\bc)\\ = &\sum_{m= 0}^{-1-a_{ij}}\sum_{m' = 0}^{-1-a_{ij}-m}\sum_{t = 0}^{\frac{1-a_{ij}-m-m'}{2}}\rho_{m,m',t}^{(i,j,a_{ij})}\mathcal{Z}_i^tB_i^mB_jB_i^{m'}\mathcal{Z}_i^{\frac{1-a_{ij}-m-m'}{2}-t}\\ 
	&+ (\mathrm{id}\otimes (\epsilon\circ P_{-\lambda_{ij}}))\left(	-F_{ij}(B_i\otimes K_i^{-1}+1\otimes F_i+c_i\mathcal{Z}_i\otimes E_iK_i^{-1}+c_i\mathcal{W}_{ij}K_j\otimes(E_jE_i-q_i^{a_{ij}}E_iE_j)K_i^{-1},1\otimes F_j) \right),
	\end{split}
	\end{align}
	where
	\begin{equation}
	\label{rho_m,m',t def}
	\rho_{m,m',t}^{(i,j,a_{ij})} = (a_{ij}+m+m')_pc_i^{\frac{1-a_{ij}-m-m'}{2}}\sum_{k = m'}^{1-a_{ij}-m}\sum_{\bl\in\mathcal{L}_{m,m',k}}\sum_{\bss\in\mathscr{S}_{m,m',k,t}}(-1)^{k+1}\begin{bmatrix}
	1-a_{ij}\\k
	\end{bmatrix}_{q_i}(\epsilon\circ P_{-\lambda_{ij}})(\plsk),
	\end{equation}
	with \(\plsk\) as in (\ref{p_l,s,k def}), \(\mathcal{L}_{m,m',k}\) as in (\ref{L and S sets for Case 1}) and 
	\begin{align}
	\label{S_m,m',k,t set Case 2}
	\begin{split}
	\mathscr{S}_{m,m',k,t}\ = & \left\{\bss\in\{0,1\}^{1-a_{ij}-m-m'}: \vert\bss\vert = \frac{1-a_{ij}-m-m'}{2},\ \vert\bss\vert_{1;1-a_{ij}-k-m} = 1-a_{ij}-k-m-t\ \mathrm{and}\right.\\ &\left.\vert\bss\vert_{1;p}\geq \frac{p}{2},\forall p\in\{1,\dots,1-a_{ij}-m-m'\}\right\}.
	\end{split}
	\end{align}
\end{corollary}
\begin{proof}
	Upon combining Lemma \ref{lemma Case 2 elimination}, Proposition \ref{prop term 1 expansion with epsilon and P} and the equations (\ref{full sum}) and (\ref{full sums 2}), we obtain
	\begin{align*}
	&F_{ij}(B_i,B_j) \\
	= & \sum_{k = 0}^{1-a_{ij}}\sum_{\substack{\bl\in\{0,1\}^{1-a_{ij}}\\\vert\bl\vert\neq 1-a_{ij}}}\sum_{\bss\in\{0,1\}^{1-a_{ij}-\vert\bl\vert}}\left((-1)^{k+1}\begin{bmatrix}
	1-a_{ij}\\k
	\end{bmatrix}_{q_i} c_i^{\frac{1-a_{ij}-\vert\bl\vert}{2}}(\epsilon\circ P_{-\lambda_{ij}})(\plsk)\right.\\&\left. \mathcal{Z}_i^{t_{\bl,\bss,k}}B_i^{\vert\bl\vert_{1;1-a_{ij}-k}}B_jB_i^{\vert\bl\vert_{2-a_{ij}-k;1-a_{ij}}}\mathcal{Z}_i^{\frac{1-a_{ij}-\vert\bl\vert}{2}-t_{\bl,\bss,k}}\right)\\
	& + (\mathrm{id}\otimes (\epsilon\circ P_{-\lambda_{ij}}))\left(-F_{ij}(B_i\otimes K_i^{-1}+1\otimes F_i+c_i\mathcal{Z}_i\otimes E_iK_i^{-1}+c_i\mathcal{W}_{ij}K_j\otimes(E_jE_i-q_i^{a_{ij}}E_iE_j)K_i^{-1},1\otimes F_j)\right),
	\end{align*}
	with, by (\ref{full sums 2}), 
	\begin{equation}
	\label{t_l,s,k def}
	t_{\bl,\bss,k} = 1-a_{ij}-k-\vert\bl\vert_{1;1-a_{ij}-k}-\vert\bss\vert_{1;1-a_{ij}-k-\vert\bl\vert_{1;1-a_{ij}-k}}.
	\end{equation}
	
	The sum over \(\bl\) can be restricted to \(\mathcal{L}_{m,m',k}\), with an additional summation over \(m,m'\), just like in the proof of Corollary \ref{cor F_ij(B_i,B_j) Case 1 with epsilon and P}. Setting \(t_{\bl,\bss,k}\) equal to a parameter \(t\), over which we sum as well, determines the condition
	\[
	\vert\bss\vert_{1;1-a_{ij}-k-m} = 1-a_{ij}-k-m-t,
	\]
	as follows from (\ref{t_l,s,k def}). This restriction, together with Conditions \ref{lemma restrictions on s condition b} and \ref{lemma restrictions on s condition c} of Lemma \ref{lemma restrictions on s}, determines the definition of \(\mathscr{S}_{m,m',k,t}\).
\end{proof}

We will now perform a similar binary expansion for the last line of (\ref{F_ij(B_i,B_j) Case 2 with epsilon and P}). 

\begin{proposition}
	\label{prop term 2 expansion with epsilon and P}
	Let \(i\in I\setminus X\) be such that \(\tau(i) = i\) and let \(j\in X\). Then one has
	\begin{align}
	\label{F_{ij}(B_i,B_j) term 2 with epsilon and P}
	\begin{split}
	&(\mathrm{id}\otimes (\epsilon\circ P_{-\lambda_{ij}}))\left(-F_{ij}(B_i\otimes K_i^{-1}+1\otimes F_i+c_i\mathcal{Z}_i\otimes E_iK_i^{-1}+c_i\mathcal{W}_{ij}K_j\otimes(E_jE_i-q_i^{a_{ij}}E_iE_j)K_i^{-1},1\otimes F_j) \right)\\
	= & \sum_{k = 1}^{1-a_{ij}}\sum_{d = 0}^{k-1}\sum_{\bl\in\{0,1\}^{-a_{ij}}}\sum_{\bss\in\{0,1\}^{-a_{ij}-\vert\bl\vert}}\left((-1)^{k+1}\begin{bmatrix}
	1-a_{ij}\\k
	\end{bmatrix}_{q_i}q_i^{a_{ij}\vert\bl\vert_{1;1-a_{ij}-k+d}}(\epsilon\circ P_{-\lambda_{ij}})(\mathfrak{r}_{\bl,\bss,k,d}^{(i,j,a_{ij})})\right.\\
	&\left. 
	(c_i\mathcal{Z}_i)^{\sum_{r = 1}^{1-a_{ij}-k+d}(1-\ell_r)(1-s_{r-\vert\bl\vert_{1;r}})}(c_i\mathcal{W}_{ij}K_j)(c_i\mathcal{Z}_i)^{\sum_{r = 2-a_{ij}-k+d}^{-a_{ij}}(1-\ell_r)(1-s_{r-\vert\bl\vert_{1;r}})}B_i^{\vert\bl\vert}
	\right),
	\end{split}
	\end{align}
	where
	\begin{equation}
	\label{r_l,s,k,d def}
	\mathfrak{r}_{\bl,\bss,k,d}^{(i,j,a_{ij})} = \left(\overrightarrow{\prod_{r = 1}^{1-a_{ij}-k}}\mathcal{T}^{i}_{\bl,\bss,r}\right)F_j\left(\overrightarrow{\prod_{r = 2-a_{ij}-k}^{1-a_{ij}-k+d}}\mathcal{T}^{i}_{\bl,\bss,r}\right)(E_jE_i-q_i^{a_{ij}}E_iE_j)K_i^{-1}\left(\overrightarrow{\prod_{r = 2-a_{ij}-k+d}^{-a_{ij}}}\mathcal{T}^{i}_{\bl,\bss,r}\right),
	\end{equation}
	with \(\mathcal{T}^{i}_{\bl,\bss,r}\) as in (\ref{Ti_l,s,r def}).
\end{proposition}
\begin{proof}
	By the definition (\ref{Fij def}) of \(F_{ij}\), the left-hand side of (\ref{F_{ij}(B_i,B_j) term 2 with epsilon and P}) can be written as
	\begin{align*}
	& (\mathrm{id}\otimes (\epsilon\circ P_{-\lambda_{ij}}))\left(\sum_{k = 0}^{1-a_{ij}} (-1)^{k+1}\begin{bmatrix}
	1-a_{ij}\\k
	\end{bmatrix}_{q_i}\left(B_i\otimes K_i^{-1}+1\otimes F_i+c_i\mathcal{Z}_i\otimes E_iK_i^{-1}\right)^{1-a_{ij}-k}(1\otimes F_j)\right.\\&\left(B_i\otimes K_i^{-1}+1\otimes F_i+c_i\mathcal{Z}_i\otimes E_iK_i^{-1}+c_i\mathcal{W}_{ij}K_j\otimes(E_jE_i-q_i^{a_{ij}}E_iE_j)K_i^{-1}\right)^k
	\Bigg).
	\end{align*}
	In the term \((B_i\otimes K_i^{-1}+1\otimes F_i+c_i\mathcal{Z}_i\otimes E_iK_i^{-1})^{1-a_{ij}-k}\) preceding \(1\otimes F_j\), the term \(c_i\mathcal{W}_{ij}K_j\otimes(E_jE_i-q_i^{a_{ij}}E_iE_j)K_i^{-1}\) does not need to be taken into account. Indeed, the standard ordering of the expansion with respect to this term would consist of terms 
	\begin{equation}
	\label{standard ordering 1}
	E_i^{N_1}E_j^{M}E_i^{N_2}F_i^{N_3}F_jF_i^{N_4}K_i^{N_5},
	\end{equation}
	with \(M\geq 1\). But of course each such term vanishes under \(\epsilon\). In the term \((B_i\otimes K_i^{-1}+1\otimes F_i+c_i\mathcal{Z}_i\otimes E_iK_i^{-1}+c_i\mathcal{W}_{ij}K_j\otimes(E_jE_i-q_i^{a_{ij}}E_iE_j)K_i^{-1})^k\) succeeding \(1\otimes F_j\), it does need to be taken into account. More precisely, in the whole sum we obtain when expanding the \(k\)-th power, each term must contain exactly one factor \(c_i\mathcal{W}_{ij}K_j\otimes(E_jE_i-q_i^{a_{ij}}E_iE_j)K_i^{-1}\), such that we may use the rule \(F_jE_j = E_jF_j - \frac{K_j-K_j^{-1}}{q_j-q_j^{-1}}\) to obtain a non-zero contribution. Indeed, if we were to take more than one such factor, then we would end up with a normal ordering consisting of terms of the form (\ref{standard ordering 1}) with \(M\geq 1\) and
	\[
		E_i^{N_1}E_j^{M}E_i^{N_2}F_i^{N_3}K_i^{N_4}K_j^{N_5}
	\]
	with \(M\geq 1\), which again disappear under \(\epsilon\), whereas if we were to take \(0\) such factors, then we would find 
	\[
	E_i^{N_1}F_i^{N_2}F_jF_i^{N_3}K_i^{N_4},
	\]
	in the normal ordering, which also yields \(0\) under \(\epsilon\) by the presence of \(F_j\). This also explains why we can replace \((B_i\otimes K_i^{-1}+1\otimes F_i+c_i\mathcal{Z}_i\otimes E_iK_i^{-1}+c_i\mathcal{W}_{ij}K_j\otimes(E_jE_i-q_i^{a_{ij}}E_iE_j)K_i^{-1})^k\) by 
	\[
	\sum_{d = 0}^{k-1}(B_i\otimes K_i^{-1}+1\otimes F_i+c_i\mathcal{Z}_i\otimes E_iK_i^{-1})^d(c_i\mathcal{W}_{ij}K_j\otimes(E_jE_i-q_i^{a_{ij}}E_iE_j)K_i^{-1})(B_i\otimes K_i^{-1}+1\otimes F_i+c_i\mathcal{Z}_i\otimes E_iK_i^{-1})^{k-d-1}.
	\]
	The claim now follows upon expanding binarily the powers of \(B_i\otimes K_i^{-1}+1\otimes F_i+c_i\mathcal{Z}_i\otimes E_iK_i^{-1}\), as in the proof of Proposition \ref{prop term 1 expansion with epsilon and P}. Note that this time, we will need a total of \(1-a_{ij}-k+d+(k-d-1) = -a_{ij}\) variables \(\ell_r\). Observe also that we have used Lemma \ref{lemma commutation B_i, Z_i, W_ij} to obtain the factor \(q_i^{a_{ij}\vert\bl\vert_{1;1-a_{ij}-k+d}}\).
\end{proof}

Once more, many of the \(\bss\) in the sum in (\ref{F_{ij}(B_i,B_j) term 2 with epsilon and P}) will not contribute. In analogy to Lemma \ref{lemma restrictions on s}, one can formulate the following result.

\begin{lemma}
	\label{lemma restrictions on s term 2}
	Let \(i\in I\setminus X\) be such that \(\tau(i) = i\) and let \(j\in X\). Let \(\bl\in\{0,1\}^{-a_{ij}}\), \(\bss\in\{0,1\}^{-a_{ij}-\vert\bl\vert}\), \(k\in\{1,\dots,1-a_{ij}\}\) and \(d\in\{0,\dots,k-1\}\). Then one has
	\[
	(\epsilon\circ P_{-\lambda_{ij}})(\mathfrak{r}_{\bl,\bss,k,d}^{(i,j,a_{ij})}) = 0
	\]
	if one of the Conditions \ref{lemma restrictions on s condition a}, \ref{lemma restrictions on s condition b}, \ref{lemma restrictions on s condition c} from Lemma \ref{lemma restrictions on s} is fulfilled, or in case we have
	\begin{enumerate}[(a), start = 4]
		\item\label{lemma restrictions on s condition d} There exists \(p\in\{1-a_{ij}-k+d-\vert\bl\vert_{1;1-a_{ij}-k+d},\dots,-a_{ij}-\vert\bl\vert\}\) such that \(\vert \bss\vert_{1;p} = \tfrac{p}{2}\).
		\item\label{lemma restrictions on s condition e} \(\vert\bss\vert_{1;1-a_{ij}-k+d-\vert\bl\vert_{1;1-a_{ij}-k+d}} = 0\).
	\end{enumerate}
\end{lemma}
\begin{proof}
	As in the proof of Lemma \ref{lemma restrictions on s}, the requirement that \(\rlskd\) must contain an equal number of factors \(F_i\) and \(E_i\) determines the conditions \ref{lemma restrictions on s condition a} and \ref{lemma restrictions on s condition b}. Note that in this case, one comes to the number \(\frac{1-a_{ij}-\vert\bl\vert}{2}\) by considering the \(-a_{ij}-\vert\bl\vert\) factors \(F_i\) or \(E_i\) arising from the \(\mathcal{T}^{i}_{\bl,\bss,r}\) in (\ref{r_l,s,k,d def}), together with the extra factor \(E_i\) in (\ref{r_l,s,k,d def}). The requirement that for each \(p\), the number of factors \(F_i\) must exceed the number of factors \(E_i\) up to position \(p\), determines in this case not only the condition \ref{lemma restrictions on s condition c}, but also the extra conditions \ref{lemma restrictions on s condition d} and \ref{lemma restrictions on s condition e}, again by the presence of \((E_jE_i-q_i^{a_{ij}}E_iE_j)K_i^{-1}\) in (\ref{r_l,s,k,d def}).
\end{proof}

As before, this means that we can determine
\begin{align}
\label{full sum 3}
\begin{split}
\sum_{r = 1}^{-a_{ij}}(1-\ell_r)(1-s_{r-\vert\bl\vert_{1;r}}) & = -a_{ij}-\vert\bl\vert-\vert\bss\vert = \frac{-1-a_{ij}-\vert\bl\vert}{2}, \\
\sum_{r = 1}^{1-a_{ij}-k+d}(1-\ell_r)(1-s_{r-\vert\bl\vert_{1;r}}) & = 1-a_{ij}-k+d-\vert\bl\vert_{1;1-a_{ij}-k+d}-\vert\bss\vert_{1;1-a_{ij}-k+d-\vert\bl\vert_{1;1-a_{ij}-k+d}}, \\
\sum_{r = 2-a_{ij}-k+d}^{-a_{ij}}(1-\ell_r)(1-s_{r-\vert\bl\vert_{1;r}}) & = \frac{-1-a_{ij}-\vert\bl\vert}{2}-\sum_{r = 1}^{1-a_{ij}-k+d}(1-\ell_r)(1-s_{r-\vert\bl\vert_{1;r}}).
\end{split}
\end{align}

Just like in the previous situation, this now leads to a complete description of \(C_{ij}(\bc)\) in Case 2.

\begin{corollary}[Case 2]
	\label{cor F_ij(B_i,B_j) Case 2 part 2 with epsilon and P}
	Let \(i\in I\setminus X\) be such that \(\tau(i) = i\) and let \(j\in X\). Then one has 
	\begin{align}
	\label{F_ij(B_i,B_j) Case 2 with epsilon and P complete}
	\begin{split}
	&F_{ij}(B_i,B_j) = C_{ij}(\bc)\\ = & \sum_{m = 0}^{-1-a_{ij}}\sum_{m' = 0}^{-1-a_{ij}-m}\sum_{t = 0}^{\frac{1-a_{ij}-m-m'}{2}}\rho_{m,m',t}^{(i,j,a_{ij})}\mathcal{Z}_i^tB_i^mB_jB_i^{m'}\mathcal{Z}_i^{\frac{1-a_{ij}-m-m'}{2}-t}\\ 
	&+ \sum_{m = 0}^{-1-a_{ij}}\sum_{t = 0}^{\frac{-1-a_{ij}-m}{2}}\sigma_{m,t}^{(i,j,a_{ij})}\mathcal{Z}_i^t\mathcal{W}_{ij}K_j\mathcal{Z}_i^{\frac{-1-a_{ij}-m}{2}-t}B_i^{m},
	\end{split}
	\end{align}
	with \(\rho_{m,m',t}^{(i,j,a_{ij})}\) as obtained in (\ref{rho_m,m',t def}) and where
	\begin{align}
	\label{sigma_m,t def}
	\begin{split}
	&\sigma_{m,t}^{(i,j,a_{ij})}\\ =\ & (a_{ij}+m)_pc_i^{\frac{1-a_{ij}-m}{2}}\sum_{k = 1}^{1-a_{ij}}\sum_{d = 0}^{k-1}\sum_{m' = 0}^{m}\sum_{\bl\in\mathcal{L}_{m,m',k,d}'}\sum_{\bss\in\mathscr{S}_{m,m',k,t,d}'}(-1)^{k+1}\begin{bmatrix}
	1-a_{ij}\\k
	\end{bmatrix}_{q_i}q_i^{m' a_{ij}}(\epsilon\circ P_{-\lambda_{ij}})(\mathfrak{r}_{\bl,\bss,k,d}^{(i,j,a_{ij})}),
	\end{split}
	\end{align}
	with \(\rlskd\) as in (\ref{r_l,s,k,d def}) and
	\begin{align}
	\label{L' and S' sets Case 2}
	\begin{split}
	\mathcal{L}_{m,m',k,d}' = &\left\{\bl\in\{0,1\}^{-a_{ij}}: \vert\bl\vert = m\ \mathrm{and}\ \vert\bl\vert_{1;1-a_{ij}-k+d} = m'\right\}, \\
	\mathscr{S}_{m,m',k,t,d}' = &\Big\{\bss\in\{0,1\}^{-a_{ij}-m}: \vert\bss\vert = \frac{1-a_{ij}-m}{2},\ \vert\bss\vert_{1;p}\geq \frac{p+\delta^{(p,k,d,m')}}{2}, \forall p\in\{1,\dots,-a_{ij}-m\}\ \mathrm{and}\\&\left. \vert\bss\vert_{1;1-a_{ij}-k-m'+d} = 1-a_{ij}-k-m'-t+d \neq 0\right\},
	\end{split}
	\end{align}
	where
	\[
	\delta^{(p,k,d,m')} = \left\{
	\begin{array}{ll}
	0& \mathrm{for}\ p<1-a_{ij}-k+d-m',\\
	1\quad &\mathrm{for}\ p\geq 1-a_{ij}-k+d-m'.
	\end{array}
	\right.
	\]
\end{corollary}
\begin{proof}
	This follows from Corollary \ref{cor F_ij(B_i,B_j) Case 2 part 1 with epsilon and P}, Proposition \ref{prop term 2 expansion with epsilon and P} and the equations (\ref{full sum 3}) in exactly the same fashion as we have derived Corollaries \ref{cor F_ij(B_i,B_j) Case 1 with epsilon and P} and \ref{cor F_ij(B_i,B_j) Case 2 part 1 with epsilon and P}, i.e.\ upon setting
	\[
	m = \vert\bl\vert, \quad m' = \vert\bl\vert_{1;1-a_{ij}-k+d}, \quad t = 1-a_{ij}-k-m'+d-\vert\bss\vert_{1;1-a_{ij}-k-m'+d}.
	\]
	Again, \(\vert\bl\vert = m\) cannot equal \(-a_{ij}\), since then \(a_{ij}+m\) would be even, which is excluded by Condition \ref{lemma restrictions on s condition a} in Lemma \ref{lemma restrictions on s}. So \(m\) runs from \(0\) to \(-1-a_{ij}\). It follows immediately that \(m'\) runs from \(0\) to \(m\).
	The conditions in Lemma \ref{lemma restrictions on s term 2} determine the definition of \(\mathscr{S}_{m,m',k,t,d}'\).
\end{proof}

The relations we have obtained in Corollaries \ref{cor F_ij(B_i,B_j) Case 1 with epsilon and P} and \ref{cor F_ij(B_i,B_j) Case 2 part 2 with epsilon and P} comply with the explicit calculations performed in \cite{Kolb-2014} and \cite{Balagovic&Kolb-2015} by Balagovi\'{c} and Kolb. They also obtained explicit values for the structure constants for a limited set of possible \(a_{ij}\): they computed \(\rho_{m,m'}^{(i,j,a_{ij})}\) for \(a_{ij}\in\{-1,-2,-3\}\) and \(\rho_{m,m',t}^{(i,j,a_{ij})}\) and \(\sigma_{m,t}^{(i,j,a_{ij})}\) for \(a_{ij}\in\{-1,-2\}\). These values are displayed below.

\renewcommand\figurename{Table}
\begin{figure}[h] 
	\renewcommand\figurename{Table}
	\centering
	\savebox{\imagebox}{\begin{tikzpicture}
		\draw[line width = 1pt] (0,0) -- (9.1,0);
		\draw[line width = 1pt] (1,0.75) -- (1,-1.8);
		\draw (0,0.75) -- (1,0);
		\node at (0.25,0.3){\(m\)};
		\node at (0.75,0.5){\(m'\)};
		\node at (0.5,-0.3){\(0\)};
		\node at (0.5,-0.9){\(1\)};
		\node at (0.5,-1.5){\(2\)};
		\node at (2.25,0.35){\(0\)};
		\node at (5.5,0.35){\(1\)};
		\node at (8.4,0.35){\(2\)};
		\node at (2.25,-0.3) {\(-c_i^2q_i^2[3]_{q_i}^2\)};
		\node at (2.25,-0.9) {\(0\)};
		\node at (2.25,-1.5) {\(c_iq_i(1+[3]_{q_i}^2)\)};
		\node at (5.5,-0.3) {\(0\)};
		\node at (5.5,-0.9) {\(-c_iq_i(q_i^2+3+q_i^{-2})[4]_{q_i}\)};
		\node at (8.4,-0.3) {\(\rho_{2,0}^{(i,j,-3)}\)};
		\draw (0,0.75) rectangle (3.5,-1.8);
		\draw (0,0.75) rectangle (7.6,-1.2);
		\draw (0,0.75) rectangle (9.1,-0.6);
		\end{tikzpicture}}
	\begin{subfigure}[t]{0.12\textwidth}
		\centering
		\raisebox{\dimexpr\ht\imagebox-\height}{
			\begin{tikzpicture}
			\draw[line width = 1pt] (0,0) -- (2,0);
			\draw[line width = 1pt] (1,0.75) -- (1,-0.6);
			\draw (0,0.75) -- (1,0);
			\node at (0.25,0.3){\(m\)};
			\node at (0.75,0.5){\(m'\)};
			\node at (0.5,-0.3){\(0\)};
			\node at (1.5,0.35){\(0\)};
			\node at (1.5,-0.3) {\(c_iq_i\)};
			\draw (0,0.75) rectangle (2,-0.6);
			\end{tikzpicture}}
		\caption{\(a_{ij} = -1\)}
	\end{subfigure}
	\begin{subfigure}[t]{0.35\textwidth}
		\centering\raisebox{\dimexpr\ht\imagebox-\height}{
			\begin{tikzpicture}
			\draw[line width = 1pt] (0,0) -- (5.6,0);
			\draw[line width = 1pt] (1,0.75) -- (1,-1.2);
			\draw (0,0.75) -- (1,0);
			\node at (0.25,0.3){\(m\)};
			\node at (0.75,0.5){\(m'\)};
			\node at (0.5,-0.3){\(0\)};
			\node at (0.5,-0.9){\(1\)};
			\node at (2.4,0.35){\(0\)};
			\node at (4.75,0.35){\(1\)};
			\node at (2.4,-0.3) {\(0\)};
			\node at (2.4,-0.9) {\(c_iq_i(q_i+q_i^{-1})^2 \)};
			\node at (4.75,-0.3) {\(-\rho_{1,0}^{(i,j,-2)}\)};
			\draw (3.9,0.75) -- (3.9,-1.2);
			\draw (5.6,0.75) -- (5.6,-0.6);
			\draw (0,-1.2) -- (3.9,-1.2);
			\draw (0,-0.6) -- (3.9,-0.6);
			\draw (0,0.75) rectangle (3.9,-1.2);
			\draw (3.9,0.75) rectangle (5.6,-0.6);
			\end{tikzpicture}}
		\caption{\(a_{ij} = -2\)}
	\end{subfigure}
	\begin{subfigure}[t]{0.50\textwidth}
		\centering
		\usebox{\imagebox}
		\caption{\(a_{ij} = -3\)}
		\label{Table aij - 2}
	\end{subfigure}
	\caption{Structure constants \({\rho}_{m,m'}^{(i,j,a_{ij})}\) for \(a_{ij}\in\{-1,-2,-3\}\)}
	\label{Table of rho_m,m'}
\end{figure}

\renewcommand\figurename{Table}
\begin{figure}[h] 
	\renewcommand\figurename{Table}
	\centering
	\savebox{\imagebox}{\begin{tikzpicture}
		\draw[line width = 1pt] (-0.9,0) -- (6.2,0);
		\draw[line width = 1pt] (1,0.75) -- (1,-2.2);
		\draw (0,0.75) -- (1,0);
		\node at (-0.25,0.3){\((m,m')\)};
		\node at (0.75,0.5){\(t\)};
		\node at (0,-0.6){\((0,1)\)};
		\node at (0,-1.7){\((1,0)\)};
		\node at (2.4,0.35){\(0\)};
		\node at (5,0.35){\(1\)};
		\node at (2.4,-0.6) {\(-c_iq_i^2\dfrac{q_i^2+2}{q_i-q_i^{-1}}\)};
		\node at (2.4,-1.7) {\(c_iq_i^2\dfrac{[3]_{q_i}}{q_i-q_i^{-1}} \)};
		\node at (5,-0.6) {\(c_i\dfrac{[3]_{q_i}}{q_i-q_i^{-1}}\)};
		\node at (5,-1.7) {\(-c_i\dfrac{q_i^{-2}+2}{q_i-q_i^{-1}} \)};
		\draw (3.9,0.75) -- (3.9,-2.2);
		\draw (6.2,0.75) -- (6.2,-2.2);
		\draw (-0.9,-2.2) -- (6.2,-2.2);
		\draw (-0.9,-1.2) -- (6.2,-1.2);
		\draw (-0.9,0.75) rectangle (6.2,-2.2);
		\end{tikzpicture}}
	\begin{subfigure}[b]{0.45\textwidth}
		\centering
		\raisebox{\dimexpr\ht\imagebox-\height}{
		\begin{tikzpicture}
		\draw[line width = 1pt] (-0.9,0) -- (6.2,0);
		\draw[line width = 1pt] (1,0.75) -- (1,-1.2);
		\draw (0,0.75) -- (1,0);
		\node at (-0.25,0.3){\((m,m')\)};
		\node at (0.75,0.5){\(t\)};
		\node at (0,-0.6){\((0,0)\)};
		\node at (2.4,0.35){\(0\)};
		\node at (5,0.35){\(1\)};
		\node at (2.4,-0.6) {\(\dfrac{c_iq_i^2}{q_i-q_i^{-1}}\)};
		\node at (5,-0.6) {\(\dfrac{-c_i}{q_i-q_i^{-1}}\)};
		\draw (3.9,0.75) -- (3.9,-1.2);
		\draw (-0.9,0.75) rectangle (6.2,-1.2);
		\end{tikzpicture}}
		\caption{\(a_{ij} = -1\)}
	\end{subfigure}
	\begin{subfigure}[b]{0.45\textwidth}
		\centering
		\usebox{\imagebox}
		\caption{\(a_{ij} = -2\)}
	\end{subfigure}
	\caption{Structure constants \({\rho}_{m,m',t}^{(i,j,a_{ij})}\) for \(a_{ij}\in\{-1,-2\}\)}
	\label{Table of rho_m,m',t}
\end{figure}

\begin{figure}[h] 
	\renewcommand\figurename{Table}
	\centering
	\begin{subfigure}[b]{0.25\textwidth}
		\centering
		\begin{tikzpicture}
		\draw[line width = 1pt] (0,0) -- (3.2,0);
		\draw[line width = 1pt] (1,0.75) -- (1,-1.2);
		\draw (0,0.75) -- (1,0);
		\node at (0.25,0.3){\(m\)};
		\node at (0.75,0.5){\(t\)};
		\node at (0.5,-0.6){\(0\)};
		\node at (2.1,0.35){\(0\)};
		\node at (2.1,-0.6) {\(c_i\dfrac{q_i+q_i^{-1}}{q_j-q_j^{-1}}\)};
		\draw (0,0.75) rectangle (3.2,-1.2);
		\end{tikzpicture}
		\caption{\(a_{ij} = -1\)}
	\end{subfigure}
	\begin{subfigure}[b]{0.4\textwidth}
		\centering
		\begin{tikzpicture}
		\draw[line width = 1pt] (0,0) -- (6.4,0);
		\draw[line width = 1pt] (1,0.75) -- (1,-1.2);
		\draw (0,0.75) -- (1,0);
		\node at (0.25,0.3){\(m\)};
		\node at (0.75,0.5){\(t\)};
		\node at (0.5,-0.6){\(1\)};
		\node at (3.7,0.35){\(0\)};
		\node at (3.7,-0.6) {\(-c_iq_i^{-2}[3]_{q_i}\dfrac{(q_i-q_i^{-1})(q_i+q_i^{-1})^2}{q_j-q_j^{-1}}\)};
		\draw (0,0.75) rectangle (6.4,-1.2);
		\end{tikzpicture}
		\caption{\(a_{ij} = -2\)}
	\end{subfigure}
	\caption{Structure constants \({\sigma}_{m,t}^{(i,j,a_{ij})}\) for \(a_{ij}\in\{-1,-2\}\)}
	\label{Table of sigma_m,t}
\end{figure}

The main purpose of this paper will be to find expressions in \(\K(q)\) for the structure constants \(\rho_{m,m'}^{(i,j,a_{ij})}\), \(\rho_{m,m',t}^{(i,j,a_{ij})}\) and \(\sigma_{m,t}^{(i,j,a_{ij})}\), valid without any restrictions on \(a_{ij}\). By Corollaries \ref{cor F_ij(B_i,B_j) Case 1 with epsilon and P}, \ref{cor F_ij(B_i,B_j) Case 2 part 1 with epsilon and P} and \ref{cor F_ij(B_i,B_j) Case 2 part 2 with epsilon and P}, this amounts to deriving how \(\epsilon\circ P_{-\lambda_{ij}}\) acts on \(\plsk\) and \(\rlskd\). This computation will be performed in the next two subsections.

\subsection{Case 1: $\tau(i) = i\in I\setminus X$ and $j\in I\setminus X$}
\label{Subsection Case 1}

Let us now fix \(i\in I\setminus X\) such that \(\tau(i) = i\) and \(j\in I\) distinct from \(i\). A priori, we don't specify whether or not \(j\in X\). Let us also fix \(m,m'\in\N\) such that \(a_{ij}+m+m'\) is odd and \(m+m'\leq -1-a_{ij}\), \(k\in\N\) such that \(m'\leq k\leq 1-a_{ij}-m\), \(t\in\{0,\dots,\frac{1-a_{ij}-m-m'}{2}\}\), \(\bl\in\mathcal{L}_{m,m',k}\) and \(\bss\in\mathscr{S}_{m,m',k,t}\). Note that this automatically implies that \(\bss\in\mathscr{S}_{m,m'}\), by (\ref{L and S sets for Case 1}) and (\ref{S_m,m',k,t set Case 2}). Hence by (\ref{rho_m,m' def}) and (\ref{rho_m,m',t def}) it suffices to compute the action of \(\epsilon\circ P_{-\lambda_{ij}}\) on \(\plsk\), defined in (\ref{p_l,s,k def}), in order to obtain the full polynomial \(C_{ij}(\bc)\) for Case 1, as well as the first of the two parts of this polynomial for Case 2. This computation will now be performed.

Let us introduce the notation \(\widehat{P}^{i}_N\), with \(N\in \Z\), for the projection operator
\begin{equation}
\label{P_i,N def}
\widehat{P}^{i}_{N}: U_q(\g')\to U^+K_i^NS(U^-)
\end{equation}
with respect to the decomposition (\ref{decomposition with P}). Let us also renormalize the element \(E_i\) as
\begin{equation}
\label{E_i tilde def}
\widetilde{E_i} = (q_i-q_i^{-1})E_i.
\end{equation}
Then we can state the following result.

\begin{proposition}
	\label{prop epsilon and other P}
	For \(i,j,m,m',k,\bl\) and \(\bss\) as fixed before, one has
	\begin{align*}
	(\epsilon\circ P_{-\lambda_{ij}})(\plsk) = &\ \left(\frac{q_i^2}{q_i-q_i^{-1}}\right)^{\frac{1-a_{ij}-m-m'}{2}} q_i^{\beta_{\bl,\boldsymbol{s},k}}\left(\epsilon\circ \widehat{P}^{i}_{-\frac{1-a_{ij}-m-m'}{2}}\right)\left(Y_{\bl,\bss}\right),
	\end{align*}
	where
	\begin{align}
	\label{Y_l,s def}
	Y_{\bl,\bss} & = \overrightarrow{\prod_{r = 1}^{1-a_{ij}}}F_i^{(1-\ell_r)s_{r-\vert\bl\vert_{1;r}}}\widetilde{E_i}^{(1-\ell_r)(1-s_{r-\vert\bl\vert_{1;r}})}, \\
	\label{beta_l,s,k def}
	\beta_{\bl,\bss,k} & = -a_{ij}\zeta^{(1-a_{ij}-k)}_{\bl,\bss}-2\sum_{r=1}^{1-a_{ij}}\zeta_{\bl,\bss}^{(r-1)}(\ell_r+(1-\ell_r)(1-s_{r-\vert\bl\vert_{1;r}})),\\
	\label{zeta_l,s,r def}
	\zeta^{(r)}_{\bl,\bss} & = 2\vert\bss\vert_{1;r-\vert\bl\vert_{1;r}}+\vert\bl\vert_{1;r}-r.
	\end{align}
\end{proposition}
\begin{proof}
	As argued in the proof of Lemma \ref{lemma restrictions on s}, the total number of factors \(F_i\) and \(E_i\) in \(\plsk\) must be equal and must yield
	\begin{equation}
	\label{number of factors F_i and E_i}
	\# (\mathrm{factors}\ E_i) = \# (\mathrm{factors}\ F_i) = \vert\bss\vert = \frac{1-a_{ij}-m-m'}{2}.
	\end{equation}
	
	When shifting the factor \(K_j^{-1}\) through the second term between brackets in (\ref{p_l,s,k def}) using (\ref{U_q(g) relations}), we will induce a factor \(q_i^{-a_{ij}x}\), with
	\begin{align}
	\label{deriving x step 1}
	\begin{split}
	x & = \# (\mathrm{factors}\ E_i\ \mathrm{succeeding}\ K_j^{-1})-\# (\mathrm{factors}\ F_i\ \mathrm{succeeding}\ K_j^{-1})\\
	& = \left(\# (\mathrm{factors}\ E_i)-\# (\mathrm{factors}\ E_i\ \mathrm{preceding}\ K_j^{-1})\right)-\left(\# (\mathrm{factors}\ F_i)-\# (\mathrm{factors}\ F_i\ \mathrm{preceding}\ K_j^{-1})\right).
	\end{split}
	\end{align}
	By (\ref{number of factors F_i and E_i}) this is reduced to
	\begin{align}
	\label{deriving x step 2}
	\begin{split}
	x & = \# (\mathrm{factors}\ F_i\ \mathrm{preceding}\ K_j^{-1})-\# (\mathrm{factors}\ E_i\ \mathrm{preceding}\ K_j^{-1}) \\
	& = \# (\mathrm{factors}\ F_i\ \mathrm{preceding}\ K_j^{-1})\\&\phantom{=}-\left[\# (\mathrm{factors}\ \mathrm{preceding}\ K_j^{-1})-\# (\mathrm{factors}\ K_i^{-1}\ \mathrm{preceding}\ K_j^{-1})-\# (\mathrm{factors}\ F_i\ \mathrm{preceding}\ K_j^{-1})\right] \\
	& = \vert\bss\vert_{1;1-a_{ij}-k-\vert\bl\vert_{1;1-a_{ij}-k}} - \left[(1-a_{ij}-k)-\vert\bl\vert_{1;1-a_{ij}-k}-\vert\bss\vert_{1;1-a_{ij}-k-\vert\bl\vert_{1;1-a_{ij}-k}}\right] \\
	& = \zeta_{\bl,\bss}^{(1-a_{ij}-k)}.
	\end{split}
	\end{align}
	So we have
	\[
	\plsk = q_i^{-a_{ij}\zeta_{\bl,\bss}^{(1-a_{ij}-k)}}
	\left(\overrightarrow{\prod_{r = 1}^{1-a_{ij}}}K_i^{-\ell_r}F_i^{(1-\ell_r)s_{r-\vert\bl\vert_{1;r}}}(E_iK_i^{-1})^{(1-\ell_r)(1-s_{r-\vert\bl\vert_{1;r}})}\right)K_j^{-1}.
	\]
	
	We will perform the same shifting process for the factors \(K_i^{-\ell_r}\) with \(\ell_r = 1\). For each such \(r\) this will induce a factor \(q_i^{-2x'}\) with
	\[
	x' = \# (\mathrm{factors}\ E_i\ \mathrm{succeeding}\ K_i^{-\ell_r})-\# (\mathrm{factors}\ F_i\ \mathrm{succeeding}\ K_i^{-\ell_r}).
	\]
	Applying the same reasoning as in (\ref{deriving x step 1})--(\ref{deriving x step 2}), we obtain \(x' = \zeta_{\bl,\bss}^{(r-1)}\), such that one can write
	\begin{align*}
	\plsk = q_i^{-a_{ij}\zeta_{\bl,\bss}^{(1-a_{ij}-k)}-2\sum_{r = 1}^{1-a_{ij}}\zeta_{\bl,\bss}^{(r-1)}\ell_r}
	\left(\overrightarrow{\prod_{r = 1}^{1-a_{ij}}}F_i^{(1-\ell_r)s_{r-\vert\bl\vert_{1;r}}}(E_iK_i^{-1})^{(1-\ell_r)(1-s_{r-\vert\bl\vert_{1;r}})}\right)K_i^{-m-m'}K_j^{-1}.
	\end{align*}
	
	Finally, we will do the same for the \(K_i^{-1}\) occurring in a factor \((E_iK_i^{-1})^{(1-\ell_r)(1-s_{r-\vert\bl\vert_{1;r}})}\) with \(\ell_r = 0\) and \(s_{r-\vert\bl\vert_{1;r}} = 0\). This will give rise to a factor \(q_i^{-2x''}\) with
	\[
	x'' = \# (\mathrm{factors}\ E_i\ \mathrm{succeeding}\ (E_iK_i^{-1})^{(1-\ell_r)(1-s_{r-\vert\bl\vert_{1;r}})})-\# (\mathrm{factors}\ F_i\ \mathrm{succeeding}\ (E_iK_i^{-1})^{(1-\ell_r)(1-s_{r-\vert\bl\vert_{1;r}})}).
	\]
	The same reasoning now shows that \(x''\) yields
	\begin{align*}
	& \# (\mathrm{factors}\ F_i\ \mathrm{preceding}\ (E_iK_i^{-1})^{(1-\ell_r)(1-s_{r-\vert\bl\vert_{1;r}})})-\# (\mathrm{factors}\ E_i\ \mathrm{preceding}\ (E_iK_i^{-1})^{(1-\ell_r)(1-s_{r-\vert\bl\vert_{1;r}})}) - 1 \\
	=\ & \zeta_{\bl,\bss}^{(r-1)}-1,
	\end{align*}
	where the extra \(-1\) comes from the \(E_i\) inside \((E_iK_i^{-1})^{(1-\ell_r)(1-s_{r-\vert\bl\vert_{1;r}})}\). 
	The total power of \(q_i\) we can put in front hence becomes
	\[
	-a_{ij}\zeta_{\bl,\bss}^{(1-a_{ij}-k)} -2\sum_{r = 1}^{1-a_{ij}}\zeta_{\bl,\bss}^{(r-1)}(\ell_r+(1-\ell_r)(1-s_{r-\vert\bl\vert_{1;r}})) +2\sum_{r = 1}^{1-a_{ij}}(1-\ell_r)(1-s_{r-\vert\bl\vert_{1;r}}) = \beta_{\bl,\bss,k}+(1-a_{ij}-m-m'),
	\]
	where we have applied (\ref{full sum}), and with \(\beta_{\bl,\bss,k}\) as in (\ref{beta_l,s,k def}).
	
	Finally, we will perform the renormalization \(\widetilde{E_i} = (q_i-q_i^{-1})E_i\), which, again taking into account the formula (\ref{number of factors F_i and E_i}), leads to
	\begin{equation}
	\label{p_l,k step 3}
	\plsk = \left(\frac{q_i^2}{q_i-q_i^{-1}}\right)^{\frac{1-a_{ij}-m-m'}{2}} q_i^{\beta_{\bl,\bss,k}}
	\left(\overrightarrow{\prod_{r = 1}^{1-a_{ij}}}F_i^{(1-\ell_r)s_{r-\vert\bl\vert_{1;r}}}\widetilde{E_i}^{(1-\ell_r)(1-s_{r-\vert\bl\vert_{1;r}})}\right)K_i^{-\frac{1-a_{ij}+m+m'}{2}}K_j^{-1}.
	\end{equation}
	
	It now follows from (\ref{Coproduct, counit, antipode def}), (\ref{P_lambda def}), (\ref{lambda_ij def}) and (\ref{P_i,N def}) that
	\[
	(\epsilon\circ P_{-\lambda_{ij}})\left[Y_{\bl,\bss}K_i^{-\frac{1-a_{ij}+m+m'}{2}}K_j^{-1} \right] 
	= \left(\epsilon\circ \widehat{P}^{i}_{-\frac{1-a_{ij}-m-m'}{2}}\right)\left[Y_{\bl,\bss} \right].
	\]
	Together with (\ref{p_l,k step 3}), this yields the anticipated result.
\end{proof}

We have now reduced the computation of \((\epsilon\circ P_{-\lambda_{ij}})(\plsk)\) to a simpler problem, namely computing how \(\epsilon\circ \widehat{P}^{i}_{-\frac{1-a_{ij}-m-m'}{2}}\) acts on a product of an equal number of factors \(F_i\) and \(\widetilde{E_i}\), which is balanced in the sense that up to each position in the product, the number of factors \(F_i\) exceeds or equals the number of factors \(\widetilde{E_i}\), as imposed by Condition \ref{lemma restrictions on s condition c} of Lemma \ref{lemma restrictions on s}. This action can be deduced from the following lemma.

We will need the notation \((N)_{q_i^2}\) for the modified \(q_i^2\)-number
\begin{equation}
\label{modified q_i^2-number}
(N)_{q_i^2} = \frac{1-q_i^{2N}}{1-q_i^2}.
\end{equation}
Note that it relates to the ordinary \(q_i\)-number as
\[
(N)_{q_i^2} = q_i^{N-1}[N]_{q_i}.
\]

\begin{lemma}
	\label{lemma epsilon and projection general}
	Let \(M\in\N\) be such that \(M\geq 1\). Let \(Y\in U_q(\g')\) be a product of \(M\) factors \(F_i\) and \(M\) factors \(\widetilde{E_i}\), appearing in any order, but with \(F_i\) as the first factor. Let \(N\in\N\) be maximal such that the first \(N\) factors of \(Y\) are \(F_i\), such that we can write \(Y = F_i^N\widetilde{E_i}X\), for some \(X\in U_q(\g')\). Then we have
	\begin{equation}
	\label{F_i^N and further}
	\left(\epsilon\circ \widehat{P}^{i}_{-M}\right)(Y) = (N)_{q_i^2}q_i^{-2N+2}\left(\epsilon\circ \widehat{P}^{i}_{-\left(M-1\right)}\right)(F_i^{N-1}X).
	\end{equation}
\end{lemma} 
\begin{proof}
	We will prove this by induction on \(N\). Our strategy will be to rewrite \(Y\) in its standard ordering, i.e.\ as a \(\K(q)\)-linear combination of terms of the form \(\widetilde{E_i}^{m_1}F_i^{m_2}K_i^{m_3}\), and then observe that for any \(M'\in\Z\) one has
	\[
	\widehat{P}^{i}_{-M'}\left(\widetilde{E_i}^{m_1}F_i^{m_2}K_i^{m_3}\right) = \left\{\begin{array}{ll}
	\widetilde{E_i}^{m_1}F_i^{m_2}K_i^{m_3} \quad & \mathrm{if}\ m_3-m_2 = -M', \\
	0 & \mathrm{otherwise},
	\end{array} \right.
	\]
	by (\ref{P_i,N def}) and the definition (\ref{Coproduct, counit, antipode def}) of the antipode. Hence, again by (\ref{Coproduct, counit, antipode def}), we have
	\[
	(\epsilon\circ \widehat{P}_{-M'}^{i})\left(\widetilde{E_i}^{m_1}F_i^{m_2}K_i^{m_3}\right) = \left\{\begin{array}{ll}
	1 \quad & \mathrm{if}\ m_1 = m_2 = 0\ \mathrm{and}\ m_3 = -M', \\
	0 & \mathrm{otherwise}.
	\end{array} \right.
	\]
	Otherwise stated, the action of \(\epsilon\circ \widehat{P}^{i}_{-M'}\) on \(Y\) equals the coefficient of \(K_i^{-M'}\) in its standard ordering.
	
	For \(N = 1\), we may apply (\ref{U_q(g) relations 2}) and (\ref{E_i tilde def}) to obtain
	\[
	Y = F_i\widetilde{E_i}X = \widetilde{E_i}F_iX - K_iX + K_i^{-1}X.
	\]
	The first term will have a standard ordering consisting of terms \(\widetilde{E_i}^{m_1}F_i^{m_2}K_i^{m_3}\) with \(m_1\geq 1\), which will all be killed by \(\epsilon\). For the second term, observe that \(X\) contains \(M-1\) factors \(F_i\) and the same number of factors \(\widetilde{E_i}\), since \(N = 1\). Each factor \(F_i\), when taken together with a factor \(\widetilde{E_i}\), can contribute at most one factor \(K_i^{-1}\) by (\ref{U_q(g) relations 2}). Hence the lowest possible power of \(K_i\) occurring in the normal ordering of \(K_iX\) will be \(1-(M-1) = -M+2 > -M\). Hence the second term will not contribute either. For the third term, we have \(K_i^{-1}X = XK_i^{-1}\), since \(X\) contains an equal number of factors \(F_i\) and \(\widetilde{E_i}\). So we have
	\[
	\left(\epsilon\circ \widehat{P}^{i}_{-M}\right)(Y) = \left(\epsilon\circ \widehat{P}^{i}_{-M}\right)(XK_i^{-1}) = \left(\epsilon\circ \widehat{P}^{i}_{-(M-1)}\right)(X),
	\]
	in agreement with (\ref{F_i^N and further}).
	
	Now suppose the claim has been proven for \(N-1\geq 1\), then we have
	\[
	Y = F_i^N\widetilde{E_i}X = F_i^{N-1}\widetilde{E_i}X' - F_i^{N-1}K_iX + F_i^{N-1}K_i^{-1}X
	\]
	where \(X' = F_iX\). As before, the second term will not contribute: the coefficient of \(K_i^{-M}\) in its standard ordering will vanish, as the lowest power of \(K_i\) that can occur will again be \(-M+2\). Consider now the third term in this sum. When shifting \(K_i^{-1}\) through \(X\), we will induce a factor \(q_i^{-2x}\), where
	\begin{align*}
	x & = \# (\mathrm{factors}\ \widetilde{E_i}\ \mathrm{in}\ X) - \# (\mathrm{factors}\ F_i\ \mathrm{in}\ X) \\
	& = \left(\# (\mathrm{factors}\ \widetilde{E_i}\ \mathrm{in}\ Y)-1\right)-\left(\# (\mathrm{factors}\ F_i\ \mathrm{in}\ Y)-N \right)\\
	& = \left(M-1\right)-\left(M-N\right) = N - 1,
	\end{align*}
	such that
	\[
	F_i^{N-1}K_i^{-1}X = q_i^{-2N+2}F_i^{N-1}XK_i^{-1}.
	\]
	So we have
	\[
	\left(\epsilon\circ \widehat{P}^{i}_{-M}\right)(Y) = \left(\epsilon\circ \widehat{P}^{i}_{-M}\right)(F_i^{N-1}\widetilde{E_i}X') + q_i^{-2N+2} \left(\epsilon\circ \widehat{P}^{i}_{-(M-1)}\right)(F_i^{N-1}X).
	\]
	Note that \(F_i^{N-1}\widetilde{E_i}X'\) still contains \(M\) factors \(F_i\) and \(M\) factors \(\widetilde{E_i}\), and has \(F_i\) as its first factor, since \(N-1\geq 1\). Hence we may apply the induction hypothesis to write
	\begin{align*}
	\left(\epsilon\circ \widehat{P}^{i}_{-M}\right)(F_i^{N-1}\widetilde{E_i}X') & = (N-1)_{q_i^2}q_i^{-2N+4}\left(\epsilon\circ \widehat{P}^{i}_{-(M-1)}\right)(F_i^{N-2}X').
	\end{align*}
	The statement now follows from \(F_i^{N-2}X'=F_i^{N-1}X\) and upon observing that 
	\[
	(N-1)_{q_i^2}q_i^{-2N+4} + q_i^{-2N+2} = (N)_{q_i^2}q_i^{-2N+2}.
	\] 
\end{proof}

Let once more \(Y,X\in U_q(\g')\) and \(M,N\in\N\) be as in the statement of Lemma \ref{lemma epsilon and projection general}. As already observed, the element \(F_i^{N-1}X\) is again of the type described in Lemma \ref{lemma epsilon and projection general}: it is a product of \(M-1\) factors \(F_i\) and the same number of factors \(\widetilde{E_i}\), and has \(F_i\) as its first factor, provided \(X\) has \(F_i\) as its first factor or \(N-1\geq 1\). If \(N'\geq N-1\) is the maximal number such that the first \(N'\) factors of \(F_i^{N-1}X\) are \(F_i\), then we may write \(F_i^{N-1}X = F_i^{N'}\widetilde{E_i}X'\), for some \(X'\in U_q(\g')\). Consequently, Lemma \ref{lemma epsilon and projection general} asserts
\[
\left(\epsilon\circ \widehat{P}^{i}_{-(M-1)} \right)(F_i^{N-1}X) = (N')_{q_i^2}q_i^{-2N'+2}\left(\epsilon\circ \widehat{P}^{i}_{-(M-2)}\right)(F_i^{N'-1}X'),
\]
and thus
\[
\left(\epsilon\circ \widehat{P}^{i}_{-M} \right)(Y) = (N)_{q_i^2}q_i^{-2N+2}(N')_{q_i^2}q_i^{-2N'+2}\left(\epsilon\circ \widehat{P}^{i}_{-(M-2)}\right)(F_i^{N'-1}X').
\]

This process will only terminate if at some position \(p\) in the product, the number of factors \(\widetilde{E_i}\) preceding \(p\) exceeds the number of factors \(F_i\) preceding \(p\). In that case, we would at some point be left with \(N' = 1\) and a corresponding \(X'\) starting with \(\widetilde{E_i}\) instead of \(F_i\). 

Let us now assume that this is not the case, i.e.\ up to each position \(p\) in \(Y\), the number of factors \(F_i\) preceding \(p\) exceeds or equals the number of factors \(\widetilde{E_i}\) preceding \(p\). Then this process of applying Lemma \ref{lemma epsilon and projection general} consecutively will continue until we have applied it \(M\) times and we have reached \(N' = 1\) and \(X' = 1\), and of course \((\epsilon\circ \widehat{P}^{i}_0)(1) = 1\). Each factor \(\widetilde{E_i}\) can now be assigned a level, which is the exponent \(N\) of \(F_i\) that will occur in front of \(\widetilde{E_i}\) at the moment this factor is cancelled when applying the formula (\ref{F_i^N and further}) in this consecutive process. Then our reasoning in fact asserts
\[
(\epsilon\circ \widehat{P}^{i}_{-M})(Y) = \prod_{\mathrm{factors}\ \widetilde{E_i}}\left(\mathrm{level}(\widetilde{E_i})\right)_{q_i^2}q_i^{-2\,\mathrm{level}(\widetilde{E_i})+2},
\]
where the product runs over all factors \(\widetilde{E_i}\) in \(Y\). Now note that each application of the formula (\ref{F_i^N and further}) cancels one factor \(F_i\) and one factor \(\widetilde{E_i}\), hence each \(F_i\) is in fact coupled to exactly one factor \(\widetilde{E_i}\). Thus instead of running over all \(\widetilde{E_i}\) in \(Y\), we might as well run over all factors \(F_i\) in \(Y\) and assign to each \(F_i\) a level, which equals the level of the \(\widetilde{E_i}\) to which it is coupled. We find
\begin{equation}
\label{informal formula}
(\epsilon\circ \widehat{P}^{i}_{-M})(Y) = \prod_{\mathrm{factors}\ F_i}\left(\mathrm{level}(F_i)\right)_{q_i^2}q_i^{-2\,\mathrm{level}(F_i)+2}.
\end{equation}

Now say the element \(Y\) contains a factor \(\widetilde{E_i}\) at position \(p\) in the product, which, in the process above, is coupled to a factor \(F_i\) at position \(r\), with of course \(r<p\).
From the definition, it follows that the level of the \(\widetilde{E_i}\) at position \(p\) is the total number of factors \(F_i\) preceding it, minus the number of factors \(\widetilde{E_i}\) preceding it, again since each application of (\ref{F_i^N and further}) kills one \(F_i\) and one \(\widetilde{E_i}\). So
\begin{align}
\label{level of F_i step 1}
\begin{split}
&\mathrm{level}(F_i\ \mathrm{at\ position}\ r) = \mathrm{level}(\widetilde{E_i}\ \mathrm{at\ position}\ p)\\ =\ &\#(\mathrm{factors}\ F_i\ \mathrm{preceding}\ p) - \#(\mathrm{factors}\ \widetilde{E_i}\ \mathrm{preceding}\ p)\\
=\ & \#(\mathrm{factors}\ F_i\ \mathrm{preceding}\ r) + 1 + \#(\mathrm{factors}\ F_i\ \mathrm{between}\ r+1\ \mathrm{and}\ p-1)\\&-\left(\#(\mathrm{factors}\ \widetilde{E_i}\ \mathrm{preceding}\ r) + \#(\mathrm{factors}\ \widetilde{E_i}\ \mathrm{between}\ r+1\ \mathrm{and}\ p-1) \right),
\end{split}
\end{align}
where the \(+1\) comes from the \(F_i\) at position \(r\) itself. Moreover, we have that
\begin{equation}
\label{level of F_i step 2}
\#(\mathrm{factors}\ F_i\ \mathrm{between}\ r+1\ \mathrm{and}\ p-1) = \#(\mathrm{factors}\ \widetilde{E_i}\ \mathrm{between}\ r+1\ \mathrm{and}\ p-1).
\end{equation}
Indeed, suppose not, then after coupling all possible \(\widetilde{E_i}\) between positions \(r+1\) and \(p-1\) with an \(F_i\), there would either still be \(F_i\)'s left, hence position \(p\) would be coupled to some position \(r' > r\), or else there would still be \(\widetilde{E_i}\) left, so position \(r\) would be coupled to \(p'<p\). Inserting (\ref{level of F_i step 2}) into (\ref{level of F_i step 1}), we obtain
\begin{equation}
\label{level of F_i final}
\mathrm{level}(F_i\ \mathrm{at\ position}\ r) = \#(\mathrm{factors}\ F_i\ \mathrm{preceding}\ r) + 1 -\#(\mathrm{factors}\ \widetilde{E_i}\ \mathrm{preceding}\ r).
\end{equation}

Let us now return to the statement of Proposition \ref{prop epsilon and other P}. The element
\[
Y_{\bl,\bss} = \overrightarrow{\prod_{r = 1}^{1-a_{ij}}}F_i^{(1-\ell_r)s_{r-\vert\bl\vert_{1;r}}}\widetilde{E_i}^{(1-\ell_r)(1-s_{r-\vert\bl\vert_{1;r}})}
\]
is a product of an equal number of factors \(F_i\) and \(\widetilde{E_i}\), namely
\begin{align*}
\# (\mathrm{factors}\ F_i) & = \vert\bss\vert = \frac{1-a_{ij}-m-m'}{2}, \\
\# (\mathrm{factors}\ \widetilde{E_i}) & = \sum_{r = 1}^{1-a_{ij}}(1-\ell_r)(1-s_{r-\vert\bl\vert_{1;r}}) = \frac{1-a_{ij}-m-m'}{2},
\end{align*}
where we have applied (\ref{full sum}). Moreover, at each position \(p\) in \(Y_{\bl,\bss}\), the number of factors \(F_i\) preceding \(p\), i.e.\ \(\vert\bss\vert_{1;p}\), exceeds or equals the number of factors \(\widetilde{E_i}\) preceding \(p\), i.e.\ \(p-\vert\bss\vert_{1;p}\), by Condition \ref{lemma restrictions on s condition c} of Lemma \ref{lemma restrictions on s}. Hence the formula (\ref{informal formula}) is applicable. 

Running over the factors \(F_i\) in \(Y_{\bl,\bss}\) amounts to running over \(r\in\{1,\dots,1-a_{ij}\}\) and checking for each \(r\) whether the element at position \(r\) is \(F_i\), i.e.\ whether \(\ell_r = 0\) and \(s_{r-\vert\bl\vert_{1;r}}=1\). Thus we have
\begin{align}
\label{levels to insert}
\begin{split}
&(\epsilon\circ \widehat{P}^{i}_{-\frac{1-a_{ij}-m-m'}{2}})(Y_{\bl,\bss}) = \prod_{r=1}^{1-a_{ij}}\left[\left(\mathrm{level}(F_i\ \mathrm{at\ position\ }r)\right)_{q_i^2}q_i^{-2\,\mathrm{level}(F_i\ \mathrm{at\ position\ }r)+2}\right]^{(1-\ell_r)s_{r-\vert\bl\vert_{1;r}}} \\
=\ & q_i^{-2\sum_{r=1}^{1-a_{ij}}(\mathrm{level}(F_i\ \mathrm{at\ position}\ r)-1)(1-\ell_r)s_{r-\vert\bl\vert_{1;r}}}\prod_{r=1}^{1-a_{ij}}\left(\left(\mathrm{level}(F_i\ \mathrm{at\ position\ }r)\right)_{q_i^2}\right)^{(1-\ell_r)s_{r-\vert\bl\vert_{1;r}}}.
\end{split}
\end{align}
Applying the formula (\ref{level of F_i final}), we find
\[
\mathrm{level}(F_i\ \mathrm{at\ position}\ r)
= \vert\bss\vert_{1;r-1-\vert\bl\vert_{1;r-1}}+1-\left(r-1-\vert\bl\vert_{1;r-1}-\vert\bss\vert_{1;r-1-\vert\bl\vert_{1;r-1}} \right) = \zeta_{\bl,\bss}^{(r-1)}+1.
\]
Combining this with (\ref{levels to insert}), we immediately obtain the following result.

\begin{corollary}
	\label{cor epsilon and other P solved}
	For \(i,j,m,m',k,\bl\) and \(\bss\) as fixed before, one has
	\[
	\left(\epsilon\circ \widehat{P}^{i}_{-\frac{1-a_{ij}-m-m'}{2}}\right)\left(\overrightarrow{\prod_{r = 1}^{1-a_{ij}}}F_i^{(1-\ell_r)s_{r-\vert\bl\vert_{1;r}}}\widetilde{E_i}^{(1-\ell_r)(1-s_{r-\vert\bl\vert_{1;r}})}\right) = q_i^{\gamma_{\bl,\bss,k}}\prod_{r = 1}^{1-a_{ij}}\left(\left(\zeta_{\bl,\bss}^{(r-1)}+1\right)_{q_i^2} \right)^{(1-\ell_r)s_{r-\vert\bl\vert_{1;r}}},
	\]
	where
	\[
	\gamma_{\bl,\bss,k} = -2\sum_{r = 1}^{1-a_{ij}}\zeta_{\bl,\bss}^{(r-1)}(1-\ell_r)s_{r-\vert\bl\vert_{1;r}},
	\]
	with \(\zeta_{\bl,\bss}^{(r)}\) as in (\ref{zeta_l,s,r def}).
\end{corollary}

Alternatively, one can also iterate over the factors \(\Ei\) in \(Y_{\bl,\bss}\) rather than the factors \(F_i\), as initially established. Since (\ref{level of F_i step 1}) implies
\[
\mathrm{level}(\Ei\ \mathrm{at}\ \mathrm{position}\ p) = \zeta_{\bl,\bss}^{(p-1)},
\]
we also have
\begin{equation}
\label{product iteration for Case 2}
\left(\epsilon\circ \widehat{P}^{i}_{-\frac{1-a_{ij}-m-m'}{2}}\right)(Y_{\bl,\bss}) = \prod_{p = 1}^{1-a_{ij}}\left[(\zeta_{\bl,\bss}^{(p-1)})_{q_i^2}q_i^{-2\zeta_{\bl,\bss}^{(p-1)}+2} \right]^{(1-\ell_p)(1-s_{p-\vert\bl\vert_{1;p}})}.
\end{equation}
This formula will be of use in Subsection \ref{Subsection Case 2}.

Corollaries \ref{cor F_ij(B_i,B_j) Case 1 with epsilon and P} and \ref{cor epsilon and other P solved} and Proposition \ref{prop epsilon and other P} now lead to an explicit expression for the structure constants \(\rho_{m,m'}^{(i,j,a_{ij})}\) for Case 1.

\begin{theorem}[Case 1]
	\label{theorem F_ij(B_i,B_j) for Case 1}
	For any \(i\in I\setminus X\) such that \(\tau(i) = i\) and any \(j\in I\setminus X\) distinct from \(i\), one has
	\begin{equation}
	\label{F_ij(B_i,B_j) for Case 1 final}
	F_{ij}(B_i,B_j) = C_{ij}(\bc) = \sum_{m=0}^{-1-a_{ij}}\sum_{m' = 0}^{-1-a_{ij}-m}\rho_{m,m'}^{(i,j,a_{ij})} \mathcal{Z}_i^{\frac{1-a_{ij}-m-m'}{2}}B_i^mB_jB_i^{m'},
	\end{equation}
	where the structure constants are given by
	\begin{align}
	\label{rho_m,m' final}
	\begin{split}
	\rho_{m,m'}^{(i,j,a_{ij})} =\ & (a_{ij}+m+m')_p \left(\frac{c_iq_i^2}{q_i-q_i^{-1}}\right)^{\frac{1-a_{ij}-m-m'}{2}}\\&\sum_{k = m'}^{1-a_{ij}-m}\sum_{\bl\in\mathcal{L}_{m,m',k}}\sum_{\bss\in\mathscr{S}_{m,m'}}(-1)^{k+1}\begin{bmatrix}
	1-a_{ij} \\k
	\end{bmatrix}_{q_i} q_i^{\theta_{\bl,\bss,k}}\prod_{r = 1}^{1-a_{ij}}\left( \left(\zeta_{\bl,\bss}^{(r-1)}+1\right)_{q_i^2} \right)^{(1-\ell_r)s_{r-\vert\bl\vert_{1;r}}},
	\end{split}
	\end{align}
	where
	\[
	\theta_{\bl,\bss,k} = -a_{ij}\zeta_{\bl,\bss}^{(1-a_{ij}-k)}-2\sum_{r = 1}^{1-a_{ij}}\zeta_{\bl,\bss}^{(r-1)},
	\]
	with \(\mathcal{L}_{m,m',k}\) and \(\mathscr{S}_{m,m'}\) as in (\ref{L and S sets for Case 1}) and \(\zeta_{\bl,\bss}^{(r)}\) as in (\ref{zeta_l,s,r def}).
\end{theorem}
\begin{proof}
	This follows upon combining Corollary \ref{cor F_ij(B_i,B_j) Case 1 with epsilon and P}, Proposition \ref{prop epsilon and other P} and Corollary \ref{cor epsilon and other P solved}. Note that for each \(k\), \(\bl\) and \(\bss\), the exponent of \(q_i\) becomes 
	\[
	\beta_{\bl,\bss,k}+\gamma_{\bl,\bss,k} = -a_{ij}\zeta_{\bl,\bss}^{(1-a_{ij}-k)}-2\sum_{r = 1}^{1-a_{ij}}\zeta_{\bl,\bss}^{(r-1)}\left(\ell_r+(1-\ell_r)(1-s_{r-\vert\bl\vert_{1;r}})+(1-\ell_r)s_{r-\vert\bl\vert_{1;r}} \right) = \theta_{\bl,\bss,k}.
	\]
\end{proof}

Similarly, this leads us to an explicit expression for the structure constants \(\rho_{m,m',t}^{(i,j,a_{ij})}\) for the first part of \(C_{ij}(\bc)\) for Case 2.

\begin{corollary}
	\label{cor F_ij(B_i,B_j) Case 2 semi-final}
	Let \(i\in I\setminus X\) be such that \(\tau(i) = i\), \(j\in X\) and \(m,\) \(m'\) and \(t\) as fixed before. Then the structure constants \(\rho_{m,m',t}^{(i,j,a_{ij})}\) are obtained from the expression (\ref{rho_m,m' final}) upon replacing \(\mathscr{S}_{m,m'}\) by \(\mathscr{S}_{m,m',k,t}\) defined in (\ref{S_m,m',k,t set Case 2}).
\end{corollary}
\begin{proof}
	This follows upon comparing (\ref{rho_m,m' def}) with (\ref{rho_m,m',t def}). 
\end{proof}

It can readily be checked that these expressions comply with the values computed in \cite{Kolb-2014} and \cite{Balagovic&Kolb-2015}, as displayed in Tables \ref{Table of rho_m,m'} and \ref{Table of rho_m,m',t}.

\subsection{Case 2: $\tau(i) = i\in I\setminus X$ and $j\in X$}
\label{Subsection Case 2}

Consequently, we will obtain the second part of the polynomial \(C_{ij}(\bc)\) for Case 2, as described by the last line of (\ref{F_ij(B_i,B_j) Case 2 with epsilon and P complete}). To this end, let us fix \(i\in I\setminus X\) such that \(\tau(i) = i\), \(j\in X\), \(m\in\{0,\dots,-1-a_{ij}\}\), \(t\in\{0,\dots,\frac{-1-a_{ij}-m}{2}\}\), \(k\in\{1,\dots,1-a_{ij}\}\), \(d\in\{0,\dots,k-1\}\), \(m'\in\{0,\dots,m\}\), \(\bl\in\mathcal{L}_{m,m',k,d}'\) and \(\bss\in\mathscr{S}_{m,m',k,t,d}'\). By (\ref{sigma_m,t def}), the calculation of the structure constants \(\sigma_{m,t}^{(i,j,a_{ij})}\) comes down to computing the action of \(\epsilon\circ P_{-\lambda_{ij}}\) on \(\rlskd\), defined in (\ref{r_l,s,k,d def}). This will be the subject of the present subsection.

As a first step, we will again shift all factors \(K_i^{-1}\) in \(\rlskd\) to the back, as we have done for \(\plsk\) in Proposition \ref{prop epsilon and other P}. Recall the notation \(\widetilde{E_i} = (q_i-q_i^{-1})E_i\) and let us write, as an extension of (\ref{P_i,N def}), 
\begin{equation}
\label{P_i,j,N def}
\widehat{P}^{i,j}_{N,M}: \Uqgp\to U^+K_i^NK_j^MS(U^-),
\end{equation}
for the projection operator with respect to the decomposition (\ref{decomposition with P}), where \(M,N\in\Z\). 

\begin{proposition}
	\label{prop epsilon and other P Case 2}
	For \(i,j,m,t,k,d,m',\bl\) and \(\bs\) as fixed before, we have
	\[
	(\epsilon\circ P_{-\lambda_{ij}})(\rlskd) = \frac{q_i^{\eta_{\bl,\bss,k,d,t,m'}}}{(q_i-q_i^{-1})^{\frac{1-a_{ij}-m}{2}}(q_j-q_j^{-1})}\left(\epsilon\circ \widehat{P}^{i,j}_{-\frac{1-a_{ij}-m}{2},-1}\right)\left(Y_{\bl,\bss,k,d}^{(0)}-q_i^{a_{ij}}Y_{\bl,\bss,k,d}^{(1)}\right),
	\]
	where
	\begin{align}
	\label{Y_l,s,k,d 0 def}
	Y_{\bl,\bss,k,d}^{(0)} =\ & \left(\overrightarrow{\prod_{r = 1}^{1-a_{ij}-k}}\mathcal{V}^{i}_{\bl,\bss,r}\right)F_j\left(\overrightarrow{\prod_{r = 2-a_{ij}-k}^{1-a_{ij}-k+d}}\mathcal{V}^{i}_{\bl,\bss,r}\right)\widetilde{E_j}\widetilde{E_i}\left(\overrightarrow{\prod_{r = 2-a_{ij}-k+d}^{-a_{ij}}}\mathcal{V}^{i}_{\bl,\bss,r}\right),\\
	\label{Y_l,s,k,d 1 def}
	Y_{\bl,\bss,k,d}^{(1)} =\ & \left(\overrightarrow{\prod_{r = 1}^{1-a_{ij}-k}}\mathcal{V}^{i}_{\bl,\bss,r}\right)F_j\left(\overrightarrow{\prod_{r = 2-a_{ij}-k}^{1-a_{ij}-k+d}}\mathcal{V}^{i}_{\bl,\bss,r}\right)\widetilde{E_i}\widetilde{E_j}\left(\overrightarrow{\prod_{r = 2-a_{ij}-k+d}^{-a_{ij}}}\mathcal{V}^{i}_{\bl,\bss,r}\right),
	\end{align}
	\begin{align}
	\label{V_i,l,s,r def}
	\mathcal{V}^{i}_{\bl,\bss,r} =\ & F_i^{(1-\ell_r)s_{r-\vert\bl\vert_{1;r}}}\widetilde{E_i}^{(1-\ell_r)(1-s_{r-\vert\bl\vert_{1;r}})}, \\
	\label{eta_l,s,k,d,t,m def}
	\begin{split}
	\eta_{\bl,\bss,k,d,t,m'} =\ & -2\zeta_{\bl,\bss}^{(1-a_{ij}-k+d)} - 2\sum_{r = 1}^{-a_{ij}}\zeta_{\bl,\bss}^{(r-1)}\left(\ell_r + (1-\ell_r)(1-s_{r-\vert\bl\vert_{1;r}}) \right)\\&
	-a_{ij}(1+a_{ij}+k+m'+t+\vert\bss\vert_{1;1-a_{ij}-k-\vert\bl\vert_{1;1-a_{ij}-k}})-2(m'+t),
	\end{split}
	\end{align}
	with \(\zeta_{\bl,\bss}^{(r)}\) as in (\ref{zeta_l,s,r def}).
\end{proposition}
\begin{proof}
	Let us start by shifting the factor \(K_i^{-1}\) arising from \((E_jE_i-q_i^{a_{ij}}E_iE_j)K_i^{-1}\) in (\ref{r_l,s,k,d def}) to the back. Reasoning as in the proof of Proposition \ref{prop epsilon and other P}, this induces a factor \(q_i^{-2x}\), where
	\begin{align*}
	x &= \#(\mathrm{factors}\ F_i\ \mathrm{preceding}\ K_i^{-1}) - \#(\mathrm{factors}\ E_i\ \mathrm{preceding}\ K_i^{-1})\\& = \vert\bss\vert_{1;1-a_{ij}-k+d-\vert\bl\vert_{1;1-a_{ij}-k+d}}-(1-a_{ij}-k+d-\vert\bl\vert_{1;1-a_{ij}-k+d}-\vert\bss\vert_{1;1-a_{ij}-k+d-\vert\bl\vert_{1;1-a_{ij}-k+d}})-1\\
	& = \zeta_{\bl,\bss}^{(1-a_{ij}-k+d)}-1,
	\end{align*}
	where the \(-1\) comes from the factor \(E_i\) in \((E_jE_i-q_i^{a_{ij}}E_iE_j)K_i^{-1}\).
	
	Now let us perform the same shifting for the factors \(K_i^{-\ell_r}\) with \(\ell_r = 1\), which leads to a factor \(q_i^{x'_r}\), where this time \(x'_r\) depends on \(r\). In general, we have
	\begin{align*}
	x'_r =\ & -2\left(\#(\mathrm{factors}\ F_i\ \mathrm{preceding}\ K_i^{-\ell_r}) - \#(\mathrm{factors}\ E_i\ \mathrm{preceding}\ K_i^{-\ell_r})\right)\\& -a_{ij}\left( \#(\mathrm{factors}\ F_j\ \mathrm{preceding}\ K_i^{-\ell_r}) - \#(\mathrm{factors}\ E_j\ \mathrm{preceding}\ K_i^{-\ell_r})\right),
	\end{align*}
	again since \(\rlskd\) contains an equal number of factors \(F_i\) and \(E_i\), and precisely 1 factor \(F_j\) and 1 factor \(E_j\).	For \(r\in\{1,\dots,1-a_{ij}-k\}\) we have
	\[
	x'_r = -2\left( \vert\bss\vert_{r-1-\vert\bl\vert_{1;r-1}}-\left(r-1-\vert\bl\vert_{1;r-1}-\vert\bss\vert_{1;r-1-\vert\bl\vert_{1;r-1}}\right)\right) = -2\zeta_{\bl,\bss}^{(r-1)}.
	\]
	For \(r\in\{2-a_{ij}-k,\dots,1-a_{ij}-k+d\}\) on the other hand, by the same reasoning this becomes
	\[
	x'_r = -2\zeta_{\bl,\bss}^{(r-1)}-a_{ij},
	\]
	whereas for \(r\in\{2-a_{ij}-k+d,\dots,-a_{ij}\}\) one has
	\[
	x'_r = -2(\zeta_{\bl,\bss}^{(r-1)}-1),
	\]
	where the \(-1\) arises from the factor \(E_i\) in \(E_jE_i-q_i^{a_{ij}}E_iE_j\).
	
	Finally, this shifting process for the factor \(K_i^{-1}\) in \((E_iK_i^{-1})^{(1-\ell_r)(1-s_{r-\vert\bl\vert_{1;r}})}\) with \(\ell_r = 0\) and \(s_{r-\vert\bl\vert_{1;r}}=0\) induces a factor \(q_i^{x_r''}\), with, reasoning as above,
	\[
	x_r'' = \left\{
	\begin{array}{ll}
	-2(\zeta_{\bl,\bss}^{(r-1)}-1) & \mathrm{for}\ r\in\{1,\dots,1-a_{ij}-k\}, \\
	-2(\zeta_{\bl,\bss}^{(r-1)}-1)-a_{ij}\quad & \mathrm{for}\ r\in\{2-a_{ij}-k,\dots,1-a_{ij}-k+d\} ,\\
 	-2(\zeta_{\bl,\bss}^{(r-1)}-2) & \mathrm{for}\ r\in\{2-a_{ij}-k+d,\dots,-a_{ij}\}.
	\end{array}
	\right.
	\] 

	In total, this shifting gives rise to a factor \(q_i^{\eta}\), with
	\begin{align*}
	\eta =\ & -2\zeta_{\bl,\bss}^{(1-a_{ij}-k+d)}+2 - 2\sum_{r = 1}^{-a_{ij}}\zeta_{\bl,\bss}^{(r-1)}\left(\ell_r + (1-\ell_r)(1-s_{r-\vert\bl\vert_{1;r}}) \right)-a_{ij}\sum_{r=2-a_{ij}-k}^{1-a_{ij}-k+d}\left(\ell_r+(1-\ell_r)(1-s_{r-\vert\bl\vert_{1;r}})\right)\\&+2\sum_{r = 2-a_{ij}-k+d}^{-a_{ij}}\left(\ell_r + (1-\ell_r)(1-s_{r-\vert\bl\vert_{1;r}}) \right) +2 \sum_{r = 1}^{-a_{ij}}(1-\ell_r)(1-s_{r-\vert\bl\vert_{1;r}}) \\
	=\ & -2\zeta_{\bl,\bss}^{(1-a_{ij}-k+d)}+2- 2\sum_{r = 1}^{-a_{ij}}\zeta_{\bl,\bss}^{(r-1)}\left(\ell_r + (1-\ell_r)(1-s_{r-\vert\bl\vert_{1;r}}) \right)\\&-a_{ij}\left(m'+t-(1-a_{ij}-k)+\vert\bss\vert_{1;1-a_{ij}-k-\vert\bl\vert_{1;1-a_{ij}-k}}\right)+2\left(\frac{1-a_{ij}+m}{2}-m'-t-1\right)+(-1-a_{ij}-m),
	\end{align*}
	in agreement with (\ref{eta_l,s,k,d,t,m def}), where we have used (\ref{full sum 3}) and the definition (\ref{L' and S' sets Case 2}) of \(\mathscr{S}_{m,m',k,t,d}'\).
	
	Finally, the renormalization (\ref{E_i tilde def}) gives rise to a factor \((q_i-q_i^{-1})^{-\frac{1-a_{ij}-m}{2}}(q_j-q_j^{-1})^{-1}\), since by (\ref{full sum 3}) we have
	\[
	\#(\mathrm{factors}\ E_i) = \#(\mathrm{factors}\ F_i) = \vert\bss\vert = \frac{1-a_{ij}-m}{2}
	\]
	and of course \(\#(\mathrm{factors}\ E_j) = \#(\mathrm{factors}\ F_j) = 1\). So we find
	\[
	\rlskd = \frac{q_i^{\eta_{\bl,\bss,k,d,t,m'}}}{(q_i-q_i^{-1})^{\frac{1-a_{ij}-m}{2}}(q_j-q_j^{-1})}\left(Y_{\bl,\bss,k,d}^{(0)}-q_i^{a_{ij}}Y_{\bl,\bss,k,d}^{(1)}\right) K_i^{-\frac{1-a_{ij}+m}{2}},
	\]
	which yields the claim by (\ref{Coproduct, counit, antipode def}), (\ref{P_lambda def}), (\ref{lambda_ij def}) and (\ref{P_i,j,N def}).
\end{proof}

We have hence reduced our problem to computing how \(\epsilon\circ \widehat{P}^{i,j}_{-\frac{1-a_{ij}-m}{2},-1}\) acts on \(Y_{\bl,\bss,k,d}^{(0)}-q_i^{a_{ij}}Y_{\bl,\bss,k,d}^{(1)}\). Each of the latter terms is a product of an equal number of factors \(F_i\) and \(\widetilde{E_i}\) and precisely one factor \(F_j\) and \(\widetilde{E_j}\), which is balanced in the sense that up to each position in the product, the number of factors \(F_i\) exceeds or equals the number of factors \(\widetilde{E_i}\), and that the factor \(F_j\) preceeds the factor \(\widetilde{E_j}\). The presence of \(F_j\) and \(\widetilde{E_j}\) now complicates matters substantially in comparison to the situation in Case 1, because \(F_i\) does not commute with \(F_j\) and similarly for \(\Ei\) and \(\Ej\). We will need to derive an analog of Lemma \ref{lemma epsilon and projection general} which takes into account the presence of these factors.

Recall the notation \((N)_{q_i^2}\) for the modified \(q_i^2\)-number (\ref{modified q_i^2-number}) and let us also define
\begin{align}
\label{alpha_N def}
\alpha_N &= (N)_{q_i^2}q_i^{-2N+2}, \\
\label{gamma_M,N def}
\gamma_{M,N} &= (N-M)_{q_i^2}q_i^{-a_{ij}-2N+2}, 
\end{align}
for \(M,N\in\N\). Write also \(\alpha_{N} = 0\) for \(N < 0\). Then one can prove the following result.

\begin{lemma}
	\label{lemma epsilon and projection general Case 2}
	Let \(M\in\N\) be such that \(M\geq 1\). Let \(Y\in\Uqgp\) be a product of \(M\) factors \(F_i\), \(M\) factors \(\widetilde{E_i}\), 1 factor \(F_j\) and 1 factor \(\widetilde{E_j}\), appearing in any order but with \(F_i\) as its first \(N_0\) factors, for some \(N_0\in\N\), followed by a factor \(F_j\). Let \(N_1\in\N\) be maximal such that the first \(N_1\) factors of \(Y\) succeeding \(F_j\) are \(F_i\), so that we can write \(Y = F_i^{N_0}F_jF_i^{N_1}\widetilde{E_i}X\), for some \(X\in\Uqgp\). Then we have
	\[
		\left(\epsilon\circ \widehat{P}^{i,j}_{-M,-1}\right)(Y) =\alpha_{N_0}\left(\epsilon\circ \widehat{P}^{i,j}_{-(M-1),-1}\right)(F_i^{N_0-1}F_jF_i^{N_1}X) + \gamma_{N_0,N_0+N_1}\left(\epsilon\circ \widehat{P}^{i,j}_{-(M-1),-1}\right)(F_i^{N_0}F_jF_i^{N_1-1}X).
	\]
\end{lemma}
\begin{proof}
	We prove this by induction on \(N_1\). As before, our strategy will be to write \(Y\) in its standard ordering, i.e.\ as a \(\K(q)\)-linear combination of \(\Ei^{m_1}\Ej^{\delta}\Ei^{m_2}F_i^{m_3}F_j^{\delta}F_i^{m_4}K_i^{m_5}K_j^{\delta'}\), with \(m_1,\dots,m_4\in\N\), \(m_5\in\Z\), \(\delta\in\{0,1\}\) and \(\delta'\in\{-1,0,1\}\), and then observe that for any \(M'\in\Z\) one has
	\begin{align}
	\label{standard ordering Case 2}
	\begin{split}
	&(\epsilon\circ \widehat{P}^{i,j}_{-M',-1})\left(\Ei^{m_1}\Ej^{\delta}\Ei^{m_2}F_i^{m_3}F_j^{\delta}F_i^{m_4}K_i^{m_5}K_j^{\delta'}\right)\\ =\ & \left\{
	\begin{array}{ll}
	1 \quad & \mathrm{if}\ m_1 = m_2 = m_3 = m_4 = \delta = 0,\ m_5 = -M'\ \mathrm{and}\ \delta' = -1,\\
	0 & \mathrm{otherwise},
	\end{array}
	\right.
	\end{split}
	\end{align}
	by (\ref{P_i,j,N def}) and (\ref{Coproduct, counit, antipode def}). Hence \(\epsilon\circ \widehat{P}^{i,j}_{-M',-1}\) in fact projects \(Y\) onto the coefficient of \(K_i^{-M'}K_j^{-1}\) in its standard ordering.
	
	For \(N_1 = 0\), we may write \(Y = F_i^{N_0}\Ei F_jX = F_i^{N_0}\Ei X'\) since \(F_j\) and \(\Ei\) commute. A straightforward generalization of Lemma \ref{lemma epsilon and projection general} then asserts
	\[
		\left(\epsilon\circ \widehat{P}^{i,j}_{-M,-1}\right)(Y) = (N_0)_{q_i^2}q_i^{-2N_0+2}\left(\epsilon\circ \widehat{P}^{i,j}_{-(M-1),-1}\right)(F_i^{N_0-1}X'),
	\]
	which yields the claim since \(X' = F_jX\) and by the definition (\ref{alpha_N def}) of \(\alpha_{N_0}\) and the fact that \(\gamma_{N_0,N_0} = 0\). 
	
	Suppose now the claim has been proven for \(N_1-1\geq 0\). Note that by (\ref{U_q(g) relations 2}) and (\ref{E_i tilde def}) we have
	\[
	Y = F_i^{N_0}F_jF_i^{N_1-1}\Ei X' -  F_i^{N_0}F_jF_i^{N_1-1}K_iX +  F_i^{N_0}F_jF_i^{N_1-1}K_i^{-1}X,
	\]
	with \(X' = F_iX\). 
	
	The second term will not contribute, since its standard ordering cannot contain a multiple of \(K_i^{-M}K_j^{-1}\). Indeed, this term contains \(M-1\) factors \(F_i\) and the same number of factors \(\Ei\), and each factor \(F_i\) can only contribute one factor \(K_i^{-1}\) to the normal ordering upon combining it with a factor \(\Ei\), by (\ref{U_q(g) relations 2}). Hence the lowest possible power of \(K_i\) occurring in the standard ordering of this term will be \(-(M-1)+1 > -M\).
	
	The third term contains again as many \(F_i\) as \(\Ei\) and can whence be rewritten as \(q_i^xF_i^{N_0}F_jF_i^{N_1-1}XK_i^{-1}\), with
	\begin{align*}
	x =\ & -2\left(\#(\mathrm{factors}\ F_i\ \mathrm{preceding}\ K_i^{-1}) - \#(\mathrm{factors}\ \Ei\ \mathrm{preceding}\ K_i^{-1})\right)\\& - a_{ij}\left(\#(\mathrm{factors}\ F_j\ \mathrm{preceding}\ K_i^{-1})-\#(\mathrm{factors}\ \Ej\ \mathrm{preceding}\ K_i^{-1}) \right) \\ =\ & -2(N_0+N_1-1) - a_{ij}.
	\end{align*}
	So we have
	\[
		\left(\epsilon\circ \widehat{P}^{i,j}_{-M,-1}\right)(Y) = \left(\epsilon\circ \widehat{P}^{i,j}_{-M,-1}\right)(F_i^{N_0}F_jF_i^{N_1-1}\Ei X') + q_i^{-2(N_0+N_1-1) - a_{ij}}\left(\epsilon\circ \widehat{P}^{i,j}_{-(M-1),-1}\right)(F_i^{N_0}F_jF_i^{N_1-1}X).
	\]
	As \(F_i^{N_0}F_jF_i^{N_1-1}\Ei X'\) still contains \(M\) factors \(F_i\) and the same number of factors \(\Ei\), and meets all other requirements of the statement as well, we may apply the induction hypothesis to write
	\begin{align*}
	 \left(\epsilon\circ \widehat{P}^{i,j}_{-M,-1}\right)(F_i^{N_0}F_jF_i^{N_1-1}\Ei X') =\ & \alpha_{N_0} \left(\epsilon\circ \widehat{P}^{i,j}_{-(M-1),-1}\right)(F_i^{N_0-1}F_jF_i^{N_1-1}X')\\& + \gamma_{N_0,N_0+N_1-1} \left(\epsilon\circ \widehat{P}^{i,j}_{-(M-1),-1}\right)(F_i^{N_0}F_jF_i^{N_1-2} X').
	\end{align*}
	The statement now follows from \(X' = F_iX\) and the fact that
	\[
	\gamma_{N_0,N_0+N_1-1} + q_i^{-2(N_0+N_1-1) - a_{ij}} = \gamma_{N_0,N_0+N_1}.
	\]
\end{proof}

The formula obtained in Lemma \ref{lemma epsilon and projection general} can easily be iterated, since its right-hand side consists of only one term, leading to a product iteration of the form (\ref{informal formula}). The formula obtained in Lemma \ref{lemma epsilon and projection general Case 2} however, is much more complicated, since its right-hand side consists of two different terms, each containing a projection operator and the counit \(\epsilon\). One iteration of Lemma \ref{lemma epsilon and projection general Case 2} hence leads to a right-hand side containing three terms. Indeed, if \(Y = F_i^{N_0}F_jF_i^{N_1}\Ei F_i^{N_2}\Ei X\) is of the type described in Lemma \ref{lemma epsilon and projection general Case 2}, then 
\begin{align*}
	\left(\epsilon\circ \widehat{P}^{i,j}_{-M,-1}\right)(Y) =\ & \alpha_{N_0}\alpha_{N_0-1}\left(\epsilon\circ \widehat{P}^{i,j}_{-(M-2),-1}\right)(F_i^{N_0-2}F_jF_i^{N_1+N_2}X) \\& + \alpha_{N_0}(\gamma_{N_0-1,N_0+N_1+N_2-1}+\gamma_{N_0,N_0+N_1})\left(\epsilon\circ \widehat{P}^{i,j}_{-(M-2),-1}\right)(F_i^{N_0-1}F_jF_i^{N_1+N_2-1}X)\\&+\gamma_{N_0,N_0+N_1}\gamma_{N_0,N_0+N_1+N_2-1}\left(\epsilon\circ \widehat{P}^{i,j}_{-(M-2),-1}\right)(F_i^{N_0}F_jF_i^{N_1+N_2-2}X).
\end{align*}
A second iteration will then lead to four terms in the right-hand side and so on. Meanwhile, the occurring coefficients become increasingly intricate at each further iteration. To describe the full outcome after \(T\) iterations, for any \(T\in\N\), let us introduce the notation
\begin{equation}
\label{c_n,m,N def}
c_{a,\boldsymbol{N}}^{(b)} = \sum_{p_1\leq p_2\leq \dots\leq p_{b-a} = 0}^a \prod_{r = 1}^{b-a} \gamma_{N_0-a+p_r,\vert\boldsymbol{N}\vert_{0;b-p_r-r+1}-(b-p_r-r)},
\end{equation}
where \(a< b\in\N\), \(\boldsymbol{N} = (N_0,N_1,\dots,N_{b+1})\in\N^{b+2}\) and \(\vert\boldsymbol{N}\vert_{0;p} = N_0+N_1+\dots+N_p\). We also set \(c_{a,\bN}^{(a)} = 1\).

\begin{proposition}
	\label{prop epsilon and projection general full induction}
	Let \(M\in\N\) be such that \(M\geq 1\). Let \(Y\in\Uqgp\) be a product of \(M\) factors \(F_i\), \(M\) factors \(\widetilde{E_i}\), 1 factor \(F_j\) and 1 factor \(\widetilde{E_j}\), of the form
	\begin{equation}
	\label{form of Y for induction}
	Y = F_i^{N_0}F_jF_i^{N_1}\Ei F_i^{N_2}\Ei \dots F_i^{N_T}\Ei \Ej X,
	\end{equation}
	for some \(X\in\Uqgp\), where \(\bN = (N_0,N_1,\dots,N_T)\in\N^{T+1}\) and \(T\geq 1\). Then we have
	\[
	\left(\epsilon\circ \widehat{P}^{i,j}_{-M,-1}\right)(Y) = \upsilon_{\bN}\left(\epsilon\circ \widehat{P}^{i}_{-(M-T)}\right)\left(F_i^{\vert\bN\vert_{0;T}-T}X\right),
	\]
	with
	\[
	\upsilon_{\bN} = \sum_{u = \max(0,T-\vert\bN\vert_{1;T}-1)}^{T-1}q_i^{-a_{ij}(N_0-u)}c_{u,\bN}^{(T-1)}\left(q_i^{a_{ij}}\alpha_{N_0-u}+\gamma_{N_0-u,\vert\bN\vert_{0;T}-(T-1)}\right)\left(\prod_{r=0}^{u-1}\alpha_{N_0-r}\right).
	\]
\end{proposition}
\begin{proof}
	We will prove this by induction on \(T\). For \(T = 1\) we have \(Y = F_i^{N_0}F_jF_i^{N_1}\Ei \Ej X\) and so its follows from Lemma \ref{lemma epsilon and projection general Case 2} that
	\[
	\left(\epsilon\circ \widehat{P}^{i,j}_{-M,-1}\right)(Y) = \alpha_{N_0}\left(\epsilon\circ \widehat{P}^{i,j}_{-(M-1),-1}\right)(F_i^{N_0-1}F_j\Ej F_i^{N_1}X) + \gamma_{N_0,N_0+N_1}\left(\epsilon\circ \widehat{P}^{i,j}_{-(M-1),-1}\right)(F_i^{N_0}F_j\Ej F_i^{N_1-1}X),
	\]
	where we have used the fact that \(F_i\) and \(\Ej\) commute. When rewriting \(F_j\Ej\) in its standard ordering via
	\[
	F_j\Ej = \Ej F_j -  K_j + K_j^{-1},
	\]
	only the last term will contribute by (\ref{standard ordering Case 2}), so we may replace \(F_j\Ej\) by \(K_j^{-1}\) in the equation above. Since both \(F_i^{N_0-1}K_j^{-1} F_i^{N_1}X\) and \(F_i^{N_0}K_j^{-1} F_i^{N_1-1}X\) contain as many \(F_i\) as \(\Ei\), it is evident that
	\[
	F_i^{N_0-1}K_j^{-1} F_i^{N_1}X = q_i^{-a_{ij}(N_0-1)}F_i^{N_0+N_1-1}XK_j^{-1}, \qquad F_i^{N_0}K_j^{-1} F_i^{N_1-1}X = q_i^{-a_{ij}N_0}F_i^{N_0+N_1-1}XK_j^{-1}.
	\]
	Hence it follows from (\ref{P_i,N def}) and (\ref{P_i,j,N def}) that
	\[
		\left(\epsilon\circ \widehat{P}^{i,j}_{-M,-1}\right)(Y) = q_i^{-a_{ij}N_0}\left(q_i^{a_{ij}}\alpha_{N_0} +  \gamma_{N_0,N_0+N_1}  \right) \left(\epsilon\circ \widehat{P}^{i}_{-(M-1)}\right)(F_i^{N_0+N_1-1}X),
	\]
	which agrees with the claim since \(c_{0,(N_0,N_1)}^{(0)} = 1\) and \(T-\vert\bN\vert_{1;T}-1\leq 0\) for \(T = 1\). 
	
	Suppose now the claim has been proven for \(T\geq 1\) and set
	\[
	Y = F_i^{N_0}F_jF_i^{N_1}\Ei F_i^{N_2}\Ei \dots F_i^{N_T}\Ei F_i^{N_{T+1}}\Ei \Ej X.
	\]
	Then Lemma \ref{lemma epsilon and projection general Case 2} asserts
	\[
	\left(\epsilon\circ \widehat{P}^{i,j}_{-M,-1}\right)(Y) = \alpha_{N_0}\left(\epsilon\circ \widehat{P}^{i,j}_{-(M-1),-1}\right)(Y') + \gamma_{N_0,N_0+N_1}\left(\epsilon\circ \widehat{P}^{i,j}_{-(M-1),-1}\right)(Y''),
	\]
	with
	\begin{align*}
	Y' =\ & F_i^{N_0-1}F_jF_i^{N_1+N_2}\Ei F_i^{N_3}\Ei\dots F_i^{N_{T+1}}\Ei\Ej X, \\
	Y'' =\ & F_i^{N_0}F_jF_i^{N_1+N_2-1}\Ei F_i^{N_3}\Ei\dots F_i^{N_{T+1}}\Ei\Ej X.
	\end{align*}
	Both \(Y'\) and \(Y''\) satisfy the requirements of the statement: they each contain \(M-1\) factors \(F_i\), the same number of factors \(\Ei\), 1 factor \(F_j\) and 1 factor \(\Ej\), and they are of the form (\ref{form of Y for induction}) with 
	\[
	\bN' = (N_0-1,N_1+N_2,N_3,\dots,N_{T+1})\ \mathrm{and}\ \bN'' = (N_0,N_1+N_2-1,N_3,\dots,N_{T+1})
	\]
	respectively. Both \(N_0-1\) and \(N_1+N_2-1\) might become negative, but in this case the corresponding coefficients \(\alpha_{N_0}\) and \(\gamma_{N_0,N_0+N_1}\) will vanish.	We may hence assume that \(\bN',\bN'' \in\N^{T+1}\) and apply the induction hypothesis to obtain
	\[
	\left(\epsilon\circ \widehat{P}^{i,j}_{-M,-1}\right)(Y) = \Theta_{\bN',\bN''}\left(\epsilon\circ \widehat{P}^{i}_{-(M-T-1)}\right)(F_i^{\vert\bN\vert_{0;T+1}-(T+1)}X)
	\]
	where \(\Theta_{\bN',\bN''}\) is given by
	\begin{align}
	\label{Theta_N',N'' def}
	\begin{split}
	& \alpha_{N_0}\left[\sum_{u = \max(0,T-\vert\bN\vert_{1;T+1}-1)}^{T-1}q_i^{-a_{ij}(N_0-1-u)}c_{u,\bN'}^{(T-1)}\left(q_i^{a_{ij}}\alpha_{N_0-1-u}+\gamma_{N_0-1-u,\vert\bN\vert_{0;T+1}-T} \right)\left(\prod_{r = 0}^{u-1}\alpha_{N_0-1-r}\right) \right]\\
	&+\gamma_{N_0,N_0+N_1}\left[\sum_{u = \max(0,T-\vert\bN\vert_{1;T+1})}^{T-1}q_i^{-a_{ij}(N_0-u)}c_{u,\bN''}^{(T-1)}\left(q_i^{a_{ij}}\alpha_{N_0-u}+\gamma_{N_0-u,\vert\bN\vert_{0;T+1}-T} \right)\left(\prod_{r = 0}^{u-1}\alpha_{N_0-r}\right) \right],
	\end{split}
	\end{align}
	where we have used the fact that \(\vert\bN'\vert_{0;T}-T = \vert\bN''\vert_{0;T}-T = \vert\bN\vert_{0;T+1}-(T+1)\). It now suffices to show that
	\begin{equation}
	\label{TBP in full induction}
	\Theta_{\bN',\bN''} =  \left[\sum_{u = \max(0,T-\vert\bN\vert_{1;T+1})}^{T}q_i^{-a_{ij}(N_0-u)}c_{u,\bN}^{(T)}\left(q_i^{a_{ij}}\alpha_{N_0-u}+\gamma_{N_0-u,\vert\bN\vert_{0;T+1}-T}\right)\left(\prod_{r=0}^{u-1}\alpha_{N_0-r}\right)\right].
	\end{equation}
	
	Upon replacing the summation index \(u\) in the first line in (\ref{Theta_N',N'' def}) by \(u' = u+1\), which we thereafter rename to \(u\) again, this term becomes
	\begin{equation}
	\label{replacement of first line}
	\alpha_{N_0}\left[\sum_{u = \max(1,T-\vert\bN\vert_{1;T+1})}^{T}q_i^{-a_{ij}(N_0-u)}c_{u-1,\bN'}^{(T-1)}\left(q_i^{a_{ij}}\alpha_{N_0-u}+\gamma_{N_0-u,\vert\bN\vert_{0;T+1}-T} \right)\left(\prod_{r = 0}^{u-2}\alpha_{N_0-1-r}\right) \right]
	\end{equation}
	and it is immediate that \(\alpha_{N_0}\left(\prod_{r = 0}^{u-2}\alpha_{N_0-1-r}\right) = \prod_{r = 0}^{u-1}\alpha_{N_0-r}\). Replacing the first line of (\ref{Theta_N',N'' def}) by (\ref{replacement of first line}) and separating the term corresponding to \(u = 0\) in the second line and the one with \(u = T\) in the first line, we find that \(\Theta_{\bN',\bN''}\) equals
	\begin{align}
	\label{Theta after induction hypothesis}
	\begin{split}
	&q_i^{-a_{ij}N_0}c_{0,\bN''}^{(T-1)}\gamma_{N_0,N_0+N_1}\left(q_i^{a_{ij}}\alpha_{N_0}+\gamma_{N_0,\vert\bN\vert_{0;T+1}-T} \right)\nu_{\bN,T} \\
	+& \sum_{u = \max(1,T-\vert\bN\vert_{1;T+1})}^{T-1}\left[q_i^{-a_{ij}(N_0-u)}\left(c_{u-1,\bN'}^{(T-1)} + \gamma_{N_0,N_0+N_1}c_{u,\bN''}^{(T-1)} \right)\left(q_i^{a_{ij}}\alpha_{N_0-u}+\gamma_{N_0-u,\vert\bN\vert_{0;T+1}-T} \right)\left(\prod_{r = 0}^{u-1}\alpha_{N_0-r}\right)  \right] \\
	+&\, q_i^{-a_{ij}(N_0-T)}c_{T-1,\bN'}^{(T-1)}\left(q_i^{a_{ij}}\alpha_{N_0-T}+\gamma_{N_0-T,\vert\bN\vert_{0;T+1}-T} \right)\left(\prod_{r = 0}^{T-1}\alpha_{N_0-r}\right),
	\end{split}
	\end{align}
	with
	\[
	\nu_{\bN, T} = \left\{
	\begin{array}{ll}
	1 \quad & \mathrm{if}\ T-\vert\bN\vert_{1;T+1}\leq 0, \\
	0 & \mathrm{otherwise}.
	\end{array}
	\right.
	\]
	By definition of \(c_{a,\bN}^{(b)}\) we have that \(c_{T-1,\bN'}^{(T-1)} = c_{T,\bN}^{(T)} = 1\), such that the last line in (\ref{Theta after induction hypothesis}) agrees with the term in the right-hand side of (\ref{TBP in full induction}) corresponding to \(u = T\). Hence it suffices to prove the following two claims:
	\begin{align}
	\label{TBP 1}
	c_{0,\bN}^{(T)} & = c_{0,\bN''}^{(T-1)}\gamma_{N_0,N_0+N_1}, \\
	\label{TBP 2}
	c_{u,\bN}^{(T)} & = c_{u-1,\bN'}^{(T-1)} + \gamma_{N_0,N_0+N_1}c_{u,\bN''}^{(T-1)},
	\end{align}
	for all \(u\in\{\max(1,T-\vert\bN\vert_{1;T+1}),\dots,T-1\}\).
	
	It follows immediately from (\ref{c_n,m,N def}) that one has
	\[
	c_{0,\bN''}^{(T-1)}\gamma_{N_0,N_0+N_1} = \left(\prod_{r=1}^{T-1}\gamma_{N_0,\vert\bN''\vert_{0;T-r}-(T-r-1)}\right)\gamma_{N_0,N_0+N_1}.
	\]
	The definition of \(\bN''\) asserts that \(\vert\bN''\vert_{0;T-r} = \vert\bN\vert_{0;T-r+1}-1\) for any \(r\in\{1,\dots,T-1\}\), and hence 
	\[
	c_{0,\bN''}^{(T-1)}\gamma_{N_0,N_0+N_1} = \prod_{r = 1}^{T}\gamma_{N_0,\vert\bN\vert_{0;T-r+1}-(T-r)} = c_{0,\bN}^{(T)},
	\]
	which proves (\ref{TBP 1}).
	
	Now let \(u\in\{\max(1,T-\vert\bN\vert_{1;T+1}),\dots,T-1\}\) be fixed. By (\ref{c_n,m,N def}) we have
	\[
	c_{u-1,\bN'}^{(T-1)} = \sum_{p_1\leq p_2\leq\dots\leq p_{T-u} = 0}^{u-1}\prod_{r = 1}^{T-u}\gamma_{N_0-u+p_r,\vert\bN'\vert_{0;T-p_r-r}-(T-p_r-r-1)},
	\]
	where we have used the fact that \(N_0' = N_0-1\). Now since for every occurring \(r\) one has \(r\leq T-u\) and \(p_r\leq u-1\), we have that \(T-p_r-r\geq 1\) and hence \(\vert\bN'\vert_{0;T-p_r-r}= \vert\bN\vert_{0;T-p_r-r+1}-1\), such that
	\begin{equation}
	\label{TBP 2 step 1}
		c_{u-1,\bN'}^{(T-1)} = \sum_{p_1\leq p_2\leq\dots\leq p_{T-u} = 0}^{u-1}\prod_{r = 1}^{T-u}\gamma_{N_0-u+p_r,\vert\bN\vert_{0;T+1-p_r-r}-(T-p_r-r)}.
	\end{equation}
	It is evident that one has
	\begin{align*}
	&\{(p_1,\dots,p_{T-u-1},p_{T-u})\in\N^{T-u}: 0\leq p_1\leq\dots\leq p_{T-u-1}\leq p_{T-u}\leq u\}\\ =\ & \{(p_1,\dots,p_{T-u-1},p_{T-u})\in\N^{T-u}: 0\leq p_1\leq\dots\leq p_{T-u-1}\leq p_{T-u}\leq u-1\}\\& \cup
	\{(p_1,\dots,p_{T-u-1},u)\in\N^{T-u}: 0\leq p_1\leq\dots\leq p_{T-u-1}\leq u\}.
	\end{align*}
	Hence (\ref{TBP 2 step 1}) implies
	\begin{align}
	\label{TBP 2 minus}
	\begin{split}
	c_{u-1,\bN'}^{(T-1)} =\ & \sum_{p_1\leq p_2\leq\dots\leq p_{T-u} = 0}^{u}\prod_{r = 1}^{T-u}\gamma_{N_0-u+p_r,\vert\bN\vert_{0;T+1-p_r-r}-(T-p_r-r)} \\
	& - \gamma_{N_0,\vert\bN\vert_{0;1}}\sum_{p_1\leq p_2\leq\dots\leq p_{T-u-1} = 0}^{u}\prod_{r = 1}^{T-u-1}\gamma_{N_0-u+p_r,\vert\bN\vert_{0;T+1-p_r-r}-(T-p_r-r)},
	\end{split}
	\end{align}
	where in the last line we have separated the factor in the product corresponding to \(r = T-u\), since here we have set \(p_{T-u} = u\). One immediately recognizes the first line as \(c_{u,\bN}^{(T)}\), and moreover one has
	\begin{equation}
	\label{TBP 2 step 2}
	c_{u,\bN''}^{(T-1)} = \sum_{p_1\leq p_2\leq\dots\leq p_{T-u-1} = 0}^{u}\prod_{r = 1}^{T-u-1}\gamma_{N_0-u+p_r,\vert\bN''\vert_{0;T-p_r-r}-(T-p_r-r-1)}.
	\end{equation}
	Again every \(T-p_r-r\geq 1\) and hence \(\vert\bN''\vert_{0;T-p_r-r} = \vert\bN\vert_{0;T+1-p_r-r}-1\), such that (\ref{TBP 2 step 2}) coincides with the sum in the last line in (\ref{TBP 2 minus}). We conclude that
	\[
	c_{u-1,\bN'}^{(T-1)} = c_{u,\bN}^{(T)} - \gamma_{N_0,N_0+N_1}c_{u,\bN''}^{(T-1)}
	\]
	and so we have shown (\ref{TBP 2}). This concludes the proof.
\end{proof}

The question now remains how one can apply Proposition \ref{prop epsilon and projection general full induction} to compute the action of \(\epsilon\circ \widehat{P}^{i,j}_{-\frac{1-a_{ij}-m}{2},-1}\) on \(Y_{\bl,\bss,k,d}^{(0)}-q_i^{a_{ij}}Y_{\bl,\bss,k,d}^{(1)}\), as defined in (\ref{Y_l,s,k,d 0 def})--(\ref{Y_l,s,k,d 1 def}). This will be addressed in the following proposition.

\begin{proposition}
	\label{prop epsilon and P final Case 2}
	Let \(i,j,m,t,k,d,m',\bl\) and \(\bss\) be as fixed before and let \(\lambda\in\{0,1\}\), then one has
	\begin{align*}
	&\left(\epsilon\circ \widehat{P}^{i,j}_{-\frac{1-a_{ij}-m}{2},-1}\right)\left(Y_{\bl,\bss,k,d}^{(\lambda)}\right)\\ = &\ q_i^{-a_{ij}N_0}
	\left(\prod_{r \in \mathcal{R}_{k,d}}\left(\alpha_{\zeta_{\bl,\bss}^{(r-1)}-\nu_{r,k}}\right)^{(1-\ell_r)(1-s_{r-\vert\bl\vert_{1;r}})}\right)\left(\alpha_{\zeta_{\bl,\bss}^{(1-a_{ij}-k+d)}}\right)^{1-\lambda}\\& \left(\sum_{u = \max(0,\xi_{\lambda})}^{T_{\bl,\bss,k,t}+\lambda-1}\left[q_i^{a_{ij}u}c_{u,\bN^{(\lambda)}}^{(T_{\bl,\bss,k,t}+\lambda-1)}\left(q_i^{a_{ij}}\alpha_{N_0-u}+\gamma_{N_0-u,\vert\bN\vert_{0;T_{\bl,\bss,k,t}+\lambda}-(T_{\bl,\bss,k,t}+\lambda-1)} \right)\left(\prod_{r = 0}^{u-1}\alpha_{N_0-r}\right) \right]\right)^{1-\delta_{T_{\bl,\bss,k,t}+\lambda,0}},
	\end{align*}
	where \(\bN^{(\lambda)} = (N_0,N_1,\dots,N_{T_{\bl,\bss,k,t}+\lambda})\), with
	\begin{align}
	\label{T_k,m',d def}
	T_{\bl,\bss,k,t} &= \zeta_{\bl,\bss}^{(1-a_{ij}-k)}+t-\vert\bss\vert_{1;1-a_{ij}-k-\vert\bl\vert_{1;1-a_{ij}-k}}, \\
	\label{N_0 def}
	N_0 & = \zeta_{\bl,\bss}^{(1-a_{ij}-k)},\\
	\label{N_1;p def}
	\vert\bN\vert_{1;b} & = r_b+a_{ij}+k-b-1-\vert\bl\vert_{2-a_{ij}-k;r_b}, \\
	\label{r_p def}
	r_b & = \sum_{r = 2-a_{ij}-k}^{1-a_{ij}-k+d}r(1-\ell_r)(1-s_{r-\vert\bl\vert_{1;r}})\delta_{r+a_{ij}+k-b-1-\vert\bl\vert_{2-a_{ij}-k;r},\vert\bss\vert_{2-a_{ij}-k-\vert\bl\vert_{1;2-a_{ij}-k};r-\vert\bl\vert_{1;r}}}
	\end{align}
	for any \(b\in\{1,\dots,T_{\bl,\bss,k,t}\}\), and
	\begin{equation}
	\label{N_0,T+1 def}
	\vert\bN\vert_{0;T_{\bl,\bss,k,t}+1}= \zeta_{\bl,\bss}^{(1-a_{ij}-k+d)}+T_{\bl,\bss,k,t},
	\end{equation}
	and moreover
	\begin{align}
	\label{R_k,d and nu_r,k def}
	\begin{split}
	\xi_{\lambda} & = T_{\bl,\bss,k,t}-\vert\bN\vert_{1;T_{\bl,\bss,k,t}+\lambda}+\lambda-1, \\
	\mathcal{R}_{k,d} & = \{1,\dots,-a_{ij}\}\setminus\{2-a_{ij}-k,\dots,1-a_{ij}-k+d\},\\
	\nu_{r,k} & = \left\{
	\begin{array}{ll}
	0\quad & \mathrm{if}\ r\leq 1-a_{ij}-k,\\
	1 & \mathrm{if}\ r > 1-a_{ij}-k.
	\end{array}
	\right.
	\end{split}
	\end{align}
\end{proposition}
\begin{proof}
	Following the same reasoning that led us to the formula (\ref{product iteration for Case 2}), we find that \(\left(\epsilon\circ \widehat{P}^{i,j}_{-\frac{1-a_{ij}-m}{2},-1}\right)\left(Y_{\bl,\bss,k,d}^{(\lambda)}\right)\) equals
	\[
	\left(\prod_{r=1}^{1-a_{ij}-k}\alpha_{\zeta_{\bl,\bss}^{(r-1)}}^{(1-\ell_r)(1-s_{r-\vert\bl\vert_{1;r}})} \right)\left(\epsilon\circ \widehat{P}^{i,j}_{-x,-1}\right)\left[F_i^{\zeta_{\bl,\bss}^{(1-a_{ij}-k)}}F_j\left(\overrightarrow{\prod_{r=2-a_{ij}-k}^{1-a_{ij}-k+d}}\mathcal{V}^{i}_{\bl,\bss,r}\right)\Ei^{\lambda}\Ej\Ei^{1-\lambda}\left(\overrightarrow{\prod_{r=2-a_{ij}-k+d}^{-a_{ij}}}\mathcal{V}^{i}_{\bl,\bss,r}\right) \right],
	\]
	where
	\[
	x = \frac{1-a_{ij}-m}{2}-\#(\mathrm{factors}\ \Ei\ \mathrm{preceding}\ F_j) = \zeta_{\bl,\bss}^{(1-a_{ij}-k)}-\vert\bss\vert_{2-a_{ij}-k-\vert\bl\vert_{1;2-a_{ij}-k};-a_{ij}-m}.
	\]
	
	To proceed, we will need to write the term between square brackets in the form
	\[
	F_i^{N_0}F_jF_i^{N_1}\Ei F_i^{N_2}\Ei\dots F_i^{N_T}\Ei\Ej X,
	\]
	for some \(N_0,\dots,N_T\in\N\), \(T\in\N\), \(X\in\Uqgp\). It is immediately clear that \(N_0\) agrees with (\ref{N_0 def}). Furthermore, let \(T\) be the total number of factors \(\Ei\) in \(\overrightarrow{\prod_{r=2-a_{ij}-k}^{1-a_{ij}-k+d}}\mathcal{V}^{i}_{\bl,\bss,r}\) and let us define \(r_1< r_2<\dots<r_T\in\{2-a_{ij}-k,\dots,1-a_{ij}-k+d\}\) such that  
	\[
	\mathcal{V}^{i}_{\bl,\bss,r_b} = \Ei
	\]
	for all \(b\in\{1,\dots,T\}\). This amounts to saying that \(r_1,\dots,r_T\) are the positions of the factors \(\Ei\) in this product. Then for any \(b\) one has
	\begin{align*}
	\vert\bN\vert_{1;b} &= \#(\mathrm{elements}\ r\in\{2-a_{ij}-k,\dots,r_b\}\ \mathrm{such} \ \mathrm{that}\ \mathcal{V}^{i}_{\bl,\bss,r} = F_i) \\
	&= \#(\mathrm{elements}\ r\in\{2-a_{ij}-k,\dots,r_b\}\ \mathrm{such} \ \mathrm{that}\ \ell_r = 0\ \mathrm{and}\ s_{r-\vert\bl\vert_{1;r}}=1) \\
	& = \#(\mathrm{elements}\ r\in\{2-a_{ij}-k,\dots,r_b\})- \#(\mathrm{elements}\ r\in\{2-a_{ij}-k,\dots,r_b\}\ \mathrm{such} \ \mathrm{that}\ \ell_r = 1)\\&\phantom{=} -\#(\mathrm{elements}\ r\in\{2-a_{ij}-k,\dots,r_b\}\ \mathrm{such} \ \mathrm{that}\ \mathcal{V}^{i}_{\bl,\bss,r} = \Ei) \\
	& = (r_b-(1-a_{ij}-k)) - \vert\bl\vert_{2-a_{ij}-k;r_b}-b.
	\end{align*} 
	Note also that this number equals
	\[
	\vert\bN\vert_{1;b}=\vert\bss\vert_{2-a_{ij}-k-\vert\bl\vert_{1;2-a_{ij}-k};r_b-\vert\bl\vert_{1;r_b}}.
	\]
	Hence for any \(b\in\{1,\dots,T\}\), the element \(r_b\) can be found as the unique \(r\in\{2-a_{ij}-k,\dots,1-a_{ij}-k+d\}\) such that \(\ell_r = 0\), \(s_{r-\vert\bl\vert_{1;r}} = 0\) and
	\[
	\vert\bss\vert_{2-a_{ij}-k-\vert\bl\vert_{1;2-a_{ij}-k};r-\vert\bl\vert_{1;r}} = (r-(1-a_{ij}-k)) - \vert\bl\vert_{2-a_{ij}-k;r}-b.
	\]
	This agrees with (\ref{N_1;p def})--(\ref{r_p def}). The total number \(T\) of factors \(\Ei\) in \(\overrightarrow{\prod_{r=2-a_{ij}-k}^{1-a_{ij}-k+d}}\mathcal{V}^{i}_{\bl,\bss,r}\) can be found as
	\begin{align*}
	T =\ & \#(r\in\{2-a_{ij}-k,\dots,1-a_{ij}-k+d\}) - \#(r\in\{2-a_{ij}-k,\dots,1-a_{ij}-k+d\}\ \mathrm{such}\ \mathrm{that}\ \ell_r = 1)\\&-\#(r\in\{2-a_{ij}-k,\dots,1-a_{ij}-k+d\}\ \mathrm{such}\ \mathrm{that}\ \ell_r = 0\ \mathrm{and}\ s_{r-\vert\bl\vert_{1;r}}=1) \\
	=\ & d - \vert\bl\vert_{2-a_{ij}-k;1-a_{ij}-k+d}-\sum_{r = 2-a_{ij}-k}^{1-a_{ij}-k+d}(1-\ell_r)s_{r-\vert\bl\vert_{1;r}} \\
	=\ & \vert\bl\vert_{1;1-a_{ij}-k}+\vert\bss\vert_{1;1-a_{ij}-k-\vert\bl\vert_{1;1-a_{ij}-k}}+a_{ij}+k+t-1,
	\end{align*}
	in agreement with (\ref{T_k,m',d def}). 
	
	With these notations one may now write
	\[
	F_i^{\zeta_{\bl,\bss}^{(1-a_{ij}-k)}}F_j\left(\overrightarrow{\prod_{r=2-a_{ij}-k}^{1-a_{ij}-k+d}}\mathcal{V}^{i}_{\bl,\bss,r}\right)\Ej\Ei = F_i^{N_0}F_jF_i^{N_1}\Ei F_i^{N_2}\Ei\dots F_i^{N_{T_{\bl,\bss,k,t}}}\Ei\Ej F_i^{N_{T_{\bl,\bss,k,t}+1}}\Ei,
	\]
	for some \(N_{T_{\bl,\bss,k,t}+1}\in\N\), where we have used the fact that \([F_i,\Ej] = 0\). The analogous term with \(\lambda = 1\) becomes
	\[
	F_i^{\zeta_{\bl,\bss}^{(1-a_{ij}-k)}}F_j\left(\overrightarrow{\prod_{r=2-a_{ij}-k}^{1-a_{ij}-k+d}}\mathcal{V}^{i}_{\bl,\bss,r}\right)\Ei\Ej = F_i^{N_0}F_jF_i^{N_1}\Ei F_i^{N_2}\Ei\dots F_i^{N_{T_{\bl,\bss,k,t}}}\Ei F_i^{N_{T_{\bl,\bss,k,t}+1}}\Ei\Ej,
	\]
	for the same unknown \(N_{T_{\bl,\bss,k,t}+1}\in\N\).	By Proposition \ref{prop epsilon and projection general full induction} we thus have, for \(\lambda = 0\)
	\begin{align*}
	&\left(\epsilon\circ \widehat{P}^{i,j}_{-x,-1}\right)\left[F_i^{\zeta_{\bl,\bss}^{(1-a_{ij}-k)}}F_j\left(\overrightarrow{\prod_{r=2-a_{ij}-k}^{1-a_{ij}-k+d}}\mathcal{V}^{i}_{\bl,\bss,r}\right)\Ej\Ei\left(\overrightarrow{\prod_{r=2-a_{ij}-k+d}^{-a_{ij}}}\mathcal{V}^{i}_{\bl,\bss,r}\right) \right]\\ =\ & \left[\sum_{u = \max(0,\xi_0)}^{T_{\bl,\bss,k,t}-1}q_i^{-a_{ij}(N_0-u)}c_{u,\bN^{(0)}}^{(T_{\bl,\bss,k,t}-1)}\left(q_i^{a_{ij}}\alpha_{N_0-u}+\gamma_{N_0-u,\vert\bN\vert_{0;T_{\bl,\bss,k,t}}-(T_{\bl,\bss,k,t}-1)}\right)\left(\prod_{r=0}^{u-1}\alpha_{N_0-r}\right)\right]^{1-\delta_{T_{\bl,\bss,k,t},0}}\\&\left(q_i^{-a_{ij}N_0}\right)^{\delta_{T_{\bl,\bss,k,t},0}}\left(\epsilon\circ \widehat{P}^{i}_{-(x-T_{\bl,\bss,k,t})}\right)\left(F_i^{\vert\bN\vert_{0;T_{\bl,\bss,k,t}+1}-T_{\bl,\bss,k,t}}\Ei\left(\overrightarrow{\prod_{r=2-a_{ij}-k+d}^{-a_{ij}}}\mathcal{V}^{i}_{\bl,\bss,r}\right)\right).
	\end{align*}
	Here we have observed that Proposition \ref{prop epsilon and projection general full induction} is only applicable for \(T\geq 1\), which explains the power \(\delta_{T_{\bl,\bss,k,t},0}\). The analogous term with \(\Ej\Ei\) replaced by \(\Ei\Ej\) becomes
	\begin{align*}
	& \left[\sum_{u = \max(0,\xi_1)}^{T_{\bl,\bss,k,t}}q_i^{-a_{ij}(N_0-u)}c_{u,\bN^{(1)}}^{(T_{\bl,\bss,k,t})}\left(q_i^{a_{ij}}\alpha_{N_0-u}+\gamma_{N_0-u,\vert\bN\vert_{0;T_{\bl,\bss,k,t}+1}-T_{\bl,\bss,k,t}}\right)\left(\prod_{r=0}^{u-1}\alpha_{N_0-r}\right)\right]\\&\left(\epsilon\circ\widehat{P}^{i}_{-(x-T_{\bl,\bss,k,t}-1)}\right)\left(F_i^{\vert\bN\vert_{0;T_{\bl,\bss,k,t}+1}-T_{\bl,\bss,k,t}-1}\left(\overrightarrow{\prod_{r=2-a_{ij}-k+d}^{-a_{ij}}}\mathcal{V}^{i}_{\bl,\bss,r}\right)\right).
	\end{align*}
	Note also that we have
	\begin{align*}
	&\vert\bN\vert_{0;T_{\bl,\bss,k,t}+1}-T_{\bl,\bss,k,t}\\ =\ & \#(\mathrm{factors}\ F_i\ \mathrm{in}\ Y_{\bl,\bss,k,d}^{(0)}\ \mathrm{preceding}\ \Ej\Ei)-\#(\mathrm{factors}\ \Ei\ \mathrm{in}\ Y_{\bl,\bss,k,d}^{(0)}\ \mathrm{preceding}\ \Ej\Ei)\\ =\ & \zeta_{\bl,\bss}^{(1-a_{ij}-k+d)},
	\end{align*}
	which determines the unknown \(N_{T_{\bl,\bss,k,t}+1}\), in agreement with (\ref{N_0,T+1 def}). By Lemma \ref{lemma epsilon and projection general} this also implies
	\begin{align*}
	&\left(\epsilon\circ \widehat{P}^{i}_{-(x-T_{\bl,\bss,k,t})}\right)\left(F_i^{\vert\bN\vert_{0;T_{\bl,\bss,k,t}+1}-T_{\bl,\bss,k,t}}\Ei\left(\overrightarrow{\prod_{r=2-a_{ij}-k+d}^{-a_{ij}}}\mathcal{V}^{i}_{\bl,\bss,r}\right)\right) \\=\ & \alpha_{\zeta_{\bl,\bss}^{(1-a_{ij}-k+d)}}\left(\epsilon\circ \widehat{P}^{i}_{-(x-T_{\bl,\bss,k,t}-1)}\right)\left(F_i^{\zeta_{\bl,\bss}^{(1-a_{ij}-k+d)}-1}\left(\overrightarrow{\prod_{r=2-a_{ij}-k+d}^{-a_{ij}}}\mathcal{V}^{i}_{\bl,\bss,r}\right)\right).
	\end{align*}
	
	It now remains only to apply the formula (\ref{product iteration for Case 2}) and find 
	\begin{align*}
	&\left(\epsilon\circ \widehat{P}^{i}_{-(x-T_{\bl,\bss,k,t}-1)}\right)\left(F_i^{\zeta_{\bl,\bss}^{(1-a_{ij}-k+d)}-1}\left(\overrightarrow{\prod_{r=2-a_{ij}-k+d}^{-a_{ij}}}\mathcal{V}^{i}_{\bl,\bss,r}\right)\right) = \prod_{r = 2-a_{ij}-k+d}^{-a_{ij}}\left(\alpha_{\zeta_{\bl,\bss}^{(r-1)}-1}\right)^{(1-\ell_r)(1-s_{r-\vert\bl\vert_{1;r}})}.
	\end{align*}
\end{proof}

With this result, we now have all necessary tools in hand to write down the polynomial \(C_{ij}(\bc)\) for Case 2.

\begin{theorem}[Case 2]
	\label{theorem F_ij(B_i,B_j) Case 2}
	For any \(i\in I\setminus X\) such that \(\tau(i) = i\) and any \(j\in X\), one has
	\begin{align}
	\label{F_ij(B_i,B_j) Case 2 complete}
	\begin{split}
	&F_{ij}(B_i,B_j) = C_{ij}(\bc)\\ = &\sum_{m=0}^{-1-a_{ij}}\sum_{m' = 0}^{-1-a_{ij}-m}\sum_{t = 0}^{\frac{1-a_{ij}-m-m'}{2}}\rho_{m,m',t}^{(i,j,a_{ij})}\mathcal{Z}_i^tB_i^mB_jB_i^{m'}\mathcal{Z}_i^{\frac{1-a_{ij}-m-m'}{2}-t}\\ &+ \sum_{m = 0}^{-1-a_{ij}}\sum_{t = 0}^{\frac{-1-a_{ij}-m}{2}}\sigma_{m,t}^{(i,j,a_{ij})}\mathcal{Z}_i^t\mathcal{W}_{ij}K_j\mathcal{Z}_i^{\frac{-1-a_{ij}-m}{2}-t}B_i^{m},
	\end{split}
	\end{align}
	with \(\rho_{m,m',t}^{(i,j,a_{ij})}\) as obtained in Corollary \ref{cor F_ij(B_i,B_j) Case 2 semi-final} and
	\begin{align*}
	\sigma_{m,t}^{(i,j,a_{ij})} =\ & (a_{ij}+m)_p\,c_i^{\frac{1-a_{ij}-m}{2}}\sum_{k = 1}^{1-a_{ij}}\sum_{d = 0}^{k-1}\sum_{m' = 0}^{m}\sum_{\bl\in\mathcal{L}_{m,m',k,d}'}\sum_{\bss\in\mathscr{S}_{m,m',k,t,d}'}(-1)^{k+1}\begin{bmatrix}
	1-a_{ij}\\k
	\end{bmatrix}_{q_i}\\&\frac{q_i^{\kappa_{\bl,\bss,k,t,d,m'}}}{(q_i-q_i^{-1})^{\frac{1-a_{ij}-m}{2}}(q_j-q_j^{-1})}\left(\prod_{r\in\mathcal{R}_{k,d}}\left(\alpha_{\zeta_{\bl,\bss}^{(r-1)}-\nu_{r,k}} \right)^{(1-\ell_r)(1-s_{r-\vert\bl\vert_{1;r}})}\right)\\&
	\left[
	\sum_{u = \max(0,\xi_0)}^{T_{\bl,\bss,k,t}-1}q_i^{a_{ij}u}\omega_{N^{(0)},u}+\sum_{u = \max(0,\xi_1)}^{T_{\bl,\bss,k,t}}q_i^{a_{ij}u}\omega_{N^{(1)},u} \right]^{1-\delta_{T_{\bl,\bss,k,t},0}}\\&
	\left[ 
	\alpha_{\zeta_{\bl,\bss}^{(1-a_{ij}-k+d)}}-q_i^{a_{ij}}\left(q_i^{a_{ij}}\alpha_{N_0}+\gamma_{N_0,N_0+N_1}\right)
	\right]^{\delta_{T_{\bl,\bss,k,t},0}},
	\end{align*}
	where
	\begin{align*}
	\omega_{N^{(0)},u} =\ & \alpha_{\zeta_{\bl,\bss}^{(1-a_{ij}-k+d)}}c_{u,\bN^{(0)}}^{(T_{\bl,\bss,k,t}-1)}\left(q_i^{a_{ij}}\alpha_{N_0-u}+\gamma_{N_0-u,\vert\bN\vert_{0;T_{\bl,\bss,k,t}}-(T_{\bl,\bss,k,t}-1)}\right)\left(\prod_{r = 0}^{u-1}\alpha_{N_0-r}\right), \\
	\omega_{N^{(1)},u} = \ &
	-q_i^{a_{ij}}c_{u,\bN^{(1)}}^{(T_{\bl,\bss,k,t})}\left(q_i^{a_{ij}}\alpha_{N_0-u}+\gamma_{N_0-u,\vert\bN\vert_{0;T_{\bl,\bss,k,t}+1}-T_{\bl,\bss,k,t}}\right)\left(\prod_{r = 0}^{u-1}\alpha_{N_0-r}\right), \\
	\kappa_{\bl,\bss,k,t,d,m'} = \ & -2\sum_{r = 1}^{-a_{ij}}\zeta_{\bl,\bss}^{(r-1)}\left(\ell_r+(1-\ell_r)(1-s_{r-\vert\bl\vert_{1;r}}) \right)\\& - a_{ij}(N_0-\vert\bN\vert_{1;T_{\bl,\bss,k,t}+1}-m'+d)+2(k+t-d-1).
	\end{align*}
	Here we have used the notations (\ref{L' and S' sets Case 2}), (\ref{zeta_l,s,r def}), (\ref{alpha_N def}), (\ref{gamma_M,N def}), (\ref{c_n,m,N def}), (\ref{T_k,m',d def}),  (\ref{N_0 def}), (\ref{N_1;p def}), (\ref{N_0,T+1 def}) and (\ref{R_k,d and nu_r,k def}).
\end{theorem}
\begin{proof}
	This follows upon combining Corollary \ref{cor F_ij(B_i,B_j) Case 2 part 2 with epsilon and P} and Propositions \ref{prop epsilon and other P Case 2} and \ref{prop epsilon and P final Case 2}, after expanding \(\eta_{\bl,\bs,k,d,t,m'}\) using (\ref{L' and S' sets Case 2}) and observing that
	\[
	\vert\bss\vert_{2-a_{ij}-k-\vert\bl\vert_{1;2-a_{ij}-k};1-a_{ij}-k-m'+d} = \vert\bN\vert_{1;T_{\bl,\bss,k,t}+1}.
	\]
\end{proof}

These expressions for the structure constants \(\sigma_{m,t}^{(i,j,a_{ij})}\) comply with the values computed in \cite{Kolb-2014}, as displayed in Table \ref{Table of sigma_m,t}. Moreover, Theorems \ref{theorem F_ij(B_i,B_j) for Case 1} and \ref{theorem F_ij(B_i,B_j) Case 2} and Corollary \ref{cor F_ij(B_i,B_j) Case 2 semi-final} make it possible to compute the structure constants for higher values of \(\vert a_{ij}\vert\). For example, it follows from Theorem \ref{theorem F_ij(B_i,B_j) for Case 1} that for \(a_{ij} = -4\) one has
\begin{align*}
F_{ij}(B_i,B_j) =\ & \rho_{0,1}^{(i,j,-4)}\mathcal{Z}_i^2B_jB_i + \rho_{1,0}^{(i,j,-4)}\mathcal{Z}_i^2B_iB_j + \rho_{0,3}^{(i,j,-4)}\mathcal{Z}_iB_jB_i^3 + \rho_{3,0}^{(i,j,-4)}\mathcal{Z}_iB_i^3B_j\\&+\rho_{1,2}^{(i,j,-4)}\mathcal{Z}_iB_iB_jB_i^2+\rho_{2,1}^{(i,j,-4)}\mathcal{Z}_iB_i^2B_jB_i,
\end{align*}
if \(i,j\in I\setminus X\) are distinct such that \(\tau(i) = i\), where the structure constants \(\rho_{m,m'}^{(i,j,-4)}\) are given in Table \ref{Table of rho_m,m' for a_ij = -4}. Similarly, for \(a_{ij} = -3\), \(i\in I\setminus X\) with \(\tau(i) = i\) and \(j\in X\) one has
\begin{align*}
F_{ij}(B_i,B_j) =\ & \rho_{0,0,0}^{(i,j,-3)}B_j\mathcal{Z}_i^2 + \rho_{0,0,1}^{(i,j,-3)}\mathcal{Z}_iB_j\mathcal{Z}_i + \rho_{0,0,2}^{(i,j,-3)}\mathcal{Z}_i^2B_j + \rho_{0,2,0}^{(i,j,-3)}B_jB_i^2\mathcal{Z}_i + \rho_{0,2,1}^{(i,j,-3)}\mathcal{Z}_iB_jB_i^2\\&+\rho_{1,1,0}^{(i,j,-3)}B_iB_jB_i\mathcal{Z}_i + \rho_{1,1,1}^{(i,j,-3)}\mathcal{Z}_iB_iB_jB_i + \rho_{2,0,0}^{(i,j,-3)}B_i^2B_j\mathcal{Z}_i + \rho_{2,0,1}^{(i,j,-3)}\mathcal{Z}_iB_i^2B_j \\ & + \sigma_{0,0}^{(i,j,-3)}\mathcal{W}_{ij}K_j\mathcal{Z}_i + \sigma_{0,1}^{(i,j,-3)}\mathcal{Z}_i\mathcal{W}_{ij}K_j + \sigma_{2,0}^{(i,j,-3)}\mathcal{W}_{ij}K_jB_i^2,
\end{align*}
with \(\rho_{m,m',t}^{(i,j,-3)}\) and \(\sigma_{m,t}^{(i,j,-3)}\) as in Tables \ref{Table of rho_m,m',t for a_ij = -3} and \ref{Table of sigma_m,t for a_ij = -3}.

\renewcommand\figurename{Table}
\begin{figure}[h] 
	\renewcommand\figurename{Table}
	\centering
	\begin{tikzpicture}
	\draw[line width = 1pt] (0,0) -- (13.6,0);
	\draw[line width = 1pt] (1,0.75) -- (1,-2.4);
	\draw (0,0.75) -- (1,0);
	\node at (0.25,0.3){\(m\)};
	\node at (0.75,0.5){\(m'\)};
	\node at (0.5,-0.3){\(0\)};
	\node at (0.5,-0.9){\(1\)};
	\node at (0.5,-1.5){\(2\)};
	\node at (0.5,-2.1){\(3\)};
	\node at (2.25,0.35){\(0\)};
	\node at (5.5,0.35){\(1\)};
	\node at (9,0.35){\(2\)};
	\node at (12,0.35){\(3\)};
	\node at (2.25,-0.3) {\(0\)};
	\node at (2.25,-0.9) {\(-\rho_{0,1}^{(i,j,-4)}\)};
	\node at (2.25,-1.5) {\(0\)};
	\node at (2.25,-2.1) {\(-\rho_{0,3}^{(i,j,-4)}\)};
	\node at (5.5,-0.3) {\(c_i^2q_i^2[2]_{q_i}^2[4]_{q_i^2}^2\)};
	\node at (5.5,-0.9) {\(0\)};
	\node at (5.5,-1.5) {\(-\rho_{1,2}^{(i,j,-4)}\)};
	\node at (9,-0.3) {\(0\)};
	\node at (9,-0.9) {\(c_iq_i[2]_{q_i}^2[3]_{q_i}[5]_{q_i}\)};
	\node at (12,-0.3) {\(c_iq_i([2]_{q_i}^2+[4]_{q_i}^2)\)};
	\draw (0,0.75) rectangle (3.5,-2.4);
	\draw (0,0.75) rectangle (7.6,-1.8);
	\draw (0,0.75) rectangle (10.4,-1.2);
	\draw (0,0.75) rectangle (13.6,-0.6);
	\end{tikzpicture}
	\caption{Structure constants \({\rho}_{m,m'}^{(i,j,a_{ij})}\) for \(a_{ij} = -4\)}
	\label{Table of rho_m,m' for a_ij = -4}
\end{figure}

\renewcommand\figurename{Table}
\begin{figure}[h] 
	\renewcommand\figurename{Table}
	\centering
	\begin{tikzpicture}
	\draw[line width = 1pt] (-0.9,0) -- (10.5,0);
	\draw[line width = 1pt] (1,0.75) -- (1,-4.4);
	\draw (0,0.75) -- (1,0);
	\node at (-0.25,0.3){\((m,m')\)};
	\node at (0.75,0.5){\(t\)};
	\node at (0,-0.6){\((0,0)\)};
	\node at (0,-1.7){\((0,2)\)};
	\node at (0,-2.8){\((1,1)\)};
	\node at (0,-3.9){\((2,0)\)};
	\node at (2.6,0.35){\(0\)};
	\node at (5.8,0.35){\(1\)};
	\node at (9,0.35){\(2\)};
	\node at (2.6,-0.6) {\(-c_i^2q_i^6\dfrac{[3]_{q_i}}{(q_i-q_i^{-1})^2}\)};
	\node at (5.8,-0.6) {\(c_i^2q_i^2\dfrac{[3]_{q_i}(q_i^2+q_i^{-2})}{(q_i-q_i^{-1})^2}\)};
	\node at (9,-0.6) {\(-c_i^2q_i^{-2}\dfrac{[3]_{q_i}}{(q_i-q_i^{-1})^2}\)};
	\node at (2.6,-1.7) {\(c_iq_i^2\dfrac{2+q_i^2[2]_{q_i}^2}{q_i-q_i^{-1}} \)};
	\node at (5.8,-1.7) {\(-c_i\dfrac{[3]_{q_i}(q_i^2+q_i^{-2})}{q_i-q_i^{-1}} \)};
	\node at (2.6,-2.8) {\(-c_iq_i^2\dfrac{[4]_{q_i}(q_i^2+2)}{q_i-q_i^{-1}} \)};
	\node at (5.8,-2.8) {\(c_i\dfrac{[4]_{q_i}(q_i^{-2}+2)}{q_i-q_i^{-1}} \)};
	\node at (2.6,-3.9) {\(c_iq_i^2\dfrac{[3]_{q_i}(q_i^2+q_i^{-2})}{q_i-q_i^{-1}} \)};
	\node at (5.8,-3.9) {\(-c_i\dfrac{2+q_i^{-2}[2]_{q_i}^2}{q_i-q_i^{-1}} \)};
	\draw (4.1,0.75) -- (4.1,-4.4);
	\draw (7.5,0.75) -- (7.5,-4.4);
	\draw (-0.9,-2.2) -- (7.5,-2.2);
	\draw (-0.9,-1.2) -- (7.5,-1.2);
	\draw (-0.9,-3.3) -- (7.5,-3.3);
	\draw (-0.9,0.75) rectangle (7.5,-4.4);
	\draw (-0.9,0.75) rectangle (10.5,-1.2);
	\end{tikzpicture}
	\caption{Structure constants \({\rho}_{m,m',t}^{(i,j,a_{ij})}\) for \(a_{ij} = -3\)}
	\label{Table of rho_m,m',t for a_ij = -3}
\end{figure}

\begin{figure}[h] 
	\renewcommand\figurename{Table}
	\centering
		\begin{tikzpicture}
		\draw[line width = 1pt] (0,0) -- (11.2,0);
		\draw[line width = 1pt] (1,0.75) -- (1,-2.4);
		\draw (0,0.75) -- (1,0);
		\node at (0.25,0.3){\(m\)};
		\node at (0.75,0.5){\(t\)};
		\node at (0.5,-0.6){\(0\)};
		\node at (3.7,0.35){\(0\)};
		\node at (8.8,0.35){\(1\)};
		\node at (0.5,-1.8){\(2\)};
		\node at (3.7,-0.6) {\(-c_i^2q_i^{2}\dfrac{[3]_{q_i}[4]_{q_i}}{(q_i-q_i^{-1})(q_j-q_j^{-1})}\)};
		\node at (3.7,-1.8) {\(c_iq_i^{-5}[2]_{q_i}[3]_{q_i}^2[4]_{q_i}\dfrac{(q_i-q_i^{-1})^2}{q_j-q_j^{-1}}\)};
		\node at (8.8,-0.6) {\(c_i^2q_i^{-4}\dfrac{[3]_{q_i}[4]_{q_i}}{(q_i-q_i^{-1})(q_j-q_j^{-1})}\)};
		\draw (0,0.75) rectangle (6.4,-2.4);
		\draw (0,0.75) rectangle (11.2,-1.2);
		\end{tikzpicture}
	\caption{Structure constants \({\sigma}_{m,t}^{(i,j,a_{ij})}\) for \(a_{ij} = -3\)}
	\label{Table of sigma_m,t for a_ij = -3}
\end{figure}

Theorems \ref{theorem F_ij(B_i,B_j) for Case 1} and \ref{theorem F_ij(B_i,B_j) Case 2}, together with the previously obtained Theorems \ref{theorem these are defining relations}, \ref{theorem F_ij(B_i,B_j) = 0} and \ref{theorem Balagovic} now yield a complete set of defining relations for the quantum symmetric pair coideal subalgebras \(\Bcs\).

\section{Alternative expressions for Case 1}
\label{Section Alternative expressions for Case 1}

In this section, we will derive alternative expressions for the polynomial \(C_{ij}(\bc)\) in Case 1. We will start from a result by Chen, Lu and Wang. In \cite{Chen&Lu&Wang-2019} these authors provide defining relations of \(q\)-Serre type for what they call \(\iota\)-quantum groups, which are in fact quasi-split quantum symmetric pair coideal subalgebras, coinciding with the algebras \(B_{\bc,\bs}\) in the special case \(X = \emptyset\). These correspond to Satake diagrams without black nodes. Since \(\mathcal{Z}_i = -1\) in this situation, the polynomial \(C_{ij}(\bc)\) will be given by
\begin{equation}
\label{F_ij(B_i,B_j) for quasi-split case}
F_{ij}(B_i,B_j) = C_{ij}(\bc) = \sum_{m=0}^{-1-a_{ij}}\sum_{m' = 0}^{-1-a_{ij}-m}(-1)^{\frac{1-a_{ij}-m-m'}{2}}\rho_{m,m'}^{(i,j,a_{ij})}B_i^mB_jB_i^{m'},
\end{equation}
if \(\tau(i) = i\), as follows from Corollary \ref{cor F_ij(B_i,B_j) Case 1 with epsilon and P}, and where the structure constants \(\rho_{m,m'}^{(i,j,a_{ij})}\) were obtained in Theorem \ref{theorem F_ij(B_i,B_j) for Case 1}. In this section we will derive equivalent expressions for \(\rho_{m,m'}^{(i,j,a_{ij})}\), based on the results of \cite{Chen&Lu&Wang-2019}. These expressions will also be valid beyond the quasi-split case. Indeed, our result (\ref{rho_m,m' def}) shows that \(\rho_{m,m'}^{(i,j,a_{ij})}\) is independent of \(X\) and can be obtained solely from the \(\Uqgp\)-relations (\ref{U_q(g) relations})--(\ref{U_q(g) relations 2})--(\ref{U_q(g) relations 3}). Hence the expressions for \(\rho_{m,m'}^{(i,j,a_{ij})}\) we will derive in this section will be valid not only for \(X = \emptyset\), but for any admissible pair \((X,\tau)\) provided we restrict to Case 1, i.e.\ \(\tau(i) = i\in I\setminus X\) and \(j\in I\setminus X\) distinct from \(i\). 

Before we can state the result from \cite{Chen&Lu&Wang-2019}, we need to introduce the following notation. 

\begin{definition}[{\cite[formulae (3.2)--(3.3)]{Chen&Lu&Wang-2019}}]
	\label{def iota divided powers}
	For any \(i\in I\) and \(m\in\N\) one defines the \(\iota\)-divided powers of \(B_i\) as the elements
	\begin{align}
	\label{def B_i,0 divided power}
	B_{i,0}^{(m)} &= 
	\dfrac{B_i^{m_p}}{[m]_{q_i}!}\prod\limits_{k = 1}^{m_e}\left(B_i^2+q_ic_i[2(k-1+m_p)]_{q_i}^2 \right), \\
	\label{def B_i,1 divided power}
	B_{i,1}^{(m)} &= 
	\dfrac{B_i^{m_p}}{[m]_{q_i}!}\prod\limits_{k = 1}^{m_e}\left(B_i^2+q_ic_i[2k-1]_{q_i}^2 \right),
	\end{align}
	where we have again used the notation (\ref{even and odd part}).
\end{definition}

Using Lusztig's theory of modified quantum groups \cite[Section 23.1]{Lusztig-1994} and a class of intricate \(q\)-binomial identities, Chen, Lu and Wang were able to prove a result, which, translated to our notations, can be formulated as follows.

\begin{theoremKolb}[{\cite[Theorem 3.1]{Chen&Lu&Wang-2019}}]
	\label{theorem CLW}
	Consider the quantum symmetric pair coideal algebra \(B_{\bc,\bs}\) corresponding to an admissible pair \((X = \emptyset,\tau)\). For any \(i\in I\) satisfying \(\tau(i) = i\) and any \(j\in I\) distinct from \(i\), the \(\iota\)-divided powers satisfy the relations
	\begin{equation}
	\label{divided powers relation with 0}
	\sum_{m = 0}^{1-a_{ij}}(-1)^m B_{i,(a_{ij})_p}^{(m)}B_jB_{i,0}^{(1-a_{ij}-m)} = 0,
	\end{equation}
	and
	\begin{equation}
	\label{divided powers relation with 1}
	\sum_{m = 0}^{1-a_{ij}}(-1)^m B_{i,1-(a_{ij})_p}^{(m)}B_jB_{i,1}^{(1-a_{ij}-m)} = 0.
	\end{equation}
\end{theoremKolb}

The relations (\ref{divided powers relation with 0}) and (\ref{divided powers relation with 1}) can be rewritten in the form \(F_{ij}(B_i,B_j) = C_{ij}(\bc)\), where \(C_{ij}(\bc)\) is an explicit polynomial in \(\sum_{J\in\mathcal{J}_{i,j}}\K(q)B_J\). The computation of this polynomial will be the subject of the following subsection.

\subsection{Quantum Serre relations from $\iota$-divided powers}
\label{Subsection iota-divided powers}

Let us now fix \(\tau\in\mathrm{Aut}(A,\emptyset)\) such that \((X=\emptyset,\tau)\) is an admissible pair, and fix also a corresponding quantum symmetric pair coideal subalgebra \(\Bcs\). In the present section, we will rewite the relations (\ref{divided powers relation with 0}) and (\ref{divided powers relation with 1}) as inhomogeneous quantum Serre relations, so as to derive two new expressions for the structure constants \(\rho_{m,m'}^{(i,j,a_{ij})}\) for Case 1. Let us start by introducing the following notation.

\begin{definition}
	\label{definition alpha_k,N}
	Let \(k\in\N\), \(N\in\N\cup\{-1\}\), \(s\in \{0,1\}\) and \(i\in I\). We will denote by \(\alpha^{(s,i)}_{k,N}\) the following elements of \(\K(q)\):
	\[
	\alpha_{k,N}^{(s,i)} = \left\{ \begin{array}{ll}\sum\limits_{\substack{\ell_1,\ell_2,\dots,\ell_k = 1-s \\ \ell_1 < \ell_2< \dots < \ell_k}}^{N}[2\ell_1+s]_{q_i}^2[2\ell_2+s]_{q_i}^2\dots[2\ell_k+s]_{q_i}^2,\quad & \mathrm{for}\ 1\leq k\leq N+s \\
	1 & \mathrm{for}\ k = 0, \\
	0 & \mathrm{else}.
	\end{array}\right.
	\]
\end{definition}

These \(\alpha_{k,N}^{(s,i)}\) arise as coefficients when expanding the \(\iota\)-divided powers from Definition \ref{def iota divided powers} as polynomials in \(B_i\).

\begin{lemma}
	\label{lemma divided powers reformulated}
	For any \(i\in I\), \(s\in\{0,1\}\) and \(r\in \N\), one can write
	\[
	B_{i,s}^{(r)} = \frac{1}{[r]_{q_i}!}\sum_{k = 0}^{r_e}(q_ic_i)^k\alpha_{k,r_e+r_p(1-s)-1}^{(s,i)}B_i^{r-2k}.
	\]
\end{lemma}
\begin{proof}
	Expanding (\ref{def B_i,1 divided power}) distributively, it is clear that
	\[
	B_{i,1}^{(r)} = \begin{cases}
	\dfrac{1}{[r]_{q_i}!}\sum\limits_{k = 0}^{\frac{r}{2}}(q_ic_i)^k\alpha^{(1,i)}_{k,\frac{r}{2}-1}B_i^{r-2k} \quad & \mathrm{for}\ r\ \mathrm{even}, \\
	\dfrac{1}{[r]_{q_i}!}\sum\limits_{k = 0}^{\frac{r-1}{2}}(q_ic_i)^k\alpha^{(1,i)}_{k,\frac{r-1}{2}-1}B_i^{r-2k} \quad & \mathrm{for}\ r\ \mathrm{odd}, 
	\end{cases}
	\]
	in agreement with the proposed formula. Similarly, for \(r\) odd one has
	\[
	B_{i,0}^{(r)} = \frac{1}{[r]_{q_i}!}\sum_{k = 0}^{\frac{r-1}{2}}(q_ic_i)^k\alpha^{(0,i)}_{k,\frac{r-1}{2}}B_i^{r-2k},
	\]
	whereas for \(r\) even, we have
	\[
	B_{i,0}^{(r)} = \frac{1}{[r]_{q_i}!}\sum_{k = 0}^{\frac{r}{2}}(q_ic_i)^k\gamma_{k,r}^{(i)} B_i^{r-2k},
	\]
	where
	\[
	\gamma_{k,r}^{(i)} = \sum_{\substack{\ell_1,\ell_2,\dots,\ell_k = 0\\ \ell_1<\dots<\ell_k}}^{\frac{r}{2}-1}[2\ell_1]_{q_i}^2[2\ell_2]_{q_i}^2\dots[2\ell_k]_{q_i}^2.
	\]
	But of course, since \([0]_{q_i} = 0\), we have that
	\[
	\gamma_{k,r}^{(i)} = \sum_{\substack{\ell_1,\ell_2,\dots,\ell_k = 1\\\ell_1<\dots<\ell_k}}^{\frac{r}{2}-1}[2\ell_1]_{q_i}^2[2\ell_2]_{q_i}^2\dots[2\ell_k]_{q_i}^2 = \alpha_{k,\frac{r}{2}-1}^{(0,i)},
	\]
	which again agrees with the statement of the lemma.
\end{proof}

In the upcoming proofs, we will often be required to switch the order of summation in a particular kind of nested sums. Below, we propose a general strategy for this resummation.

\begin{lemma}
	\label{lemma switching order of summation}
	Let \(f\) be any function of three discrete variables \(k, \ell\) and \(m\), and let \(N\) be any natural number, then one has
	\[
	\sum_{m = 0}^N\sum_{k = 0}^m\sum_{\ell = 0}^{N-m}f(k,\ell,m) = \sum_{k = 0}^N\sum_{\ell = 0}^{N-k}\sum_{m = 0}^{\ell}f(m,\ell-m,m+k).
	\]
\end{lemma}
\begin{proof}
	We will derive this identity in several steps, which are explained below.
	\begin{align*}
	&\sum_{m = 0}^N\sum_{k = 0}^m\sum_{\ell = 0}^{N-m}f(k,\ell,m) & \overset{(1)}{=}& \sum_{m = 0}^N\sum_{k = 0}^m\sum_{\ell = 0}^{N-m}f(m-k,N-m-\ell,m)\\
	\overset{(2)}{=}& \sum_{k = 0}^N\sum_{m = k}^N\sum_{\ell = 0}^{N-m}f(m-k,N-m-\ell,m) &
	\overset{(3)}{=}&\sum_{k = 0}^N\sum_{m = 0}^{N-k}\sum_{\ell = 0}^{m}f(N-m-k,m-\ell,N-m) \\
	\overset{(4)}{=}&\sum_{k = 0}^N\sum_{\ell = 0}^{N-k}\sum_{m = \ell}^{N-k}f(N-m-k,m-\ell,N-m) &
	\overset{(5)}{=}&\sum_{k = 0}^N\sum_{\ell = 0}^{N-k}\sum_{m = 0}^{\ell}f(m,\ell-m,m+k).
	\end{align*}
	\begin{description}
		\item[\(\boldsymbol{(1)}\)] Replace \(k\) by the new summation index \(k' = m-k\), replace \(\ell\) by \(\ell' = N-m-\ell\) and rename \(k'\to k\), \(\ell'\to\ell\).
		\item[\(\boldsymbol{(2)}\)] Switch the summation order of the sums over \(m\) and \(k\).
		\item[\(\boldsymbol{(3)}\)] Replace \(m\) by the new summation index \(m' = N-m\) and then rename \(m' \to m\).
		\item[\(\boldsymbol{(4)}\)] Switch the summation order of the sums over \(m\) and \(\ell\).
		\item[\(\boldsymbol{(5)}\)] Replace \(\ell\) by the new summation index \(\ell' = N-k-\ell\), replace \(m\) by \(m' = N-m-k\) and rename \(\ell'\to\ell\), \(m'\to m\).
	\end{description}
\end{proof}

We will now rewrite the results of Theorem \ref{theorem CLW} using Lemma \ref{lemma divided powers reformulated}.

\begin{proposition}
	\label{prop inhomogeneous quantum Serre with 0}
	The relations (\ref{divided powers relation with 0})--(\ref{divided powers relation with 1}) can equivalently be expressed as
	\begin{equation}
	\label{inhomogeneous quantum Serre with 0}
	F_{ij}(B_i,B_j) = \sum_{m = 0}^{-1-a_{ij}}\sum_{m' = 0}^{-1-a_{ij}-m}(a_{ij}+m+m')_p(-1)^{a_{ij}+m}(q_ic_i)^{\frac{1-a_{ij}-m-m'}{2}}\Theta_{m,m'}^{(0,i,a_{ij})}B_i^mB_jB_i^{m'},
	\end{equation}
	where
	\begin{equation}
	\label{Theta 0}
	\Theta_{m,m'}^{(0,i,a_{ij})} = \sum_{r = 0}^{\frac{1-a_{ij}-m-m'}{2}}\begin{bmatrix}
	1-a_{ij}\\ m+2r
	\end{bmatrix}_{q_i}\alpha_{r,r+m_e+m_p-1}^{(0,i)}\alpha_{\frac{1-a_{ij}-m-m'}{2}-r,-(a_{ij})_e-(a_{ij})_p-r-m_e-m_p}^{((a_{ij})_p,i)},
	\end{equation}
	with \(\alpha_{k,N}^{(s,i)}\) as in Definition \ref{definition alpha_k,N}.
\end{proposition}
\begin{proof}
	Let us consider the case where \(a_{ij}\) is even.
	We will start by splitting the sum over \(m\) in (\ref{divided powers relation with 0}) into a sum over \(m\) even and one over \(m\) odd:
	\[
	\sum_{m = 0}^{-\frac{a_{ij}}{2}} B_{i,0}^{(2m)}B_jB_{i,0}^{(1-a_{ij}-2m)} - \sum_{m = 0}^{-\frac{a_{ij}}{2}}B_{i,0}^{(2m+1)}B_jB_{i,0}^{(-a_{ij}-2m)} = 0.
	\]
	Substituting the expressions for \(B_{i,0}^{(r)}\) obtained in Lemma \ref{lemma divided powers reformulated}, we find
	\begin{align*}
	& \sum_{m = 0}^{-\frac{a_{ij}}{2}}\sum_{k = 0}^m\sum_{\ell = 0}^{-\frac{a_{ij}}{2}-m}\frac{(q_ic_i)^{k+\ell}}{[2m]_{q_i}![1-a_{ij}-2m]_{q_i}!}\alpha_{k,m-1}^{(0,i)}\alpha_{\ell,-\frac{a_{ij}}{2}-m}^{(0,i)}B_i^{2m-2k}B_jB_i^{1-a_{ij}-2m-2\ell}\\
	-& \sum_{m = 0}^{-\frac{a_{ij}}{2}}\sum_{k = 0}^m\sum_{\ell = 0}^{-\frac{a_{ij}}{2}-m}\frac{(q_ic_i)^{k+\ell}}{[2m+1]_{q_i}![-a_{ij}-2m]_{q_i}!}\alpha_{k,m}^{(0,i)}\alpha_{\ell,-\frac{a_{ij}}{2}-m-1}^{(0,i)}B_i^{2m-2k+1}B_jB_i^{-a_{ij}-2m-2\ell} = 0.
	\end{align*}
	Multiplying both sides with \([1-a_{ij}]_{q_i}!\) and applying Lemma \ref{lemma switching order of summation}, this becomes
	\begin{align*}
	&\sum_{k = 0}^{-\frac{a_{ij}}{2}}\sum_{\ell = 0}^{-\frac{a_{ij}}{2}-k}(q_ic_i)^{\ell}\left(\sum_{m = 0}^{\ell}\begin{bmatrix}
	1-a_{ij}\\ 2m+2k
	\end{bmatrix}_{q_i}\alpha^{(0,i)}_{m,m+k-1}\alpha^{(0,i)}_{\ell-m,-\frac{a_{ij}}{2}-m-k}\right)B_i^{2k}B_jB_i^{1-a_{ij}-2k-2\ell}\\
	-& \sum_{k = 0}^{-\frac{a_{ij}}{2}}\sum_{\ell = 0}^{-\frac{a_{ij}}{2}-k}(q_ic_i)^{\ell}\left(\sum_{m = 0}^{\ell}\begin{bmatrix}
	1-a_{ij}\\ 2m+2k+1
	\end{bmatrix}_{q_i}\alpha^{(0,i)}_{m,m+k}\alpha^{(0,i)}_{\ell-m,-\frac{a_{ij}}{2}-m-k-1}\right)B_i^{2k+1}B_jB_i^{-a_{ij}-2k-2\ell} = 0.
	\end{align*}
	Referring to the notation (\ref{Theta 0}), we may write the terms between brackets above as \(\Theta_{2k,1-a_{ij}-2k-2\ell}^{(0,i,a_{ij})}\) and \(\Theta_{2k+1,-a_{ij}-2k-2\ell}^{(0,i,a_{ij})}\) respectively. Replacing then \(2k\) by \(m\) in the first sum and \(2k+1\) by \(m\) in the second, this becomes
	\begin{align}
	\label{intermediary sum 0}
	\begin{split}
	&\sum_{\substack{m = 0 \\ m\ \mathrm{even}}}^{1-a_{ij}}\sum_{\ell = 0}^{(1-a_{ij}-m)_e}(q_ic_i)^{\ell}\Theta_{m,1-a_{ij}-m-2\ell}^{(0,i,a_{ij})}B_i^{m}B_jB_i^{1-a_{ij}-m-2\ell}\\ - & \sum_{\substack{m = 0 \\ m\ \mathrm{odd}}}^{1-a_{ij}}\sum_{\ell = 0}^{(1-a_{ij}-m)_e}(q_ic_i)^{\ell}\Theta_{m,1-a_{ij}-m-2\ell}^{(0,i,a_{ij})}B_i^{m}B_jB_i^{1-a_{ij}-m-2\ell} \\
	=& \sum_{m = 0}^{1-a_{ij}}\sum_{\ell = 0}^{(1-a_{ij}-m)_e}(-1)^m(q_ic_i)^{\ell}\Theta_{m,1-a_{ij}-m-2\ell}^{(0,i,a_{ij})}B_i^{m}B_jB_i^{1-a_{ij}-m-2\ell} = 0.
	\end{split}
	\end{align}
	Now observe that the term corresponding to \(\ell = 0\) can be written as
	\begin{align*}
	& \sum_{m = 0}^{1-a_{ij}}(-1)^m\Theta_{m,1-a_{ij}-m}^{(0,i,a_{ij})}B_i^mB_jB_i^{1-a_{ij}-m}
	= (-1)^{1+a_{ij}}\sum_{m = 0}^{1-a_{ij}} (-1)^{m}\begin{bmatrix}
	1-a_{ij}\\ m
	\end{bmatrix}_{q_i}B_i^{1-a_{ij}-m}B_jB_i^{m}\\ = &\  (-1)^{1+a_{ij}}F_{ij}(B_i,B_j),
	\end{align*}
	where we have replaced \(m\) by the new summation index \(m' = 1-a_{ij}-m\) for the first equality, which we have thereafter renamed to \(m\) again, and where we have used the fact that 
	\(
	\Theta_{1-a_{ij}-m,m}^{(0,i,a_{ij})} = \begin{bmatrix}
	1-a_{ij}\\ m
	\end{bmatrix}_{q_i}
	\) by (\ref{prop q-binom symbol}).
	Consequently, when separating the term corresponding to \(\ell = 0\) in (\ref{intermediary sum 0}), we obtain
	\[
	F_{ij}(B_i,B_j) = (-1)^{a_{ij}}\sum_{m = 0}^{1-a_{ij}}\sum_{\ell = 1}^{(1-a_{ij}-m)_e}(-1)^m(q_ic_i)^{\ell}\Theta_{m,1-a_{ij}-m-2\ell}^{(0,i,a_{ij})}B_i^{m}B_jB_i^{1-a_{ij}-m-2\ell}
	\]
	Now observe that when \(m\) equals \(-a_{ij}\) or \(1-a_{ij}\), the range of the second summation index \(\ell\) is empty. Hence the sum over \(m\) runs in fact from 0 to \(-1-a_{ij}\). Moreover, we may replace \(\ell\) by the new summation index \(m' = 1-a_{ij}-m-2\ell\), which runs over \(\{0,2,\dots,-1-a_{ij}-m\}\) if \(1-a_{ij}-m\) is even and over \(\{1,3,\dots,-1-a_{ij}-m\} \) if \(1-a_{ij}-m\) is odd, hence over \(\{0,1,\dots,-1-a_{ij}-m\}\) after multiplying the summandum with \((a_{ij}+m+m')_p\). This leads us to
	\[
	F_{ij}(B_i,B_j)
	= \sum_{m = 0}^{-1-a_{ij}}\sum_{m' = 0}^{-1-a_{ij}-m}(a_{ij}+m+m')_p(-1)^{a_{ij}+m}(q_ic_i)^{\frac{1-a_{ij}-m-m'}{2}}\Theta_{m,m'}^{(0,i,a_{ij})}B_i^{m}B_jB_i^{m'},
	\]
	as was to be proven. The statement for \(a_{ij}\) odd follows analogously, starting from the relation (\ref{divided powers relation with 1}).
\end{proof}

In a similar fashion, one can combine the relation (\ref{divided powers relation with 0}) for \(a_{ij}\) odd with the relation (\ref{divided powers relation with 1}) for \(a_{ij}\) even. This gives rise to the following expressions. 

\begin{proposition}
	\label{prop inhomogeneous quantum Serre with 1}
	The relations (\ref{divided powers relation with 0})--(\ref{divided powers relation with 1}) can equivalently be expressed as
	\begin{equation}
	\label{inhomogeneous quantum Serre with 1}
	F_{ij}(B_i,B_j) = \sum_{m = 0}^{-1-a_{ij}}\sum_{m' = 0}^{-1-a_{ij}-m}(a_{ij}+m+m')_p(-1)^{a_{ij}+m}(q_ic_i)^{\frac{1-a_{ij}-m-m'}{2}}\Theta_{m,m'}^{(1,i,a_{ij})}B_i^mB_jB_i^{m'},
	\end{equation}
	where
	\begin{equation}
	\label{Theta 1}
	\Theta_{m,m'}^{(1,i,a_{ij})} =\sum_{r = 0}^{\frac{1-a_{ij}-m-m'}{2}}\begin{bmatrix}
	1-a_{ij}\\ m+2r
	\end{bmatrix}_{q_i}\alpha_{r,r+m_e-1}^{(1,i)}\alpha_{\frac{1-a_{ij}-m-m'}{2}-r,-(a_{ij})_e-r-m_e-1}^{(1-(a_{ij})_p,i)}.
	\end{equation}
\end{proposition}

Comparing the relations (\ref{inhomogeneous quantum Serre with 0}) and (\ref{inhomogeneous quantum Serre with 1}) with (\ref{F_ij(B_i,B_j) for quasi-split case}), we obtain alternative expressions for the structure constants \(\rho_{m,m'}^{(i,j,a_{ij})}\). As explained in the introduction of Section \ref{Section Alternative expressions for Case 1}, these will not only be valid for the quasi-split case, but for any admissible pair, provided we restrict to Case 1. Hence from now on we may again assume \((X,\tau)\) to be an arbitrary admissible pair and consider the corresponding quantum symmetric pair coideal subalgebra \(\Bcs\).

\begin{theorem}
	\label{theorem Case 1 with CLW}
	For any distinct \(i,j\in I\setminus X\) such that \(\tau(i) = i\), one has
	\[
	F_{ij}(B_i,B_j) = C_{ij}(\bc) = \sum_{m=0}^{-1-a_{ij}}\sum_{m' = 0}^{-1-a_{ij}-m}\rho_{m,m'}^{(i,j,a_{ij})} \mathcal{Z}_i^{\frac{1-a_{ij}-m-m'}{2}}B_i^mB_jB_i^{m'},
	\]
	where the structure constants are given by
	\begin{equation}
	\label{structure constants with 0}
	\rho_{m,m'}^{(i,j,a_{ij})} = (a_{ij}+m+m')_p(-1)^{a_{ij}+m}(-q_ic_i)^{\frac{1-a_{ij}-m-m'}{2}}\Theta_{m,m'}^{(s,i,a_{ij})}
	\end{equation}
	with \(s\in\{0,1\}\) and where we have used the notations (\ref{Theta 0}) and (\ref{Theta 1}).
\end{theorem}
\begin{proof}
	We will use the same strategy as in the proof of \cite[Proposition 6.1]{Kolb-2014}. Assume first that \(X = \emptyset\). If we write
	\[
	\omega_{m,m'}^{(i,j,a_{ij})} = (a_{ij}+m+m')_p(-1)^{a_{ij}+m}(q_ic_i)^{\frac{1-a_{ij}-m-m'}{2}}\Theta_{m,m'}^{(s,i,a_{ij})},
	\]
	with \(s\in\{0,1\}\), then comparison of (\ref{inhomogeneous quantum Serre with 0}) and (\ref{inhomogeneous quantum Serre with 1}) with (\ref{F_ij(B_i,B_j) for quasi-split case}) yields
	\begin{equation}
	\label{sigma minus}
	\sum_{m=0}^{-1-a_{ij}}\sum_{m' = 0}^{-1-a_{ij}-m}\left((-1)^{\frac{1-a_{ij}-m-m'}{2}}\rho_{m,m'}^{(i,j,a_{ij})}-\omega_{m,m'}^{(i,j,a_{ij})}\right)B_i^mB_jB_i^{m'} = 0.
	\end{equation}
	Separating the term \(F_i\) in each \(B_i = F_i+c_i\theta_q(F_iK_i)K_i^{-1}+s_iK_i^{-1} = F_i-c_iE_{\tau(i)}K_i^{-1}+s_iK_i^{-1}\) and the \(F_j\) in \(B_j\), the relation (\ref{sigma minus}) asserts \(\mathcal{F}_{i,j} + \mathcal{D}_{i,j} = 0\), where
	\[
	\mathcal{F}_{i,j} = \sum_{m=0}^{-1-a_{ij}}\sum_{m' = 0}^{-1-a_{ij}-m}\left((-1)^{\frac{1-a_{ij}-m-m'}{2}}\rho_{m,m'}^{(i,j,a_{ij})}-\omega_{m,m'}^{(i,j,a_{ij})}\right)F_i^mF_jF_i^{m'}
	\]
	and \(\mathcal{D}_{i,j}\) lies in the set \(\mathscr{E}_{i,j}\) of \(\K(q)\)-linear combination of monomials in \(\Uqgp\) containing at most \(-1-a_{ij}\) factors \(F_i\), and either one factor \(K_j^{-1}\), or one factor \(F_j\) and at least one factor \(K_i^{-1}\). Since the \(\Uqgp\)-relations (\ref{U_q(g) relations})--(\ref{U_q(g) relations 2})--(\ref{U_q(g) relations 3}) imply \(\mathscr{E}_{i,j}\cap U^- = \{0\}\), both \(\mathcal{D}_{i,j}\) and \(\mathcal{F}_{i,j}\) must vanish. The assertion \(\mathcal{F}_{i,j} = 0\) is a polynomial equation of degree 1 in \(F_j\) and at most of degree \(-1-a_{ij}\) in \(F_i\). But such a polynomial must have trivial coefficients, since the lowest degree \(\K(q)\)-linear combination of \(F_j\) and powers of \(F_i\) with non-trivial coefficients that vanishes, is precisely the quantum Serre polynomial \(F_{ij}(F_i,F_j)\), which is of degree \(1-a_{ij}\) in \(F_i\).
	Hence we find
	\[
	(-1)^{\frac{1-a_{ij}-m-m'}{2}}\rho_{m,m'}^{(i,j,a_{ij})}=\omega_{m,m'}^{(i,j,a_{ij})},
	\]
	for any \(m,m'\) with \(m+m'\leq -1-a_{ij}\). This holds for the special case \(X = \emptyset\), and since \(\rho_{m,m'}^{(i,j,a_{ij})}\) is independent of \(X\) as explained above, this establishes the same relation for admissible pairs with \(X\neq \emptyset\).
\end{proof}

\begin{remark}
	It follows from Propositions \ref{prop inhomogeneous quantum Serre with 0} and \ref{prop inhomogeneous quantum Serre with 1} that \(\Theta_{m,m'}^{(0,i,a_{ij})} = \Theta_{m,m'}^{(1,i,a_{ij})}\) for any \(m,m'\in\N\) with \(m+m'\leq -1-a_{ij}\) and any distinct \(i,j\in I\setminus X\) with \(\tau(i) = i\). Hence the expressions (\ref{Theta 0}) and (\ref{Theta 1}) must be equal, which determines a non-trivial identity of \(q\)-binomial type.
\end{remark}

To conclude, we will show that the structure constants \(\rho_{m,m'}^{(i,j,a_{ij})}\) exhibit certain symmetry properties, as suggested by the values in Table \ref{Table of rho_m,m'}. In practical calculations, this significantly reduces the number of couples \((m,m')\) for which the structure constants must be computed.

\begin{proposition}
	\label{structure constants symmetry}
	The structure constants \(\rho_{m,m'}^{(i,j,a_{ij})}\) are symmetric in \(m\) and \(m'\) if \(a_{ij}\) is odd and antisymmetric if \(a_{ij}\) is even. In other words:
	\[
	\rho_{m,m'}^{(i,j,a_{ij})} = (-1)^{1-a_{ij}}\rho_{m',m}^{(i,j,a_{ij})}.
	\]
\end{proposition}
\begin{proof}
	We will treat the case \(a_{ij}\) odd, which is the most subtle case in some sense. The statement is trivial for \(m+m'\) odd, since in this case \(\rho_{m,m'}^{(i,j,a_{ij})}\) will vanish, because of the factor \((a_{ij}+m+m')_p\) in (\ref{structure constants with 0}). So we may assume \(m+m'\) to be even. Let us start by observing that \(\rho_{m',m}^{(i,j,a_{ij})}\) yields 
	\[ (-1)^{a_{ij}+m'}(-q_ic_i)^{\frac{1-a_{ij}-m-m'}{2}}\sum_{r = 0}^{\frac{1-a_{ij}-m-m'}{2}}\begin{bmatrix}
	1-a_{ij}\\ m'+2r
	\end{bmatrix}_{q_i}\alpha_{r,r+m'_e+m'_p-1}^{(0,i)}\alpha_{\frac{1-a_{ij}-m-m'}{2}-r,\frac{-a_{ij}-1}{2}-r-m'_e-m'_p}^{(1,i)}
	\]
	by (\ref{structure constants with 0}) with \(s = 0\).
	Since \(m+m'\) is even, we have \((-1)^{m'} = (-1)^m\). 
	Moreover, we can rewrite the sum above using a new summation index \(r' = \frac{1-a_{ij}-m-m'}{2}-r\), which we thereafter rename to \(r\) again. This way, \(\rho_{m',m}^{(i,j,a_{ij})}\) becomes
	\[ 
	(-1)^{a_{ij}+m}(-q_ic_i)^{\frac{1-a_{ij}-m-m'}{2}}\sum_{r = 0}^{\frac{1-a_{ij}-m-m'}{2}}\begin{bmatrix}
	1-a_{ij}\\ m+2r
	\end{bmatrix}_{q_i}\alpha_{\frac{1-a_{ij}-m-m'}{2}-r,\frac{1-a_{ij}-m-m'}{2}-r+m'_e+m'_p-1}^{(0,i)}\alpha_{r,r+\frac{m+m'}{2}-1-m'_e-m'_p}^{(1,i)},
	\]
	where we have used the property (\ref{prop q-binom symbol}) of the \(q_i\)-binomial symbol. Next, since \(m'\) and \(m\) have the same parity, we find
	\begin{align*}
	m'_e+m'_p & = \left.\begin{cases}
	\frac{m'}{2} \qquad  & \mathrm{if}\ m\ \mathrm{is\ even}, \\
	\frac{m'+1}{2} \qquad  & \mathrm{if}\ m\ \mathrm{is\ odd}
	\end{cases}\right\}  = \frac{m'+m_p}{2}
	\end{align*}
	and so
	\[
	\frac{1-a_{ij}-m-m'}{2}-r+m'_e+m'_p = \frac{1-a_{ij}}{2}-r-\left(\frac{m-m_p}{2}\right) = \frac{1-a_{ij}}{2}-r-m_e = -(a_{ij})_e-r-m_e.
	\]
	Thus we obtain
	\begin{align*}
	\rho_{m',m}^{(i,j,a_{ij})} = (-1)^{a_{ij}+m}(-q_ic_i)^{\frac{1-a_{ij}-m-m'}{2}}\sum_{r = 0}^{\frac{1-a_{ij}-m-m'}{2}}\begin{bmatrix}
	1-a_{ij}\\ m+2r
	\end{bmatrix}_{q_i} \alpha_{\frac{1-a_{ij}-m-m'}{2}-r,-(a_{ij})_e-r-m_e-1}^{(0,i)}\alpha_{r,r+m_e-1}^{(1,i)},
	\end{align*}	
	which precisely equals \(\rho_{m,m'}^{(i,j,a_{ij})}\) according to (\ref{structure constants with 0}) with \(s = 1\). This proves the symmetry.
	
	For \(a_{ij}\) even, the proof goes along the same lines, starting from (\ref{structure constants with 0}) with either \(s = 0\) or \(s = 1\). 
\end{proof}

\subsection{Generalized $q$-Onsager algebras and their classical counterparts}
\label{Subsection Generalized q-Onsager algebras}

A special class of quantum symmetric pair coideal subalgebras is known under the name generalized \(q\)-Onsager algebras. They coincide with the algebras \(B_{\bc,\bs}\) in the split case, i.e.\ for the trivial admissible pair \((X = \emptyset, \tau = \mathrm{id})\), corresponding to Satake diagrams without black nodes and with the trivial diagram involution. In this case we have \(\theta_q(F_iK_i) = -E_i\) by Lemma \ref{lemma Z_i^+} and moreover \(Q^{\Theta} = \{0\}\) since \(w_X = \mathrm{id}\). Hence we may formulate the following definition.

\begin{definition}
	\label{generalized q-Onsager algebra def}
	The generalized \(q\)-Onsager algebra \(\mathscr{O}_q(\g)\) associated to the Kac-Moody algebra \(\g\) is the subalgebra of \(U_q(\g')\) generated by the elements
	\begin{equation}
	B_i = F_i-c_iE_iK_i^{-1}+s_iK_i^{-1},
	\end{equation}
	with \(i\in I\), and where \((\bc,\bs)\) takes values in the set \(\mathcal{C}\times\mathcal{S}\) defined in (\ref{def C})--(\ref{def S}). By Theorem \ref{theorem these are defining relations}, Corollary \ref{cor F_ij(B_i,B_j) Case 1 with epsilon and P} and the fact that in this case \(\mathcal{Z}_i = -1\), it is abstractly defined by the relations
	\begin{equation}
	\label{F_ij(B_i,B_j) for q-Onsager}
	F_{ij}(B_i,B_j) = \sum_{m = 0}^{-1-a_{ij}}\sum_{m' = 0}^{-1-a_{ij}-m}(-1)^{\frac{1-a_{ij}-m-m'}{2}}\rho_{m,m'}^{(i,j,a_{ij})}B_i^mB_jB_i^{m'},
	\end{equation}
	for \(i\neq j\in I\), with \(\rho_{m,m'}^{(i,j,a_{ij})}\) as obtained in (\ref{rho_m,m' final}) or equivalently in (\ref{structure constants with 0}).
\end{definition}

In the special case \(\g = \widehat{\mathfrak{sl}_2}\), this algebra coincides with the \(q\)-Onsager algebra \cite{Baseilhac-2005}, which is typically described as generated by two elements \(B_0\) and \(B_1\) subject to the \(q\)-Dolan-Grady relations
\begin{equation}
\label{q-Onsager relations original}
[B_0,[B_0,[B_0,B_1]_{q}]_{q^{-1}}] = -c_0q(q+q^{-1})^2[B_0,B_1], \qquad [B_1,[B_1,[B_1,B_0]_{q}]_{q^{-1}}] = -c_1q(q+q^{-1})^2[B_1,B_0],
\end{equation}
for certain \(c_0,c_1\in\K(q)\), where \([\cdot,\cdot]_q\) denotes the \(q\)-commutator, defined by
\[
[A,B]_q = qAB-q^{-1}BA.
\]
Its generalization \(\mathscr{O}_q(\g)\) to other Kac-Moody algebras \(\g\) was introduced in \cite{Baseilhac&Belliard-2010}, where defining relations were presented for the cases \(a_{ij}\in\{0,-1,-2,-3,-4\}\) under some additional restrictions on \(a_{ji}\). The relations (\ref{F_ij(B_i,B_j) for q-Onsager}) we have derived in this paper establish for the first time a complete set of defining relations for the generalized \(q\)-Onsager algebras, valid without restrictions on \(a_{ij}\). By Remark \ref{Remark specialization}, we may equivalently write these relations as
\begin{equation}
\label{q-Dolan-Grady relations}
\left(\overrightarrow{\prod_{m = \frac{a_{ij}}{2}}^{-\frac{a_{ij}}{2}}}\mathrm{ad}_{q_i^m}(B_i)\right)(B_j) = \sum_{m = 0}^{-1-a_{ij}}\sum_{m' = 0}^{-1-a_{ij}-m}(-1)^{\frac{1-a_{ij}-m-m'}{2}}\rho_{m,m'}^{(i,j,a_{ij})}B_i^mB_jB_i^{m'},
\end{equation}
which, by the presence of nested \(q\)-commutators, can be considered relations of \(q\)-Dolan-Grady type.

To conclude, we will consider the limit of the generalized \(q\)-Onsager algebra \(\mathscr{O}_q(\g)\) under the specialization \(q\to 1\) described in Remark \ref{Remark specialization}, which is precisely the algebra \(b_{\bs}=b_{\bs}(X,\tau)\) from Definition \ref{b_s definition} in the special case \(X = \emptyset\) and \(\tau = \mathrm{id}\). It follows immediately that in this case \(\mathrm{Ad}(s(X,\tau)) = \mathrm{Ad}(m_X) = \mathrm{id}\), and hence the automorphism \(\theta(X,\tau)\) coincides with the classical Chevalley involution \(\omega\) defined in (\ref{Chevalley involution def}). Moreover, Definition \ref{b_s definition} asserts that we may state the following.

\begin{definition}
	The (classical) generalized Onsager algebra is the Lie subalgebra \(\mathscr{O}(\g)\) of the Kac-Moody algebra \(\g\) generated by the elements
	\[
	b_i = f_i+\omega(f_i) = f_i-e_i,
	\]
	with \(i\in I\). By \cite[Lemma 2.2]{Stokman-2019}, \(\mathscr{O}(\g)\) is the fixed point Lie subalgebra of \(\g\) under the Chevalley involution \(\omega\).
\end{definition}

For ease of notation, we have chosen to retain only the case \(\bs = \boldsymbol{0}\) from Definition \ref{b_s definition}, which, in the terminology of \cite[Section 7]{Letzter-2002}, is known as the standard case.

The algebras \(\mathscr{O}(\g)\) were studied by Stokman in \cite{Stokman-2019}, where a complete set of defining relations of inhomogeneous Serre type or Dolan-Grady type was given. To describe these relations, we will need the following recursively defined coefficients.

\begin{definition}
	\label{c_s^ij coeff def}
	Let \(i,j\) be distinct elements of \(I\) and \(r\in \N\) arbitrary. For any \(s\in\N\) satisfying \(s\leq r\) we define \(c_s^{ij}[r]\) through the recursion relation
	\begin{equation}
	\label{recurrence relation for c_s^ij}
	c_s^{ij}[r] = c_{s-1}^{ij}[r-1] + (r-1)c_s^{ij}[r-2],
	\end{equation}
	for \(r\geq 2\), with the convention that \(c_{-1}^{ij}[r] = 0\) for any \(r\), and with boundary conditions \(c_r^{ij}[r] = 1\) for \(r\geq 0\) and \(c_{r-1}^{ij}[r] = 0\) for \(r\geq 1\).
\end{definition}

The relation (\ref{recurrence relation for c_s^ij}) coincides with \cite[formula (2.4)]{Stokman-2019} upon setting \(r = 1-a_{ij}\), as we will do in the upcoming proposition. 

\begin{theoremKolb}[{\cite[Proposition 2.4, Theorem 2.7]{Stokman-2019}}]
	\label{prop inhomogeneous Serre classical}
	The algebra \(\mathscr{O}(\g)\) is abstractly defined by the inhomogeneous Serre relations
	\begin{equation}
	\label{inhomogeneous Serre}
	\sum_{s = 0}^{1-a_{ij}}(-1)^{s+1}c_s^{ij}[1-a_{ij}]\left(\ad\, b_i\right)^sb_j = 0,
	\end{equation}
	for any distinct \(i,j\in I\).
\end{theoremKolb}

Note that the relations (\ref{inhomogeneous Serre}) differ from those given in \cite{Stokman-2019} by a factor \((-1)^{s+1}\). This is caused by the fact that the generators used in \cite{Stokman-2019} differ from ours by a sign as well, but of course this does not alter the algebra under consideration.

It follows from Theorem \ref{theorem specialization} that the generators \(B_i\) of the generalized \(q\)-Onsager algebra \(\mathscr{O}_q(\g)\) reduce to the generators \(b_i\) of \(\mathscr{O}(\g)\) under the specialization \(q\to 1\), provided the parameters \(\bc\in\mathcal{C}\) are specializable and \(\bs = \boldsymbol{0}\). Consequently, the same holds true for the defining relations of the \(q\)-deformed and classical Onsager algebras. It will hence be possible to derive closed expressions for the recursively defined coefficients \(c^{ij}_s[1-a_{ij}]\) in (\ref{inhomogeneous Serre}) from the previously obtained equation (\ref{q-Dolan-Grady relations}). We begin with a straightforward identity.

\begin{lemma}
	\label{lemma commutator classical identity}
	For any \(A,B\in U(\g)\) and \(r\in\N\), one has
	\[
	(\ad\, A)^r(B) = \sum_{k = 0}^r(-1)^k\binom{r}{k}A^{r-k}BA^k.
	\]	
\end{lemma}
\begin{proof}
	This follows immediately from the equation (\ref{nested q-commutators induction}) in the limit \(q\to 1\).
\end{proof}

This identity allows to expand the nested commutators in the relation (\ref{inhomogeneous Serre}).

\begin{lemma}
	The inhomogeneous Serre relations (\ref{inhomogeneous Serre}) defining the algebra \(\mathscr{O}(\g)\) can be rewritten as
	\begin{equation}
	\label{inhomogeneous Serre with sum}
	(\ad\, b_i)^{1-a_{ij}}b_j = \sum_{m= 0}^{-1-a_{ij}}\sum_{m' = 0}^{-1-a_{ij}-m}(-1)^{a_{ij}+m}\binom{m+m'}{m'}c^{ij}_{m+m'}[1-a_{ij}]b_i^mb_jb_i^{m'}.
	\end{equation}
\end{lemma}
\begin{proof}
	It follows immediately from Lemma \ref{lemma commutator classical identity} and the fact that \(c_{1-a_{ij}}^{ij}[1-a_{ij}] = 1\) and \(c^{ij}_{-a_{ij}}[1-a_{ij}] = 0\) that the relations (\ref{inhomogeneous Serre}) can be rewritten as
	\begin{align*}
	(\ad\, b_i)^{1-a_{ij}}b_j & = (-1)^{a_{ij}}\sum_{s = 0}^{-1-a_{ij}}(-1)^sc^{ij}_s[1-a_{ij}](\ad\, b_i)^sb_j
	= \sum_{s = 0}^{-1-a_{ij}}\sum_{m' = 0}^s(-1)^{a_{ij}+s+m'}\binom{s}{m'}c^{ij}_s[1-a_{ij}]b_i^{s-m'}b_jb_i^{m'}.
	\end{align*}
	The claim now follows upon changing the order of summation, replacing \(s\) by the new summation index \(m = s-m'\) and observing that
	\begin{align*}
	&\left\{(m,m'): m\in\{0,\dots,-1-a_{ij}\}, \ m'\in\{0,\dots,-1-a_{ij}-m\} \right\} \\ =\ & 
	\left\{(m,m'): m\in\{0,\dots,-1-a_{ij}-m'\}, \ m'\in\{0,\dots,-1-a_{ij}\} \right\}.
	\end{align*}
\end{proof}

An alternative set of defining relations for the generalized Onsager algebras \(\mathscr{O}(\g)\) can be found by taking the limit \(q\to 1\) of the \(\mathscr{O}_q(\g)\)-relations (\ref{F_ij(B_i,B_j) for q-Onsager}). Comparison of both types of relations leads to closed expressions for the recursively defined coefficients \(c_s^{ij}[r]\).

\begin{theorem}
	\label{theorem Stokman via limits}
	For any distinct \(i,j\in I\) and any \(r,s\in \N\) with \(s\leq r\) we have
	\begin{equation}
	\label{Classical Onsager coefficients 1}
	c_s^{ij}[r] = (r-s+1)_p\sum_{\substack{\ell_1,\dots,\ell_{\frac{r-s}{2}} = r_p \\ \ell_1 < \dots < \ell_{\frac{r-s}{2}}}}^{r_e+r_p-1}(2\ell_1+1-r_p)^2(2\ell_2+1-r_p)^2\dots(2\ell_{\frac{r-s}{2}}+1-r_p)^2,
	\end{equation}
	or equivalently
	\begin{equation}
	\label{Classical Onsager coefficients 2}
	c_s^{ij}[r] = (r-s+1)_p\sum_{m = 0}^{\frac{r-s}{2}}\left[\binom{r}{2m}\left(\prod_{k = 0}^m(2k-1)^2\right)\sum_{\substack{\ell_1,\dots,\ell_{\frac{r-s}{2}-m}=1-r_p\\\ell_1<\dots<\ell_{\frac{r-s}{2}-m}}}^{r_e-m-1}(2\ell_1+r_p)^2\dots(2\ell_{\frac{r-s}{2}-m}+r_p)^2\right],
	\end{equation}
	where the sum in (\ref{Classical Onsager coefficients 1}) and (\ref{Classical Onsager coefficients 2}) should be read as \(1\) if \(s = r\) respectively if \(m = \frac{r-s}{2}\), and where we have used the notation (\ref{even and odd part}).
\end{theorem}
\begin{proof}
	By the above observations, in the limit \(q\to 1\) the relation (\ref{q-Dolan-Grady relations}) becomes
	\[
	(\ad\, b_i)^{1-a_{ij}}b_j = \sum_{m= 0}^{-1-a_{ij}}\sum_{m' = 0}^{-1-a_{ij}-m}\left[\lim_{q\to 1}\left((-1)^{\frac{1-a_{ij}-m-m'}{2}}\rho_{m,m'}^{(i,j,a_{ij})}\right)\right]b_i^mb_jb_i^{m'}.
	\]
	Upon comparison with (\ref{inhomogeneous Serre with sum}), and following the same reasoning as in the proof of Theorem \ref{theorem Case 1 with CLW}, it follows that
	\begin{align*}
	(-1)^{a_{ij}+m}\binom{m+m'}{m'}c^{ij}_{m+m'}[1-a_{ij}] = \lim_{q\to 1}\left((-1)^{\frac{1-a_{ij}-m-m'}{2}}\rho_{m,m'}^{(i,j,a_{ij})}\right),
	\end{align*}
	for any \(m,m'\in\N\) with \(m+m'\leq -1-a_{ij}\). For \(m = 0\), upon using the expression (\ref{structure constants with 0}) with \(s = 0\), this becomes
	\[
	(-1)^{a_{ij}}c^{ij}_{m'}[1-a_{ij}] = (a_{ij}+m')_p \lim_{q\to 1}\left((-1)^{a_{ij}}(q_ic_i)^{\frac{1-a_{ij}-m'}{2}}\alpha_{\frac{1-a_{ij}-m'}{2}, -(a_{ij})_e-(a_{ij})_p}^{((a_{ij})_p,i)} \right),
	\]
	by (\ref{Theta 0}), where we have used the fact that \(\alpha_{r,r-1}^{(0,i)} = \delta_{r,0}\). The expression (\ref{Classical Onsager coefficients 1}) now follows upon setting \(r = 1-a_{ij}\), renaming \(m'\) to \(s\) and using Definition \ref{definition alpha_k,N} and the assumption of specializability of \(\bc\). Equation (\ref{Classical Onsager coefficients 2}) follows similarly from (\ref{structure constants with 0}) with \(s = 1\).
\end{proof}

\section*{Acknowledgements}

HDC is a PhD Fellow of the Research Foundation Flanders (FWO). This work was also supported by FWO Grant EOS 30889451.

\end{document}